 \documentclass[12pt]{amsart}
\usepackage{graphicx, tikz}
\usepackage{xcolor}
\usepackage{amssymb,amsfonts,amsthm,amsmath,calligra, enumerate, stmaryrd}
\usepackage[utf8]{inputenc}
\usepackage{accents}
\usepackage{slashed}
\usepackage{yfonts}
\usepackage{mathrsfs,pifont}
\theoremstyle{definition}
\usepackage[all]{xy}

\usepackage[bookmarksnumbered=true, bookmarksopen=true]{hyperref}

\def\be{\begin{eqnarray}}
\def\ee{\end{eqnarray}}
\def\ben{\begin{eqnarray*}}
\def\een{\end{eqnarray*}}

\usepackage{geometry}
\allowdisplaybreaks

\newcommand{\bC}{\mathbb{C}}

\newcommand{\bP}{\mathbb{P}}

\newcommand{\bR}{\mathbb{R}}

\newcommand{\bZ}{\mathbb{Z}}

\newcommand{\bp}{\mathbf{p}}

\newcommand{\br}{\mathbf{r}}

\newcommand{\cA}{\mathcal{A}}

\newcommand{\cC}{\mathcal{C}}
\newcommand{\cD}{\mathcal{D}}
\newcommand{\cE}{\mathcal{E}}
\newcommand{\cF}{\mathcal{F}}
\newcommand{\cG}{\mathcal{G}}

\newcommand{\cK}{\mathcal{K}}

\newcommand{\cM}{\mathcal{M}}
\newcommand{\cN}{\mathcal{N}}
\newcommand{\cO}{\mathcal{O}}

\newcommand{\cR}{\mathcal{R}}
\newcommand{\cT}{\mathcal{T}}
\newcommand{\cU}{\mathcal{U}}
\newcommand{\cV}{\mathcal{V}}
\newcommand{\cW}{\mathcal{W}}

\newcommand{\fE}{\mathfrak{E}}

\newcommand{\fg}{\mathfrak{g}}

\newcommand{\fk}{\mathfrak{k}}
\newcommand{\fM}{\mathfrak{M}}

\newcommand{\fX}{\mathfrak{X}}
\newcommand{\fY}{\mathfrak{Y}}

\newcommand{\sA}{\mathsf{A}}
\newcommand{\sB}{\mathsf{B}}
\newcommand{\sG}{\mathsf{G}}
\newcommand{\sK}{\mathsf{K}}
\newcommand{\sN}{\mathsf{N}}
\newcommand{\Pol}{\mathsf{Pol}}
\newcommand{\sT}{\mathsf{T}}
\newcommand{\sW}{\mathsf{W}}

\newcommand{\ab}{\mathrm{ab}}
\newcommand{\BFN}{\mathrm{BFN}}

\newcommand{\loc}{\mathrm{loc}}

\newcommand{\pol}{\mathrm{pol}}
\newcommand{\rat}{\mathrm{rat}}

\newcommand{\taut}{\mathrm{taut}}

\newcommand{\vir}{\mathrm{vir}}

\newcommand{\ann}{\operatorname{ann}}

\newcommand{\chr}{\operatorname{char}}
\newcommand{\cochar}{\operatorname{cochar}}

\newcommand{\Cone}{\operatorname{Cone}}
\newcommand{\diag}{\operatorname{diag}}
\newcommand{\Eff}{\operatorname{Eff}}

\newcommand{\ev}{\operatorname{ev}}
\newcommand{\ex}{\operatorname{ex}}
\newcommand{\ff}{\operatorname{ff}}

\newcommand{\Gr}{\operatorname{Gr}}

\newcommand{\Hom}{\operatorname{Hom}}

\newcommand{\pr}{\operatorname{pr}}

\newcommand{\pt}{\operatorname{pt}}
\newcommand{\QM}{\operatorname{QM}}

\newcommand{\sgn}{\operatorname{sgn}}

\newcommand{\Spec}{\operatorname{Spec}}

\newcommand{\la}{\langle}
\newcommand{\ra}{\rangle}

\theoremstyle{definition}
\newtheorem{Definition}{Definition}[section]
\newtheorem{Assumption}[Definition]{Assumption}

\newtheorem{Remark}[Definition]{Remark}
\newtheorem{Example}[Definition]{Example}
\numberwithin{equation}{section}

\theoremstyle{Theorem}
\newtheorem{Theorem}[Definition]{Theorem}
\newtheorem{Proposition}[Definition]{Proposition}
\newtheorem{Lemma}[Definition]{Lemma}
\newtheorem{Corollary}[Definition]{Corollary}

\begin{document}
	
	\title{Virtual Coulomb branch and vertex functions}
	\author{Zijun Zhou}
	\date{}

\begin{abstract}
\sloppy We introduce a variant of the $K$-theoretic quantized Coulomb branch constructed by Braverman--Finkelberg--Nakajima, by application of a new virtual intersection theory. 
In the abelian case, we define Verma modules for such virtual Coulomb branches, and relate them to the moduli spaces of quasimaps into the corresponding Higgs branches. 
The descendent vertex functions, defined by $K$-theoretic quasimap invariants of the Higgs branch, can be realized as the associated Whittaker functions. 
The quantum $q$-difference modules and Bethe algebras (analogue of quantum $K$-theory rings) can then be described in terms of the virtual Coulomb branch.
As an application, we prove the wall-crossing result for quantum $q$-difference modules under the variation of GIT. 
Nonabelian cases are also treated via abelianization. 
\end{abstract}

	\maketitle

\setlength{\parskip}{1ex}

\tableofcontents

\vspace{1ex}

\section{Introduction}

\subsection{Motivation and background}

Let $\sG$ be a complex reductive group and $\sN$ be a representation. 
The pair $(\sG, T^*\sN)$ defines a 3d $\cN = 4$ supersymmetric gauge theory, which in the last two decades has received a lot of attention in both mathematics and physics. 
Such theories admit two interesting components in their moduli spaces of vacua, known as the \emph{Higgs branch} and \emph{Coulomb branch}. 
The famous phenomenon of 3d mirror symmetry \cite{PhysMir3, PhysMir1, PhysMir2,Ga-Wit, HW, BDGH} predicts the existence of certain mirror pairs whose Higgs and Coulomb branches are exchanged with each other. 

Much progress has been made in understanding the $K$-theoretic enumerative geometry arising from the Higgs branch, which is mathematically the holomorphic symplectic quotient $X = T^*\sN /\!/\!/ \sG$. 
The \emph{vertex function} $V(Q)$ was introduced by A. Okounkov \cite{pcmilec}, defined by $K$-theoretic counting of quasimaps from $\bP^1$ to the Higgs branch. 
Its most important property is to satisfy two sets of $q$-difference equations, in terms of the K\"ahler and equivariant parameters respectively, which are conjectured to be switched under 3d mirror symmetry. 
Together with the uniqueness, this implies that the vertex functions of a 3d mirror pair are equal to each other, up to a nontrivial transition matrix in terms of the elliptic stable envelope \cite{AOelliptic, Oko-ind}. 
3d mirror symmetry for vertex functions for Higgs branches such as hypertoric varieties, $T^*\Gr(k,n)$ and $T^*Fl_n$ has been studied in \cite{SZ, Din-Gr, Din-Fl}. 

Another interesting enumerative invariant is the \emph{Bethe algebra} \cite{NS1, AOBethe, PSZ, KPSZ}, which as an analogue of the quantum $K$-theory ring, can be obtained as the $q\to 1$ limit of the $q$-difference module generated by vertex functions. 
An example of recent work on the 3d mirror symmetry of Bethe algebras for the instanton moduli spaces has been done in \cite{KZ-ins}. 

The Coulomb branch, however, has been much less studied. 
Only recently was its mathematical construction introduced by Braverman--Finkelberg--Nakajima \cite{Nak-I, BFN, BFN-III}. 
Their approach is to consider the affine Grassmannian $\Gr_\sG$, and certain moduli spaces $\cR$ of triples $(P, \varphi, s)$ over it. 
Here $(P, \varphi)$ is a point in $\Gr_\sG$ and $s$ is some section of the associated vector bundle $\sN_\cO$. 
Using the geometry of $\cR$, and essentially the convolution product on $\Gr_\sG$, they construct a product $*$ on the equivariant Borel--Moore homology $\cA := H^{G_\cO}_* (\cR)$, realizing it as a commutative convolution algebra. 
The spectrum $\Spec \cA$ is then defined as the Coulomb branch, which is an (usually singular) affine variety. 
One can also consider the $G_\cO \rtimes \bC^*_q$-equivariant theory, which produces a noncommutative convolution algebra, called the \emph{quantized Coulomb branch}. 
Similar construction works in $K$-theory, and defines the $K$-theoretic Coulomb branch. 

A physics construction for the quantized Coulomb branch is given by Bullimore--Dimofte--Gaiotto \cite{BDG}, where the monopole operators are counterparts of the generators $r_d$ in BFN's approach for abelian theories. 
Coulomb branches in general then embed into their abelianizations, after appropriate localization of parameters.   
It was then observed in \cite{BDGHK} that the quantized Coulomb branch acts on the moduli space of quasimaps to the 3d $\cN = 2$ Higgs branch of $\sN /\!/ \sG$ (i.e. GIT quotient, rather than the hyperk\"ahler quotient), realizing the homology of the latter as a \emph{Verma module}. 
The $J$-function of $\sN /\!/\sG$ can then be expressed in terms of the Verma module. 
There is also some recent work \cite{HKW} in this direction. 

\subsection{Main idea}

As both the Higgs and Coulomb branch are associated with a 3d $\cN = 4$ theory, we could naturally ask how the Coulomb branch interacts with the 3d $\cN = 4$ Higgs branch. 
However, it seems not easy to do so with BFN's original definition, and we will introduce a variant.

To explain the main idea, we start with a simple example. 
Consider $\sG = \bC^*$ and $\sN = \bC^{n+1}$, whose Higgs branch is $T^*\bP^n$. 
The vertex function can be written down as
\begin{equation} \label{V-TP}
V(Q) = \sum_{d=0}^\infty Q^d \prod_{i=1}^{n+1} (-q^{1/2} \hbar^{-1/2} )^d \frac{(1- \hbar a_i s) \cdots (1-q^{d-1} \hbar a_i s)}{(1 - q a_i s) \cdots (1-q^d a_i s)} , 
\end{equation}
where $Q$ and $q$, $\hbar$, $a_i$'s are the K\"ahler and equivariant parameters respectively.
$s$ is the tautological line bundle. 

Considering the geometry of quasimaps from $\bP^1$ to $T^*\bP^n$, one could find that those with nontrivial degrees are exactly those mapping into the zero-section $\bP^n$, simply because they need to have proper images. 
However, the vertex function for $T^*\bP^n$ is different from the $K$-theoretic $I$-function for $\bP^n$, because their \emph{deformation-obstruction} theories are differernt. 
The cotangent direction of $T^*\bP^n$ is responsible for the obstruction part of the obstruction theory, which yields the numerators in (\ref{V-TP}). 

The key part of BFN's construction for the Coulomb branch involves the following diagram
$$
\xymatrix{
\cT \times \cR & \sG_\cK \times \cR \ar[l]_-p \ar[r]^-q & \sG_\cK \times_{\sG_\cO} \cR \ar[r]^-m & \cT,   
}
$$
where $m$ is ind-proper, $q$ is a $\sG_\cO$-fibration (which is smooth), and $p$ is a composition of a diagonal embedding with a $\sG_\cO$-fibration. 
In particular, $p$ is an l.c.i. morphism. 
The convolution product is then defined via the usual ``pullback-pushforward" operation $m_* \circ (q^*)^{-1} \circ p^!$. 

The result of this definition can be described explicitly in the abelian case. 
For the example $T^*\bP^n$, the equivariant $K$-theoretic quantized Coulomb branch is generated over $\bC [a_1^{\pm 1}, \cdots, a_{n+1}^{\pm 1}, q^{\pm 1}, s^{\pm 1}]$ by $r_d$, $d\in \bZ$, up to the relations
$$
r_c \cdot r_d = r_{c+d}, \ \text{ if $c, d$ have the same sign}; \quad r_{-d} \cdot r_d = \prod_{i=1}^{n+1} (1 - q a_i s) \cdots (1-q^d a_i s). 
$$
We see that the denominator in (\ref{V-TP}) can be interpreted as $r_{-d} \cdot r_d$, but the numerators cannot. 
In other words, \emph{only the deformation part is detected}. 

To solve this obstacle, our idea is to introduce \emph{nontrivial perfect obstruction theories} in the definition of the Coulomb branch, as well as the convolution product. 
Concretely, we consider symmetric relative perfect obstruction theories for the projection\footnote{This only works in the abelian case. In general, one works with $\cT \to \Gr_\sG$. See Section \ref{sec-virtual-BFN}.} $\cR \to \Gr_\sG$, where the obstruction part is the dual of the deformation part.  
This then defines virtual pullbacks from $\Gr_\sG$ to $\cR$, which are virtual analogues of the structure sheaves $r_d$. 

To make this definition work for the convolution product, we need to deal with l.c.i. morphisms equipped with obstruction theories of amplitude in $[-2, 0]$, which are not perfect in the usual sense. 
In Section \ref{sec-v-i-t}, we define the \emph{virtual l.c.i. pullback} for such maps, by treating the obstruction part as a virtual normal bundle, which then appears in the denominator. 

The construction results in a definition of the \emph{virtual $K$-theoretic quantized Coulomb branch} in Section \ref{sec-v-c-b}. 
In particular, the ``excess" part of the obstruction has to carry a nontrivial equivariant weight $y = q^{-1} \hbar$ for the definition to work. 
Here $q$ is the quantization parameter coming from the scaling of the formal disk, and $\hbar$ is a new parameter, implicilly identified with the scaling of the cotangent fiber $T^*\sN$. 
The virtual Coulomb branch is then an $\hbar$-deformation of BFN's Coulomb branch, with the coefficient ring appropriately localized.

The relations for the example $T^*\bP^n$ then look like
$$
r_{-d} \cdot r_d = \prod_{i=1}^{n+1} (-q^{1/2} \hbar^{-1/2} )^{-d} \frac{(1 - q a_i s) \cdots (1-q^d a_i s)}{(1-\hbar a_i s) \cdots (1-q^{d-1} \hbar a_i s) }, 
$$
which resembles terms in the vertex function. 
Similar explicit presentations are written down for all abelian cases in Section \ref{sec-ab}. 

\begin{Remark}
The virtual construction and the extra $\hbar$-deformation may seem unnatural from the existing literature on Coulomb branches. 
However, as we explained above, our motivation mainly comes from the enumerative geometry, where virtual classes are ubiquitous. 
There are also examples of virtual classes appearing in convolution algebras, e.g. \cite{SV}. 
This paper is a first attempt to introduce virtual counting to the study of Coulomb branches.
Another import reason is that we hope the techniques developed here can be used to study the enumerative geometry of 3d mirror symmetry \cite{AOelliptic}, for which it is crucial to consider the $\bC^*_\hbar$-equivariant theory. 
\end{Remark}

\begin{Remark}
The same construction could also be done for the homological Coulomb branch, which would yield similar results for cohomological vertex functions and quantum cohomology. 
There are at least two reasons that we choose to work out the $K$-theoretic case. 
First, compared to the homological case, relatively less is known in $K$-theoretic enumerative geometry. 
For example, in our case, the wall-crossing result in Section \ref{sec-wall-cx} is already known in the homological case by Maulik--Okounkov's formula of quantum multiplication by divisors \cite{MO}, which is not available in $K$-theory. 
Second, again for possible applications in 3d mirror symmetry \cite{AOelliptic}, it is important to consider $K$-theoretic vertex functions, where K\"ahler and equivariant parameters are both multiplicative and possible to be exchanged with each other.
\end{Remark}

\subsection{Mixed polarization and Verma modules}

To make the above idea work for more general cases, some extra modifications are needed. 
The reason is that, the construction requires the choice of a Lagrangian subspace $\sN$ in the symplectic representation $T^*\sN$, which we call a \emph{polarization}. 
The nice property of $T^*\bP^n$ which makes the above idea work is that, the vector bundle associated with $\sN$ on $T^*\bP^n$ is always \emph{non-negative} with respect to the quasimaps. 
This is not satisfied by Higgs branches in general. 

For that purpose, our idea is to choose different polarizations according to different degree classes of quasimaps. 
This is done in Section \ref{sec-Verma-vertex}. 
In the abelian case, we analyze in detail how the virtual Coulomb branch behaves under the change of polarizations, and finally define the virtual Coulomb branch $\cA_\sT (\sK, \sN)_X$ with \emph{mixed polarization}. 
Here $\sK$ is the abelian gauge group, $\sT = (\bC^*)^n$ is the standard flavor torus acting on $\sN = \bC^n$, and $X$ is the associated hypertoric variety, i.e. the abelian Higgs branch. 
Compared to the ``full localized" virtual Coulomb branch $\cA_\sT (\sK, \sN)_\loc$, whose Cartan subalgebra is a field, $\cA_\sT (\sK, \sN)_X$ is a subalgebra depending on the Higgs branch $X$. 
It is defined such that we carefully choose which elements in Cartan to invert. 
It turns out that this choice can be nicely described in terms of the $K$-theory presentation of the Higgs branch (see Section \ref{Sec-val}).

For each fixed point $\bp \in X^\sT$, we introduced the notion of Verma module $M(\bp)$, associated with $\bp$. 
Let $\QM(X)^\circ_d$ be the moduli space of quasimaps from $\bP^1$ to $X$ of degree $d$, where $\infty$ maps to the stable locus; we denote by $\QM(X; \bp)_d^\circ \subset \QM(X)_d^\sT$ its subspace of quasimaps into $\bp$, and $K_0^{\sT \times \bC_\hbar^* \times \bC^*_q} (\QM (X; \bp)^\circ_d )_\loc := K_0^{\sT \times \bC_\hbar^* \times \bC^*_q} (\QM (X; \bp)^\circ_d ) \otimes \bC(q)$.  
We have the following result.

\begin{Theorem} [Theorem \ref{thm-verma}]
Let $X$ be a hypertoric variety, $\bp \in X^\sT$ be a fixed point, and $\Eff(\bp)$ be its effective cone. 
There is a natural action of $\cA_\sT (\sK, \sN)_\bp$ on 
$$
\bigoplus_{d\in \Eff(\bp)} K_0^{\sT \times \bC_\hbar^* \times \bC^*_q} (\QM (X; \bp)^\circ_d )_\loc,
$$
realizing it as the Verma module $M(\bp)$. 
\end{Theorem}

\subsection{Vertex functions and quantum $q$-difference modules}

Following the idea of \cite{Bra-ins-1}, we define the Whittaker vector $W_\bp (Q)$ for the virtual Coulomb branch $\cA_\sT (\sK, \sN)_\bp$, which lives in a completion of $M(\bp)$. 
The vertex function can then be written as the \emph{Whittaker function}. 

\begin{Proposition}[Proposition \ref{thm-V-Verma}]
The descendent vertex function restricted to $\bp$ is the descendent Whittaker function associated with $W_\bp (Q)$
$$
V^{(\tau (s) )} (Q) \Big|_\bp = \la W_\bp (Q), \tau (s) W_\bp (Q) \ra. 
$$
\end{Proposition}

The nonabelian vertex functions can then be described via abelianization. 
Let $(\sG, \sN)$ be as in the general case, $\sK \subset \sG$ be a maximal torus, and $X$ be the holomorphic symplectic quotient. 
Under some geometric assumptions on $X$, we have the following result.
Here $\tilde\bp$ is a lift of $\bp$ in the abelianized variety, and $\Delta(s) := \prod_{\alpha\in \Phi} (\hbar s^\alpha)_\infty / (q s^\alpha)_\infty$; $v$ is the highest weight vector of the abelianized Verma module $M(\tilde\bp)$.

\begin{Proposition}[Proposition \ref{Prop-V-Verma-nonab}]
Assume that $X$ admits isolated $(\sA\times \bC^*_\hbar)$-fixed points. 
The descendent vertex function of $X$ restricted to $\bp$ is
\ben
V^{(\tau(s))} (X; Q) \Big|_\bp &=& \frac{1}{\Delta(s) |_\bp} \cdot V^{(\tau (s) \Delta (s))} (X^\ab; \tilde Q) \Big|_{\tilde Q^d \mapsto Q^{\bar d}}, 
\een
where $\bar d$ is the image of $d$ in $\pi_1 (\sG)$.
\end{Proposition}

\sloppy Let $\widetilde V^{(\tau(s))} (X; Q)$ be the renormalized vertex function obtained from $V^{(\tau(s))}(X; Q)$ by the multiplication of a prefactor (see Definition \ref{Defn-vertex}).
The quantum $q$-difference module generated by the descendent vertex functions can also be described in terms of the virtual Coulomb branch. 
As a corollary, the Bethe algebras are also described. 

\begin{Theorem}[Theorem \ref{Thm-D-mod-nonab}]

Assume that $X$ admits isolated $(\sA \times \bC^*_\hbar)$-fixed points. 
\begin{enumerate}[1)]
\setlength{\parskip}{1ex}

\item The $\cD_q$-module generated by all descendent vertex functions $\widetilde V^{(\tau(s))} (X ;Q)$ over $K_{\sA \times \bC^*_\hbar} (\pt)_\loc$ is isomorphic to \footnote{The tensor product $\otimes_\bC$ is as $\bC [ q^{\chi Q \partial_{Q}} , \chi\in \chr(\sK)^\sW ]$-modules, not as algebras.}
$$
\dfrac{
\bC [[Q^{\Eff(X)}]] \otimes_\bC \cA_\sA^0 (\sK, \sN)^\sW_X
}{
\left\la 1\otimes \br_{wc} \tau (s) \br_{-wc} \cdot  \prod_{\alpha\in \Phi} \frac{(q s^\alpha)_{-\la \alpha, w c \ra}}{(\hbar s^\alpha)_{-\la \alpha, wc\ra }} - Q^{\bar c} \otimes \tau(s)   \right\ra \cap \left( \bC [[Q^{\Eff(X)}]] \otimes_\bC \cA_\sA^0 (\sK, \sN)^\sW_X \right)
}, 
$$
where $c$ ranges over $\Eff(X^\ab) \cap \cochar(\sG)_+$ and $w$ ranges over $\sW$. 
$\widetilde V^{(\tau (s))} (X; Q)$ is sent to $1 \otimes \tau(s)$. 

\item The Bethe algebra of $X$ over $K_{\sA \times \bC^*_\hbar} (\pt)_\loc$ is isomorphic to
$$
\dfrac{ 
\cA_\sA^0 (\sK, \sN)^\sW_{X, q=1} [[Q^{\Eff(X)}]] }{ \left\la  \br_{wc} \br_{-wc} \big|_{q=1} \cdot \prod_{\alpha\in \Phi} \frac{(1-s^\alpha)^{-\la \alpha, w c \ra}
}{
(1- \hbar s^\alpha)^{-\la \alpha, wc\ra }}  - Q^{\bar c} \right\ra \cap \cA_\sA^0 (\sK, \sN)^\sW_{X, q=1} [[Q^{\Eff(X)}]] 
} , 
$$
where $c$ ranges over $\Eff(X^\ab) \cap \cochar(\sG)_+$ and $w$ ranges over $\sW$. $\tau(s) |_{q=1} \in \cA^0_\sA (\sK, \sN)_{X, q=1}$ is sent to the quantum tautological class $\widehat\tau (Q)$. 

\end{enumerate}

\end{Theorem}

A description of the non-localized quantum $q$-difference module also follows (Corollary \ref{Cor-refinement}).
This result can be interpreted as a $K$-theoretic generalization of the quantum Hikita conjecture \cite{KMP}.
See Appendix \ref{sec-q-Hikita} for a discussion.

\subsection{Wall-crossing}

Let $X$ and $X'$ be holomorphic symplectic quotients obtained from the same theory $(\sG, \sN)$, with different choices of stability conditions. 
The virtual Coulomb branch (with mixed polarization) admits good properties under the variation of GIT of the Higgs branches. 
As a result, in Section \ref{sec-wall-cx} we prove the wall-crossing property for the quantum $q$-difference modules (and therefore the Bethe algebras).

\begin{Theorem}[Theorem \ref{Thm-wall-cx}]
Based changed over $\cD_q^\rat$, the quantum $q$-difference modules $\cD_q^{\mathrm{pol}} (X)$ and $\cD_q^{\mathrm{pol}} (X')$ generated by descendent vertex functions $\widetilde V^{(\tau(s))} (X; Q)$ and $\widetilde V^{(\tau(s))} (X'; Q)$ for $X$ and $X'$ over $K_{\sA\times \bC^*_\hbar} (\pt)$ are isomorphic. 
In particular, their $q$-difference equations are the same.
The same holds for the Bethe algebras. 
\end{Theorem}

This proves the \cite[Conjecture 1]{Din-cap}, which states that the $q$-difference equations satisfied by vertex functions stay unchanged under the variation of GIT.

\subsection{Acknowledgement}

The author would like to thank Jun Li and Peng Shan for the introduction and suggestion of a project related to Coulomb branches, and discussions at the early stage of this project.
We would also like to thank Tudor Dimofte, Hiraku Nakajima, Andrei Okounkov and Andrey Smirnov for helpful discussions. 
This work is supported by World Premier International Research Center Initiative (WPI), MEXT, Japan.

\vspace{2em}

\section{Virtual intersection theory} \label{sec-v-i-t}

Intersection theory has been studied for a long time in algebraic geometry, and plays a fundamental role in defining invariants in enumerative geometry. 
In the spririt of \cite{Ful}, for an appropriately chosen class of schemes, a well-behaved intersection theory requires the introduction of several crucial ingredients: proper pushforward, flat pullback, exterior product, and most importantly, the refined Gysin pullback along regular embeddings. 

Let us recall the definition of the intersection product in $K$-theory. Let $X$ be a smooth scheme, and $a, b\in K(X)$ \footnote{We denote $K(X) := K^0(X)$ for smooth schemes.}. 
The product $a\otimes b$ is naturally defined for representatives of $a$ and $b$ which are (virtual) vector bundles. 
Alternatively, it can also be defined as
$$
a \otimes b := \Delta^! (a \boxtimes b), 
$$
where $\Delta^!$ is the $K$-theoretic refined Gysin pullback. The diagonal map $\Delta: X \to X \times X$ is a regular embedding by the smoothness of $X$. 

In the modern world of enumerative geometry, due to the singularity of moduli spaces, it becomes inevitable to count the numbers \emph{virtually}. 
By that, we mean the moduli spaces are not considered simply as a scheme or stack, but equipped with an extra structure, called the \emph{perfect obstruction theory} \cite{LT, BF}. 
The virtual fundamental cycle, obtained in this way, enters the definition of the enumerative invariants as the usual fundamental cycle. 

The notion of perfect obstruction theory is later generalized to morphisms $f$ of Deligne--Mumford type, producing the virtual pullback $f^!$ \cite{Man, Qu}, which can be regarded as a virtual analogue of the refined Gysin pullback.
Given the notion of the virtual pullback, it is then natural to introduce a virtual notion of  intersection products. This is done by Kiem--Park \cite{KP}, in a very general context. In this paper, we will apply it to a special case. 

\subsection{Virtual pullbacks}

Let us recall the definition of virtual pullback in \cite{Man, Qu}. Let $f: X \to Y$ be a DM-type morphism between algebraic stacks, and let $L_f \in D_{\rm qcoh}^{\leq 0} (X)$ be its cotangent complex. An obstruction theory is a map $\phi: E^\bullet \to L_f$ in $D_{\rm qcoh} (X)$, such that $h^0 (\phi)$ is an isomorphism and $h^{-1} (\phi)$ is surjective. An obstruction theory is called perfect, if $E^\bullet$ is a perfect complex of amplitude $[-1, 0]$. 

Based on the ideas in \cite{LT, BF}, a perfect obstruction theory is equivalent to a closed embedding $\iota: \cC_f \hookrightarrow \cE_f$, where $\cC_f$ is the intrinsic normal cone stack of the map $f$, and $\cE_f := h^1 / h^0 (E^{\bullet\vee})$ is the vector bundle stack defined by $E^\bullet$. The virtual pullback $f^!: K_0(Y) \to K_0(X)$ is then defined \cite{Qu} as the composition
$$
\xymatrix{
	K_0 (Y) \ar[r]^-{\sigma_f} & K_0 (\cC_f) \ar[r]^-{\iota_*} & K_0 (\cE_f) \ar[r]^-{0^!} & K_0 (X), 
}
$$
where $\sigma_f$ is the deformation-to-normal-cone map, and $0^!$ is the refined Gysin pullback along the zero section of the vector bundle stack $\fE$. 

As in \cite[Construction 3.13]{Man}, let $X$ and $Y$ be DM stacks, equipped with relative perfect obstruction theories $E_{X / \fM}^\bullet$ and $E_{Y/\fM}^\bullet$ over an algebraic stack $\fM$. Note that we do \emph{not} assume that $\fM$ is smooth. Let $f: X \to Y$ be a morphism, and $\varphi: f^* E_{Y/\fM}^\bullet \to E_{X/\fM}^\bullet$ be a map commuting with $f^* L_{Y/\fM} \to L_{X / \fM}$. Then, by completing the triangles, one can construct an obstruction theory $\phi: E^\bullet_f \to L_f$, which fits into the following diagram
\begin{equation} \label{triple}
\xymatrix{
	f^* E_{Y/\fM}^\bullet \ar[r] \ar[d] & E_{X/\fM}^\bullet
	\ar[r] \ar[d] & E_f^\bullet \ar[r] \ar[d] &  f^* E_{Y/\fM}^\bullet [1] \ar[d] \\
	f^* L_{Y/\fM} \ar[r] & L_{X/\fM} \ar[r] & L_f \ar[r] & f^* L_{Y/\fM} [1], 
}
\end{equation}
where both rows are distinguished triangles. 

In the special case where $E^\bullet_f$ happens to be perfect (e.g. if $Y$ is smooth over $\fM$), the virtual pullback $f^!_\vir$ is well-defined. 
Denote by $\pi_X : X \to \fM$ and $\pi_Y : Y \to \fM$ the projection maps. 
The three relative obstruction theories form a compatible triple, satisfying:
$$
f^!_\vir \circ \pi_{Y, \vir}^! = \pi_{X, \vir}^!. 
$$
In particular, if $\fM$ is smooth, one has the virtual structure sheaves $\cO_{Y, \vir} = \pi_{Y, \vir}^! \cO_\fM$, and $\cO_{X, \vir} = \pi_{X, \vir}^! \cO_\fM$. As a result, $f^!_\vir \cO_{Y,\vir} = \cO_{X,\vir}$.

However, the obstruction theory $E^\bullet_f$, in general, has amplitude in $[-2, 0]$, and hence fails to be perfect. This happens in a particular case which we are interested in, where $f$ is a regular embedding. In the following, we would like to explore another possibility where $f^!_\vir$ could be defined and admits good functorial properties.

\subsection{Virtual l.c.i. pullback} 

Let $X$ and $Y$ be quasi-projective DM stacks.
In the rest of this section, we will assume that $f: X \to Y$ is a \emph{locally complete intersection morphism}, i.e. $f$ is a regular embedding followed by a smooth projection.
Equivalently, the cotangent complex $L_f$ is perfect of amplitude $[-1, 0]$, which defines a tautological perfect obstruction theory $E^\bullet_{f, \taut}:= L_f$, whose associated pullback is the usual l.c.i. pullback $f^!$. 

Now, suppose that $f$ is equipped with an obstruction theory $E_f^\bullet$, of amplitude $[-2, 0]$.
Let
$$
\ex (E^\bullet_f) := \Cone (E_f^\bullet \to L_f) [-1],
$$
which fits into the distinguished triangle $\ex (E^\bullet_f) \to E^\bullet_f \to L_f \to \ex (E^\bullet_f)[1]$. 
Since $E_f^\bullet$ is an obstruction theory, $\ex (E^\bullet_f)$ is perfect of amplitude $[-2, -1]$.

For the invertibility of K-theoretic Euler class, we need to work in the equivariant setting. 
Let $\bC^*$ be a 1-dimensional torus, whose equivariant character we denote by $y \in K_{\bC^*} (\pt)$. 
Let $\bC^*$ act trivially on $X$ and $Y$.
We make the following assumption.
\begin{Assumption} \label{Ass-invert}
$\ex (E^\bullet_f)$ is of weight $1$ under the $\bC^*$-action. 
\end{Assumption}
Since $L_f$ has trivial $\bC^*$-weights, this assumption implies that $E_f^\bullet$ actually splits into the direct sum of $\ex (E^\bullet_f)$ and $L_f$. 

Let $K^{\bC^*}_0 (X)_\loc := K_0 ( X ) \otimes_\bC \bC(y)$ be the localized $K$-group.
The assumption implies the invertibility of the $K$-theoretic Euler class
$$
\bigwedge^\bullet [y\otimes  \ex (E^\bullet_f)] := \bigwedge^\bullet [y\otimes (h^{-2} \ex (E^\bullet_f) - h^{-1} \ex (E^\bullet_f)) ]
$$
in $K^{\bC^*}_0 (X)_\loc$. 

\begin{Definition} \label{pullback}

Suppose that Assumption \ref{Ass-invert} holds.
\begin{enumerate}[1)]
\setlength{\parskip}{1ex}

\item Define the \emph{virtual l.c.i. pullback} to be
	$$
	f_\vir^! := \frac{f^!}{\bigwedge^\bullet [y\otimes \ex (E^\bullet_f) ] } : \ K_0^{\bC^*} (Y) \to K_0^{\bC^*} (X)_\loc, 
	$$
	where $f^!: K^{\bC^*}_0 (Y) \to K^{\bC^*}_0 (X)$ is the usual l.c.i. pullback. 

\item Given a Cartesian diagram
\begin{equation} \label{Cart}
\xymatrix{
X' \ar[r]^-{f'} \ar[d]_-{u'} & Y' \ar[d]^-u \\
X \ar[r]^-f & Y, 
}
\end{equation}
The virtual l.c.i. pullback is defined as
$$
f_\vir^! := \frac{f^!}{\bigwedge^\bullet [ y\otimes \ex ( (u')^* E^\bullet_f) ]} : \ K_0^{\bC^*} (Y') \to K_0^{\bC^*} (X')_\loc,
$$
where $f^!: K^{\bC^*}_0 (Y') \to K^{\bC^*}_0 (X')$ is the usual l.c.i. pullback.

\end{enumerate}

\end{Definition}

\begin{Lemma}
\begin{enumerate}[1)]
\setlength{\parskip}{1ex}

\item The virtual l.c.i. pullback commutes with proper pushforwards, flat pullbacks, and the usual l.c.i. pullback. In other words, $f_\vir^!$ is a bivariant class. 

\item In diagram (\ref{Cart}), when $u$ is flat, and hence $f'$ is also l.c.i., the pull-back complex $(u')^* E_f^\bullet$ is an obstruction theory of amplitude $[-2,0]$ for $f'$, and one has $(u')^* \circ f_\vir^! = (f')_\vir^! \circ u^*$. 

\item If $E_f^\bullet = L_f$, the virtual l.c.i. pullback $f_\vir^!$ coincides with the usual l.c.i. pullback $f^!$.

\item Under the limit $y\to 0$, we have $\lim_{y\to 0} f_\vir^! = f^!$. 

\end{enumerate}

\begin{proof}
1) follows from the compatibility of the usual l.c.i. pullback and ($K$-theoretic) Chern classes with proper pushforwards, flat pullbacks and the refined Gysin pullback. 
2) follows from the fact that $L_{f'} = (u')^* L_f$ when $u$ is flat. 
3) and 4) are obvious. 
\end{proof}

\end{Lemma}

\subsection{Compatibility}

Now we show the virtual l.c.i. pullback defined as above admits the expected functorial property. 
Suppose $f: X\to Y$ is as above, and there is a diagram
$$
\xymatrix{
X \ar[rr]^-f \ar[dr]_{\pi_X} && Y \ar[dl]^-{\pi_Y} \\
& \fM, &
}
$$
where $\fM$ is an algebraic stack, $\pi_X$ and $\pi_Y$ are of DM type. 

Let $\bC^*$ be a torus acting trivially on $X$, $Y$ and $\fM$. 
Let $E^\bullet_{X/\fM}$, $E^\bullet_{Y/\fM}$ and $E^\bullet_f$ be $\bC^*$-equivariant relative perfect obstruction theories, associated with the maps $\pi_X$, $\pi_Y$ and $f$, such that $(E^\bullet_{Y/\fM}, E^\bullet_{X/\fM} , E_f^\bullet)$ form a \emph{compatible triple}, i.e. they fit in diagram (\ref{triple}). 
Suppose that $E_f^\bullet$ satisfies Assumption \ref{Ass-invert}.

Consider the truncation distinguished triangle $\tau_{\leq -2} E_f^\bullet \to  E_f^\bullet \to \tau_{\geq -1} E_f^\bullet \to \tau_{\leq -2} E_f^\bullet [1]$. 
Since $L_f$ is of amplitude $[-1, 0]$, the obstruction theory $\phi: E_f^\bullet \to L_f$ naturally induces $\phi_{\geq -1} : \tau_{\geq -1} E_f^\bullet \to L_f$, which is a perfect obstruction theory. 

Consider the complexes
$$
F^\bullet := \mathrm{Cone} (E^\bullet_{X/\fM} \to \tau_{\geq -1} E_f^\bullet ) [-1], \qquad G^\bullet := \Cone ( E^\bullet_{X/\fM} \to L_f) [-1], 
$$
where the morphisms are the compositions $E_{X/\fM}^\bullet \to E_f^\bullet \to \tau_{\geq -1} E^\bullet_f$ and $E_{X/\fM}^\bullet \to E_f^\bullet \to L_f$.
Denote by $\cF := h^1/h^0 (F^{\bullet, \vee})$ and $\cG := h^1/h^0 (G^{\bullet, \vee})$ their associated vector bundle stacks
in the sense of \cite{BF}.

\begin{Lemma} \label{compare}
\begin{enumerate}[1)]
\setlength{\parskip}{1ex}

\item $F^\bullet$ is a perfect complex of amplitude $[-1, 0]$, and there is a short exact sequence of vector bundle stacks
$$
		\xymatrix{
		 \cF \ar[r] & f^* \cE_{Y/\fM} \ar[r] & h^2 (E^{\bullet, \vee} ) =  h^{2} (\ex (E^\bullet_f)^\vee). 
		}
$$
Moreover, the pullback intrinsic normal cone $f^*\cC_{Y/\fM}$ lies in the sub-vector bundle stack $\cF \subset f^* \cE_{Y/\fM}$. 

\item $G^\bullet$ is a perfect complex of amplitude $[-1, 0]$, and there are short exact sequences of vector bundle stacks
$$
\xymatrix{
h^{1} (\ex (E^\bullet_f)^\vee) \ar[r] & \cG \ar[r] & \cF,  & \cN_{X/Y} \ar[r] & \cE_{X/\fM} \ar[r] & \cG, 
}
$$
where $\cN_{X/Y} := h^1 / h^0 (L_f^\vee)$ is the normal bundle stack. 
\end{enumerate}
\end{Lemma}

\begin{proof}
By the pullback of the cone and vector bundle stack, we mean $f^* \cC_{Y/\fM} := X \times_Y \cC_{Y/\fM}$ and $f^*\cE_{Y/\fM} := X \times_Y \cE_{Y/\fM}$. 

1) follows from the following diagram, where all rows and columns are distinguished triangles \footnote{In general, the morphism between the 3rd object of two distinguished triangles obtained from the axiom of triangulated categories is not canonical. But in our case, the morphism $f^* E_{Y/\fM}^\bullet \to F^\bullet$ turns out to be unique by the vertical distinguished triangle.}:
	$$
	\xymatrix{
h^{-2} (E^\bullet_f) [1] \ar[r] \ar[d] & 0 \ar[d] \ar[r] & h^{-2} (E^\bullet_f)[2]  \ar@{=}[r] \ar[d] & h^{-2} (E^\bullet_f) [2]  \ar[d] \\ 
		f^* E_{Y/\fM}^\bullet \ar[r] \ar[d] & E_{X/\fM}^\bullet
		\ar[r] \ar@{=}[d] & E_f^\bullet \ar[r] \ar[d] &  f^* E_{Y/\fM}^\bullet [1] \ar[d] \\
		F^\bullet \ar[r] & E^\bullet_{X/\fM} \ar[r] & \tau_{\geq -1} E^\bullet_f \ar[r] & F^\bullet [1].  
	}
	$$
That $F^\bullet$ has amplitude $[-1,0]$ follows from the 3rd row. 
The exact sequence then follows from \cite[Proposition 2.7]{BF}. 
To see the last statement, consider the map from the 2nd row of the diagram to the distinguished triangle $f^* L_{Y/\fM} \to L_{X/\fM} \to L_f \to f^* L_{Y/\fM}[1]$. 
One can see that the map $h^{-2} (E^\bullet_f) \to h^{-1} (f^* L_{Y/\fM})$ vanishes as it factors through $h^{-2} (L_f) = 0$.
The last statement follows.

2) follows from the following diagram, where all rows and columns are distinguished triangles:
$$
	\xymatrix{
h^{-1} (\ex (E^\bullet_f)) \ar[r] \ar[d] & 0 \ar[d] \ar[r] & h^{-1} (\ex (E^\bullet_f))[1]  \ar@{=}[r] \ar[d] & h^{-1} (\ex (E^\bullet_f))[1]  \ar[d] \\ 
	F^\bullet \ar[r] \ar[d] & E_{X/\fM}^\bullet \ar[r] \ar@{=}[d] & \tau_{\geq -1} E_f^\bullet \ar[r] \ar[d] &  F^\bullet [1] \ar[d] \\
	G^\bullet \ar[r] & E^\bullet_{X/\fM} \ar[r] & L_f \ar[r] & G^\bullet [1].  
	}
	$$
The two exact sequences are the 1st column and 3rd row.
\end{proof}

\begin{Proposition} \label{functoriality}
In $K_0^T (X)_\loc$, we have $f^!_\vir \circ \pi_{Y, \vir}^! = \pi_{X, \vir}^!. $
\end{Proposition}

\begin{proof}
Consider the commutative diagram
$$
\xymatrix{
K_0 (\fM) \ar[r]^-{\sigma_{\pi_Y}} \ar@{=}[d] & K_0 (\cC_{Y/\fM}) \ar[r] \ar[d]^{f^!} & K_0 (\cE_{Y/\fM}) \ar[r] \ar[d]^{f^!} & K_0 (Y) \ar[d]^{f^!} \\
K_0 (\fM) \ar[r]^-\sigma & K_0 (f^* \cC_{Y/\fM}) \ar[r]^-{\iota_*} & K_0 (f^* \cE_{Y/\fM}) \ar[r]^-{0^!} & K_0 (X) ,
}
$$
where $\sigma$ is defined as $f^! \circ \sigma_{\pi_Y}$. 
The first row of the diagram defines $\pi_{Y, \vir}^!$, and hence the second row gives the composition $f^! \circ \pi_{Y, \vir}^!$. 

By Lemma \ref{compare} 1), the pullback intrinsic normal cone $f^* \cC_{Y/\fM}$ actually lies in the sub-vector bundle stack $\cF \subset f^* \cE_{Y/\fM}$.
By excess intersection, we can compare the last two arrows $0^! \circ \iota_*$ with $0^!_\cF \circ \iota_{\cF*}$, with $\iota_{\cF*}: f^* \cC_{Y/\fM} \hookrightarrow \cF$. 
It follows that
\begin{equation} \label{tempt}
f^! \circ \pi_{Y, \vir}^! = 0^! \circ \iota_* \circ \sigma =  \bigwedge^\bullet [y \otimes h^{-2} (E_f^\bullet) ] \cdot 0^!_\cF \circ \iota_{\cF*} \circ \sigma.
\end{equation}

Next, Lemma \ref{compare} 2) implies that $\cG$ is a $h^{1} (\ex (E^\bullet_f)^\vee)$-torsor over $\cF$, which is trivial due to Assumption \ref{Ass-invert}, and induces an embedding $\iota_{\cF, \cG}: \cF \hookrightarrow \cG$.
Consider the diagram (which commutes \emph{except at} the upper-right triangle)
$$
\xymatrix{
& & K_0 (\cF) \ar[d]^-{\iota_{\cF, \cG*}} \ar[dr]^-{0_\cF^!} \\
K_0 (\fM) \ar[r]^-\sigma \ar@{=}[d] & K_0 (f^* \cC_{Y/\fM}) \ar[ur]^-{\iota_{\cF*}} \ar[r]^-{\iota_{\cG*}} \ar[d] & K_0 (\cG) \ar[r]^-{ 0_\cG^!} \ar[d] & K_0 (X) \ar@{=}[d] \\
K_0 (\fM) \ar[r]^-{\sigma_{\pi_X}} & K_0 (\cC_{X/\fM} ) \ar[r] & K_0 (\cE_{X/\fM}) \ar[r]^{0_{\cE_{X/\fM}}^!} & K_0(X),
}
$$
where all vertical arrows except $\iota_{\cF, \cG}$ are flat pullbacks, and the last row defines the virtual pullback $\pi_{X, \vir}^!$. 
The 2nd square commutes by the short exact sequences of cone stacks $\cN_{X/Y} \to \cC_{X/\fM} \to f^* \cC_{Y/\fM}$ (see  \cite[Proposition 3.14]{BF}), and $\cN_{X/Y} \to \cE_{X/\fM} \to \cG$. 

At the upper-right triangle, we have $0_\cF^! = \bigwedge^\bullet [- y \otimes h^{-1} (\ex (E^\bullet_f))  ] \cdot 0_\cG^! \circ \iota_{\cF, \cG*}$, in the localized $K$-theory.
Comparing with (\ref{tempt}), we obtain
$$
f^! \circ \pi_{Y, \vir}^!  =  \bigwedge^\bullet [y \otimes \ex (E^\bullet_f)  ] \cdot \pi_{X, \vir}^!. 
$$
The proposition follows. 
\end{proof}

\subsection{Symmetric obstruction theory} \label{sec-sym-obs}

A special example that will be important for us is the following. 
Let $\fM$ be an algebraic stack, and let $X$ be quasi-projective DM stack, such that there is a \emph{smooth} projection $\pi_X: X \to \fM$. 
Let the torus $\bC^* = \bC^*_y$ act trivially on $\fM$ and $X$, and let $K_0^{\bC^*} (X)_\loc$ be as in the previous section.

We define a $\bC^*$-equivariant relative perfect obstruction theory for $\pi_X : X \to \fM$ as 
\begin{equation} \label{E-double}
	E^\bullet := y \otimes T_{X/\fM} [1] \oplus \Omega_{X/\fM}, 
\end{equation}
	with the natural map to $\Omega_{X/\fM}$. 

The intrinsic normal cone is $X$ itself, and the obstruction bundle is $\Omega_{X/\fM}$. 
Hence the virtual pullback $\pi_{X, \vir}^! : K_0^{\bC^*} (\fM) \to K_0^{\bC^*_y}(X)$ is
\begin{equation} \label{Ovir-sym}
\pi_{X, \vir}^! = \bigwedge^\bullet (y \otimes T_{X/\fM}) \cdot \pi_X^*. 
\end{equation}

The virtual canonical bundle $K_\vir := \det E^\bullet = y^{-\dim X/\fM} K_{X/\fM}^2$ admits a square root $K_\vir^{1/2} = y^{-(\dim X/\fM) /2} K_{X/\fM}$. One can then twist $\pi_{X, \vir}^! $ by $K_\vir^{1/2}$ as a modification, denoted by $\widehat \pi_{X, \vir}^! $.

The same can be generalized to l.c.i. morphisms. 
Let $X$ and $Y$ be two quasi-projective DM stacks, with smooth projections $\pi_X: X\to \fM$ and $\pi_Y: Y\to \fM$, equipped respectively with symmetric obstruction theories $E_{X/\fM}^\bullet$ and $E_{Y/\fM}^\bullet$ defined as above. 
Let $f: X \to Y$ be a map such that $\pi_X = \pi_Y\circ f$. 

By the smoothness of $\pi_X$ and $\pi_Y$, $f$ is automatically l.c.i..  
Consider the relative obstruction theory
\begin{equation} \label{rel-sym-f}
E_f^\bullet := y \otimes L_f^\vee [1] \oplus L_f ,
\end{equation}
which lies in degree $[-2,-1]$, and satisfies Assumption \ref{Ass-invert}. 
The obstruction theories form a $\bC^*_y$-equivariant compatible triple $f^* E_{Y/\fM}^\bullet \to E_{X/\fM}^\bullet \to E_f^\bullet \to f^* E_{Y/\fM}^\bullet [1]$. 

The excess complex $\ex (E_f^\bullet)$ now reduces to $y \otimes L_f^\vee [1]$. 
Since it carries a nontrivial $y$-weight, $\bigwedge^\bullet (y \otimes L_f^\vee)$ is invertible in the localized ring $K(X)_\loc$. 
The virtual l.c.i. pullback $f_\vir^!$ can then be defined as in Definition \ref{pullback}. 
Moreover, the square root $K_{f, \vir}^{1/2} = (\det E_f^\bullet)^{1/2}$ exists. 

\begin{Definition} \label{defn-v-lci-p}
The modified virtual l.c.i. pullback is defined as 
$$
\widehat f_{\vir}^! := K_{f, \vir}^{1/2} \otimes f_\vir^!: K_0^{\bC^*} (Y) \to K_0^{\bC^*} (X)_\loc . 
$$ 
\end{Definition} 

\begin{Remark}
\begin{enumerate}[1)]
\setlength{\parskip}{1ex}

\item The analogue for (\ref{Ovir-sym}) is
\begin{equation} \label{Ovir-sym-hat}
\widehat \pi_{X, \vir}^! = (-1)^{\dim X / \fM} \cdot \widehat a (y \otimes T_{X/\fM}) \cdot \pi_X^*, 
\end{equation}
where $\hat a$ is defined as $\widehat a (\sum_i x_i) := \prod_i (x_i^{1/2} - x_i^{-1/2} )$. 

Moreover, Lemma \ref{functoriality} still holds for the modified virtual pullback:
\begin{equation} \label{functoriality-hat}
\widehat f_\vir^! \circ \widehat \pi_{Y, \vir}^! = \widehat \pi_{X, \vir}^!. 
\end{equation}

\item More generally, given a Cartesian diagram
$$
\xymatrix{
X' \ar[r]^-{f'} \ar[d] & Y' \ar[d] \\
X \ar[r]^-f & Y, 
}
$$
where $X'$, $Y'$ are not necessarily smooth over $\fM$, and $f'$ is not necessarily l.c.i., $f^! : K(Y') \to K(X')$ is defined in the usual way. 
One can also define the similar virtual pullbacks $f_\vir^!$ and $\widehat f_\vir^!$, as in 2) of Definition \ref{pullback}.

\item The $K$-theoretic Euler class $\bigwedge^\bullet (y \otimes T_{X/\fM})$ is invertible.
Therefore, the embeddings
$$
\Psi , \, \widehat\Psi : K_0^{\bC^*} (X) \otimes \bC(y) \xrightarrow{\sim} K_0^{\bC^*} (X)_\loc, 
$$
$$
\Psi (\cF) := \cF \otimes \bigwedge^\bullet (y \otimes T_{X/\fM}), \qquad \widehat\Psi (\cF) :=  \cF \otimes (-1)^{\dim X / \fM} \cdot \widehat a (y \otimes T_{X/\fM})
$$
are isomorphisms.

In particular, when $\fM = \pt$, the structure sheaf $\cO_X$ is mapped to the virtual structure sheaf $\cO_{X, \vir}$ (resp. $\widehat\cO_{X, \vir}$) under $\Psi$ (resp. $\widehat\Psi$). 

\item When $\fM = \pt$, the diagonal map $\Delta : X \to X \times X$ is l.c.i. with the symmetric obstruction theory lying in the degree $[-2, -1]$. 
One can then define the \emph{virtual intersection product}
$$
\otimes_\vir , \, \widehat\otimes_\vir : K_0^{\bC^*} (X)_\loc \otimes K_0^{\bC^*} (X)_\loc \to K_0^{\bC^*} (X)_\loc , 
$$
$$
a \otimes_\vir b := \Delta^!_\vir ( a \boxtimes b) , \qquad a \, \widehat\otimes_\vir \, b := \widehat\Delta^!_\vir ( a \boxtimes b) .
$$
Since $\ex(E_\Delta^\bullet) = h^{-2} (E_\Delta^\bullet) = y \otimes T_X$, the maps $\Psi$ and $\widehat\Psi$ introduced above are isomorphism of rings, intertwining the usual and the virtual intersection products. 
\end{enumerate}

\end{Remark}

\vspace{2em}

\section{$K$-theoretic quantized virtual Coulomb branch} \label{sec-v-c-b}

In this section we would like to introduce a variant of the BFN construction of the quantized Coulomb branch, defined in \cite{BFN}. The idea is to introduce symmetric perfect obstruction theories, and apply our virtual intersection product. 

Throughout this paper, let $\sG$ be a connected reductive group, with $\pi_1(\sG)$ torsion-free.

\subsection{BFN construction}

In this subsection, we recall the BFN construction of the quantized Coulomb branch. 
Let $\cO := \bC [[z]]$, $\cK := \bC ((z))$, and $D:= \Spec \cO$, $D^* := \Spec \cK$ be the formal disk and punctured formal disk respectively. 

The affine Grassmannian $\Gr_\sG$ has the following moduli interpretation
$$
\Gr_\sG = \sG_\cK / \sG_\cO = \{ (P, \varphi) \mid P: \ \sG\text{-bundle on } D, \ \varphi: P|_{D^*} \xrightarrow{\sim} D^* \times \sG \} / \sG_\cO,  
$$
where $\sG_\cO$ acts as automorphisms on $P$. 
$\Gr_\sG$ is an infinite dimensional variety; more precisely, it is an ind-projective ind-scheme.

The connected components of $\Gr_\sG$ are indexed by $\pi_0(\Gr_\sG) \cong \pi_1 (\sG)$. 
A connected component of $\Gr_\sG$ can be further stratified into $\sG_\cO$-orbits, indexed by dominant cocharacters $\lambda$ of $\sG$. 

Let $\sN$ be a finite-dimensional representation of $\sG$, and denote by $\sN_\cO := P\times_\sG \sN$ the vector bundle associated with a given $\sG$-bundle $P$ over $D$. 
The idea of the BFN construction is to consider the following. 
\begin{Definition}[\cite{BFN}]
The moduli of triples is defined as
$$
\cR := \{ (P, \varphi, s) \mid (P, \varphi) \in \Gr_\sG, \ s \in \sN_\cO, \ \varphi (s) \text{ extends over } 0 \in D \} / \sG_\cO, 
$$
which is a closed subvariety in 
$$
\cT := \{ (P, \varphi, s) \mid (P, \varphi) \in \Gr_\sG, \ s \in \sN_\cO \} / \sG_\cO. 
$$
\end{Definition}
$\cT$ is a pro-vector bundle over $\Gr_\sG$, in which $\cR$ can be seen as a subspace. 
The classical $K$-theoretic Coulomb branch is then defined as the spectrum of the convolution algebra $K^{\sG_\cO} (\cR )$. 

Using the presentation $\cT \cong \sG_\cK \times_{\sG_\cO} \sN_\cO$, a point in $\cT$ can be written as $[g, s]$, with $g\in \sG_\cK$ and $s \in \sN_\cO$. 
There is an embedding $\Pi: \cT \hookrightarrow \Gr_\sG \times \sN_\cK$, $[g,s] \mapsto gs$. 
The space $\cR$ is then identified with $\Pi^{-1} (\Gr_\sG \times \sN_\cO)$. 
The group $\sG_\cK$ acts on $\cT$ by left multiplication in the presentation $\cT \cong \sG_\cK \times_{\sG_\cO} \sN_\cO$. 
The subgroup $\sG_\cO$ preserves the subspace $\cR$. 

The affine Grassmannian $\Gr_\sG$, as well as $\cT$ and $\cR$, admits an approximation by finite-type subspaces, and one can define their $K$-groups as the limits of those. 
Recall that 
$$
\Gr_\sG = \bigsqcup_{\lambda \in \cochar_+ (\sG)} \Gr_\sG^\lambda, \qquad \cT_{\leq \lambda} := \pi^{-1} ( \overline{\Gr_\sG^\lambda} ), 
$$
and similarly we define $\cR_{\leq \lambda}$. 
Then $\overline{\Gr_\sG^\lambda}$ is of finite type, and $\Gr_\sG = \bigcup_\lambda \overline{\Gr_\sG^\lambda}$. 
But $\cT_{\leq\lambda}$ and $\cR_{\leq\lambda}$ are still of infinite type. 

Let $m > 0$ be an integer. Consider
$$
\cT^m := \sG_\cK \times_{\sG_\cO} ( \sN_\cO / z^m \sN_\cO ). 
$$
Then $\cT$ is the limit of the inverse system given by the affine fibrations $\cT^m \to \cT^l$, for $m>l$. 
Let $\pi^m: \cT^m \to \Gr_\sG$ be the projection. 
Then 
$$
\cT^m_{\leq \lambda} := (\pi^m)^{-1} (\overline {\Gr_\sG^\lambda} ) 
$$
is of finite type, and $\cT^m = \bigcup_\lambda \cT^m_{\leq \lambda}$. 

Moreover, the action of $\sG_\cO$ on $\sN_\cO / z^m \sN_\cO$ factors through the group $\sG_{\cO/z^m \cO}$, and we have
$$
\cT^m = \sG_{\cK / z^m \cO} \times_{\sG_{\cO / z^m \cO}} (\sN_\cO / z^m \sN_\cO). 
$$
One can define the similar approximation for $\cR$. 
Moreover, when $m \gg 0$ is sufficiently large, $\cR_{\leq \lambda} = \lim_\leftarrow \cR_{\leq \lambda}^m$ is invariant under the translation by $z^m \sN_\cO$, and the projection $\cR^m_{\leq\lambda} \to \cR^l_{\leq \lambda}$ for each $\lambda$, $m>l$ is an affine fibration. 
The action of $\sG_\cO$ on $\cR^m_{\leq \lambda}$ factors through $\sG_{\cO / z^i \cO}$ for sufficiently large $i$ (which depends on $\lambda$). 

\begin{Definition}[\cite{BFN}] 
The $\sG_\cO$-equivariant $K$-group of $\cR_{\leq \lambda}$ and $\cR$ are defined as
$$
K_0^{\sG_\cO} (\cR_{\leq \lambda}) := K_{0}^{\sG_{\cO / z^i \cO}} ( \cR_{\leq \lambda}^m ), \qquad K_0^{\sG_\cO } (\cR) := \lim_\lambda K_0^{G_\cO } (\cR_{\leq \lambda}) , 
$$
where the limit is with respect to the push-forward $K_0^{G_\cO} (\cR_{\leq \lambda}) \to K_0^{G_\cO} (\cR_{\leq \mu})$, induced by the embedding $\cR_{\leq \lambda} \hookrightarrow \cR_{\leq\mu}$ for $\lambda \leq \mu$. 
\end{Definition}

The definition is independent of the choice of $i$ and $m$, and hence is well-defined. 

\begin{Remark}
Let $\bC^*_q$ be the 1-dimensional torus, scaling the formal disk $D$. One can also consider the quantized version, i.e., the $(\sG_\cO \rtimes \bC^*_q)$-equivariant $K$-theory. 
Moreover, if $\sG$ embeds into a larger group $\tilde \sG$, with exact sequence $1\to \sG \to \tilde \sG \to \sT_F\to 1$, then one can replace the group $\sG$ with $\tilde \sG$. 
\end{Remark}

Now let us recall the construction of the convolution product on $K_0^{\sG_\cO \rtimes \bC^*_q} (\cR)$. Recall the following diagram, 
\begin{equation} \label{3-row-diagram}
\xymatrix{
\cR \times \cR \ar[d] & p^{-1} (\cR \times \cR) \ar[l] \ar[d] \ar[r] & q (p^{-1} (\cR \times \cR)) \ar[d] \ar[r] & \cR \ar[d] \\
\cT \times \cR \ar[d] & \sG_\cK \times \cR \ar[l] \ar[r] \ar[d] & \sG_\cK \times_{\sG_\cO} \cR \ar[r] \ar[d] & \cT \ar[d] \\
\cT \times \Gr_\sG \times \sN_\cO  & \sG_\cK \times \Gr_\sG \times \sN_\cO \ar[l]_-{p} \ar[r]^-q &  \sG_\cK \times_{\sG_\cO} (\Gr_\sG \times \sN_\cO) \ar[r]^-m & \Gr_\sG \times \sN_\cK , 
}
\end{equation}
where the maps on the second and third row are given by 
$$
\xymatrix{
( [g_1, g_2 s], [g_2, s] ) \ar@{|->}[d] & ( g_1, [g_2, s] ) \ar@{|->}[l] \ar@{|->}[r] \ar@{|->}[d] & [g_1, [g_2, s]] \ar@{|->}[r] \ar@{|->}[d] & [g_1 g_2,  s] \ar@{|->}[d] \\
([g_1, g_2 s], [g_2], g_2 s) & (g_1, [g_2], g_2 s) \ar@{|->}[l]_-p \ar@{|->}[r]^-q & [g_1, ([g_2], g_2 s) ] \ar@{|->}[r]^-m & ([g_1 g_2], g_1 g_2 s) . 
}
$$
All squares are Cartesian. Given $a, b\in K_0^{\sG_\cO \rtimes \bC^*_q} (\cR)$, the convolution product is defined as
$$
a * b := m_* (q^*)^{-1} p^! ( a \boxtimes b) . 
$$
The map $p$ factorizes as
\begin{equation} \label{factorization}
\xymatrix{
\sG_\cK \times \Gr_\sG \times \sN_\cO \ar[r]^-\Delta & \sG_\cK \times \sN_\cO \times \Gr_\sG \times \sN_\cO \ar[r]^-{p'_\cT } & \cT \times \Gr_\sG \times \sN_\cO,
}
\end{equation}
where $\Delta$ is the diagonal map of $\sN_\cO$, and $p'_\cT$ is a $\sG_\cO$-fibration. In other words, $p$ is l.c.i., and 
$$
p^! = \Delta^! \circ (p'_\cT)^*
$$ 
is the l.c.i. pullback, where $\Delta^!$ is the refined Gysin pullback. 
All maps here are defined first for the finite models, checked to be compatible with connecting morphisms, and hence well-defined for the limits.

\subsection{Virtual BFN construction} \label{sec-virtual-BFN}

We first introduce a perfect obstruction theory on $\cT$. 
As in Section \ref{sec-sym-obs}, let $\bC^*_\hbar $ be the 1-dimensional torus, with $\hbar \in K_{\bC^*_\hbar} (\pt)$. 

As our virtual intersection takes value only in a localized $K$-group, let us first make the relevant notions precise. 
Let $\sK \subset \sG$ be a maximal torus.
Let $n = \dim \sN$, and $\{ \chi_i, 1\leq i\leq n \}$ be weights of $\sN$. 
Let $\Phi$ be the set of roots of $\sG$. 

$K_0^{\sG_\cO \rtimes \bC^*_q \times \bC^*_\hbar} (\cT)$ is a $K_{\sG \times \bC^*_q \times \bC^*_\hbar}(\pt)$-module. 
We denote the characters corresponding to generalized weights $\chi_i$ and $\alpha$'s by $s^{\chi_i}$ and $s^\alpha$, viewed as elements in $K_\sG(\pt)$. 

\begin{Definition} \label{loc}
Consider the localized coefficient ring
$$
K_{\sG\times \bC^*_\hbar \times \bC^*_q} (\pt)_\sN := K_{\sG \times \bC_\hbar^* \times \bC^*_q} (\pt) \Big[\frac{1}{1 - q^\bZ \hbar s^{\chi_i}}, 1\leq i\leq n ; \frac{1}{1 - q^\bZ s^\alpha}, \alpha\in \Phi \Big].
$$
We define 
$$
K_0^{\sG_\cO \rtimes \bC^*_q \times \bC^*_\hbar} (\cT)_\sN 
:= 
K_{\sG \times \bC^*_\hbar \times \bC^*_q} (\pt)_\sN  \otimes_{ K_{\sG \times \bC^*_\hbar \times \bC^*_q} (\pt) } K_0^{\sG_\cO \rtimes \bC^*_q \times \bC^*_\hbar} (\cT)  . 
$$
$K_0^{\sG_\cO \rtimes \bC^*_q \times \bC^*_\hbar} (\cR)_\sN$ can be defined in the same way. 
\end{Definition}

Fix $\lambda \in \cochar(\sG)$, and $m\gg 0$. 
Recall the finite model $\cT^m_{\leq \lambda} :=(\pi^m)^{-1} ( \overline{\Gr_\sG^\lambda} )$, where $\pi^m: \cT^m \to \Gr_\sG$ is the projection. 
A point $(P, \varphi, s)$ in $\cT^m_{\leq \lambda}$ is given by $(P, \varphi) \in \overline{\Gr_\sG^\lambda}$, and $s \in P \times_\sG (\sN_\cO / z^m \sN_\cO)$.  

The deformation theory of $\cT^m_{\leq \lambda}$ over $\overline{\Gr_\sG^\lambda}$ is given by the deformation of the section $s$. 
Therefore, $\cT^m_{\leq \lambda}$ is smooth over $\overline{\Gr_\sG^\lambda}$, and the relative cotangent complex is a vector bundle $\Omega_{\pi^m}$. 
As in Section \ref{sec-sym-obs}, we introduce the symmetric perfect obstruction theory 
$$
E^\bullet_{\pi^m} := \hbar \otimes \Omega_{\pi^m}^\vee [1] \oplus \Omega_{\pi^m} . 
$$
After twisting by $K_\vir^{1/2}$, one can define a virtual pullback $\widehat{(\pi^m)}^!_\vir : K_0^{\sG_\cO \rtimes \bC^*_q \times \bC^*_\hbar} ( \overline{\Gr_\sG^\lambda} ) \to K_0^{\sG_\cO \rtimes \bC^*_q \times \bC^*_\hbar} (\cT^m_{\leq \lambda})_\sN$. 
Note that this is the usual virtual pullback, in the sense of \cite{Qu}. 

Given $m > l \gg 0$, let $p^m_l: \cT^m_{\leq \lambda}\to \cT^l_{\leq \lambda}$ be the affine fibration, which is smooth, and admits the relative perfect obstruction theory $E^\bullet_{p^m_l}:= \hbar \otimes \Omega_{p^m_l}^\vee[1] \oplus \Omega_{p^m_l}$. 
The symmetric obstruction theories $(E^\bullet_{\pi^l}, E^\bullet_{\pi^m}, E^\bullet_{p^m_l})$ then form a compatible triple, and by \cite[Proposition 2.11]{Qu} we have the compatibility of virtual pullbacks $\widehat{(\pi^m)}^!_\vir = \widehat{(p^m_l)}^!_\vir \circ \widehat{(\pi^l)}^!_\vir$. 
Moreover, the virtual pullback commutes with proper pushforwards. 
Therefore, the definition of the virtual pullback makes sense for the limit space:
$$
\widehat\pi_\vir^! : K_0^{\sG_\cO \rtimes \bC^*_q \times \bC^*_\hbar} ( \Gr_\sG ) \to K_0^{\sG_\cO \rtimes \bC^*_q \times \bC^*_\hbar} (\cT)_\sN . 
$$

Now let us define the convolution product. 
Consider the 3rd row of the convolution diagram (\ref{3-row-diagram}):
$$
\xymatrix{
\cT \times \Gr_\sG \times \sN_\cO  & \sG_\cK \times \Gr_\sG \times \sN_\cO \ar[l]_-{p} \ar[r]^-q &  \sG_\cK \times_{\sG_\cO} (\Gr_\sG \times \sN_\cO) \ar[r]^-m & \Gr_\sG \times \sN_\cK. 
}
$$
Recall that $q$ is a $\sG_\cO$-fibration, and $p$ is an l.c.i. morphism, which factorizes as in (\ref{factorization}):
$$
\xymatrix{
\sG_\cK \times \Gr_\sG \times \sN_\cO \ar[r]^-\Delta & \sG_\cK \times \sN_\cO \times \Gr_\sG \times \sN_\cO \ar[r]^-{p'_\cT } & \cT \times \Gr_\sG \times \sN_\cO.
}
$$
Each of the three spaces here admits a smooth projection to $\Gr_\sG \times \Gr_\sG$, and hence has a relative symmetric perfect obstruction theory as in Section \ref{sec-sym-obs}, where the base is taken as $\Gr_\sG \times \Gr_\sG$. 
The maps $\Delta$ and $p'_\cT$ then admit symmetric obstruction theories of the form (\ref{rel-sym-f}). 

Therefore, by constructions in Section \ref{sec-sym-obs}, one can define the virtual pullback $\widehat q^!_\vir$, which is invertible, and the virtual l.c.i. pullback $\widehat p^!_\vir$. 
Note that for $\widehat p^!_\vir$ we need to use Definition \ref{defn-v-lci-p}.

Finally, by the bivariance of the virtual pullback and virtual l.c.i. pullback, the definition extends over the limit, and we have: 

\begin{Definition}

The \emph{$K$-theoretic quantized virtual Coulomb branch} $\cA (\sG, \sN)$ is defined as the localized $K$-group $K_0^{\sG_\cO \rtimes \bC^*_q \times \bC^*_\hbar} (\cR)_\sN$, equipped with a \emph{virtual convolution product} $\widehat{*}_\vir$, where for $a, b \in K_0^{\sG_\cO \rtimes \bC^*_q \times \bC^*_\hbar } (\cR)_\sN$, 
$$
a \, \widehat{*}_\vir \, b \ := \ m_* (\widehat q^!_\vir)^{-1} \widehat p_\vir^! ( a \boxtimes b) . 
$$
Here $m, p, q$ are the maps in diagram (\ref{3-row-diagram}), and the parameter $y$ in Section \ref{sec-sym-obs} is taken as $y = q^{-1} \hbar$. 
\end{Definition}

For simplicity, if there's no risk of confusion, we will write the usual multiplication $\cdot$ to denote the virtual convolution product $\widehat{*}_\vir$, and we will simply refer to the virtual $K$-theoretic quantized Coulomb branch as the \emph{virtual Coulomb branch}. 

We prove some properties of the virtual Coulomb branch. 

\begin{Lemma}
\begin{enumerate}[1)]
\setlength{\parskip}{1ex}

\item For $a, b\in K_0^{\sG_\cO \rtimes \bC^*_q \times \bC^*_\hbar} (\cR)$ in the non-localized $K$-group, the virtual convolution product without twisting by $K_\vir^{1/2}$ (i.e. defined by $p_\vir^!$, $q_\vir^!$ instead of $\widehat p_\vir^!$, $\widehat q_\vir^!$) approaches to BFN's convolution product as $\hbar \to 0$. 

\item Let $[1] \in \Gr_\sG$ be the identity component of the affine Grassmannian.  $\widehat\pi_\vir^! \cO_{[1]} \in K_0^{\sG_\cO \rtimes \bC^*_q \times \bC^*_\hbar} (\cR)_\sN$, which we call the virtual structure sheaf of the component $\pi^{-1} [1]$ in $\cR$ over $[1]$, is the identity element of $\cA(\sG, \sN)$. 

\item $\cA(\sG, \sN)$ is associative. 
The convolution product is $K_{\sG \times \bC^*_\hbar \times \bC^*_q} (\pt)_\sN$-linear in the first variable. 

\item The $q\to 1$ limit of the convolution product is commutative. 

\item $\cM(\sG, \sN) := K_0^{\sG_\cO \rtimes \bC^*_q \times \bC^*_\hbar} (\cT)_\sN$ admits the structure of a right $\cA(\sG, \sN)$-module. 

\end{enumerate}
\end{Lemma}

\begin{proof}
As $\hbar\to 0$, all the virtual pullbacks in the definition of the convolution product approach to the usual l.c.i. pullbacks. Hence 1) follows. 

2) follows from the observation that over the identity component, the obstruction factor $(-1)^{\dim \cR} \cdot \hat a (\hbar \otimes T_{\cR})$ contributed from $\pi^{-1}(1)$ coincides with the one contributed from $h^{-2} (E_\Delta^\bullet)$, with $\Delta$ in (\ref{factorization}). 

3) and 4) follow from the same argument as in \cite[Theorem 3.10]{BFN}. 
They may also be proved via the abelianization.
For 5), it suffices to consider the 2nd row of diagram (\ref{3-row-diagram}), and the same proof for associativity verifies the right module strucure. 
\end{proof}

\vspace{2em}

\section{Virtual Coulomb branch: abelian case} \label{sec-ab}

In this section, we study the abelian case in details. 
As in \cite[Section 4]{BFN}, explicit computations and presentations are available in this situation. 
We first fix some notations.

Let $k\leq n$ be nonnegative integers. 
Throughout this section, let $\sG = \sK \cong (\bC^*)^k$ be a torus of rank $k$, and $\sN \cong \bC^n$ be a $\sK$-representation, with weights $\chi_i \in \bZ^k$, $1\leq i\leq n$. 
In other words, we have the weight decomposition
$$
\sN = \bigoplus_{i=1}^n \bC_{\chi_i},
$$
Let $\sT = (\bC^*)^n$ be the standard $n$-dimensional torus acting on $\sN$. 


In the following, we regard $\sK$ as the \emph{gauge group}, and $\sT$ as the \emph{flavor group}. 
Let 
$$
a_1, \cdots, a_n \in K_\sT (\pt), \qquad s_1, \cdots, s_k \in K_\sK (\pt)
$$ 
be the equivariant parameters associated with the standard characters of $\sT$ and $\sK$. Although defined as equivariant parameters, $s_j$'s  will often be refered to as (exponentiated) ``Chern roots".

We also introduce the following notation:
\begin{equation} \label{x-s-rel}
x_i = a_i \prod_{j=1}^k s_j^{\langle \chi_i, e_j \rangle} = a_i s^{\chi_i}, \qquad 1\leq i\leq n, 
\end{equation}
where $\{ e_j \}$ is the standard basis of $\bZ^k$. 
Geometrically they are characters, or (exponentiated) Chern roots of the 1-dimensional representations $\bC_{\chi_i}$, under the action of $\sK \times \sT$.

Let us describe the localized coefficient ring introduced in Definition \ref{loc}.
Denote by $a^{\pm 1}$ the collection of equivariant parameters $\{a_i^{\pm 1} \mid 1\leq i\leq n\}$, and $s^{\pm 1}$ the collection $\{s_j^{\pm 1} \mid 1\leq j\leq k\}$. 
Then
$$
K_{\sK \times \sT \times \bC^*_\hbar \times \bC^*_q} (\pt)_\sN := \bC[a^{\pm 1}, s^{\pm 1}, q^{\pm 1}, \hbar^{\pm 1}] \Big[\frac{1}{1 - q^\bZ \hbar x_i}, 1\leq i\leq n \Big].
$$

\subsection{Explicit presentation}

The affine Grassmannian $\Gr_\sK$, with the reduced scheme structure, is identified with the cocharacter lattice $\pi_1 (\sK) = \bZ^k$. 
One can write $\Gr_\sK = \{[z^d] \mid d = (d_1, \cdots, d_k) \in \bZ^k\}$, where $z^d = \diag \{z^{d_1}, \cdots, z^{d_k} \}$. 
Then
$$
\cT = \bigsqcup_{d\in \cochar(\sK)} [z^d] \times z^{d} \sN_\cO , 
\qquad 
\cR = \bigsqcup_{d\in \cochar(\sK)} [z^d] \times \left( z^d \sN_\cO \cap \sN_\cO \right) . 
$$
Unlike the general case, when $\sG = \sK$ is abelian, the space $\cR$ is \emph{smooth over $\Gr_\sK$}. 

Let $\cA_\sT(\sK, \sN)$ be the $\sT$-equivariant virtual Coulomb branch. 
More precisely, this means that we consider the equivariant theory $\cA_\sT (\sK, \sN) = K_{\sK_\cO \rtimes \bC^*_q \times \sT \times \bC^*_\hbar} (\cR)_\sN$. 
As in \cite{BFN}, denote by 
$$
r_d \in \cA_\sT (\sK, \sN)
$$ 
the \emph{virtual} structure sheaf of the connected component $\pi^{-1} [z^d] \subset \cR$ over $[z^d]$. 

For any $C, D\in \bZ$, define the notation
$$
\delta (C, D) := \left\{ \begin{aligned}
& 0, \qquad && C, D \text{ have the same sign,} \\
& \min\{ |C|, |D| \}, \qquad && C, D \text{ have opposite signs.}
\end{aligned} \right. 
$$
and the \emph{sign function} 
$$
\epsilon (C) := \left\{ \begin{aligned}
& 1, \qquad && C>0 \\
& 0, \qquad && C=0 \\
& -1, \qquad && C<0. 
\end{aligned} \right. 
$$
For the definition of the Pochhammer symbols, we refer the readers to Appendix \ref{P-symbol}. 

\begin{Proposition} \label{thm-r_d}
\begin{enumerate}[1)]
\setlength{\parskip}{1ex}

\item For any $d\in \cochar(\sK)$, we have 
\begin{equation} \label{ab-rel}
r_{c} \cdot r_{d} = \prod_{i=1}^n \left[ (-q^{1/2} \hbar^{-1/2})^{\epsilon (C_i) \delta_i} \cdot  \frac{(q^{-C_i} \hbar x_i)_{\epsilon (C_i) \delta_i} }{(q^{-C_i} q x_i)_{\epsilon (C_i) \delta_i} }   \right]^{-\epsilon (C_i)}  \cdot r_{c+d}, 
\end{equation}
where $C_i := \langle \chi_i, c \rangle$, $D_i := \langle \chi_i, d\rangle$, and $\delta_i := \delta(C_i, D_i)$.
In particulcar, we have
\begin{equation} \label{ab-rel-special}
 r_{-d} \cdot r_{d} =  \prod_{i: D_i >0} \left[ (- q^{1/2} \hbar^{-1/2} )^{D_i} \frac{(\hbar x_i)_{ D_i } }{(q x_i)_{D_i } }   \right]^{-1}   \cdot \prod_{ i: D_i <0} \left[ (- q^{1/2} \hbar^{-1/2} )^{D_i} \frac{(\hbar x_i)_{ D_i } }{(q x_i)_{D_i } } \right] . 
\end{equation}

\item The $K$-theoretic quantized virtual Coulomb branch $\cA_\sT (\sK, \sN)$ is generated by $\{ r_d \mid d\in \cochar(\sK) \}$ over $K_{\sK \times \sT \times \bC^*_\hbar \times \bC^*_q} (\pt)_\sN$, up to the relations (\ref{ab-rel}). 

\end{enumerate}
\end{Proposition}

\begin{proof}
The proof is essentially the same as \cite[Theorem 3.1]{BFN}. 
Given $c, d\in \cochar(\sK)$, consider the components $\pi^{-1} [z^c], \pi^{-1} [z^{d}] \in \cR$. 
According to diagram (\ref{3-row-diagram}), we need to look at the image of $\pi^{-1}[z^c] \times \pi^{-1} [z^{d}]$ embedded into $\cT \times \Gr_\sK \times \sN_\cO$, and pull it back along $p$. 
By definition of the maps, the image is
$$
[z^c] \times ( z^c \sN_\cO \cap \sN_\cO) \times [z^{d}] \times (z^{c+d} \sN_\cO \cap z^c \sN_\cO) . 
$$
The pullback along $p$, by the factorization (\ref{factorization}), involves the virtual l.c.i. pullback along the diagonal $\Delta: \sN_\cO \to \sN_\cO \times \sN_\cO$. 
Finally, to compute the image along $m_*$, we need to compare the class obtained above, with the pullback of $\pi^{-1} [z^{c+d}]$ along $m$, which is $[z^{c+d}] \times (z^{c+d} \sN_\cO \cap \sN_\cO)$.

In summary, the argument above yields
$$
r_c \cdot r_{d} = \frac{\bigwedge^\bullet (M_\vir^\vee)}{\bigwedge^\bullet (q^{-1} \hbar M^\vir)} \cdot K_\vir^{-1/2} \cdot r_{c+d}, 
$$ 
where $K_\vir := \det (M_\vir^\vee - q^{-1} \hbar M_\vir)$, and 
$$
M_\vir := - z^c \sN_\cO \cap \sN_\cO - z^{c+d} \sN_\cO \cap z^c \sN_\cO + z^c \sN_\cO + z^{c+ d} \sN_\cO \cap \sN_\cO. 
$$

Let us look at the contribution of the $i$-th summand of $\sN$ to $M_\vir$. 
First, it is easy to see that if both $C_i = \langle \chi_i, c\rangle$ and $D_i = \langle \chi_i, d\rangle$ are $\geq 0$ or $\leq 0$, the contribution is $0$. 

Now if $C_i > 0 > D_i$, the $i$-th contribution to $M_\vir$ is $-z^{C_i} x_i \cO + z^{\max\{C_i+D_i, 0\}} x_i \cO = x_i (z^{C_i - \delta_i} + \cdots + z^{C_i - 1} )$. 
The contribution to $r_c \cdot r_d$ is \footnote{The character associated with $z$ is $T_0^*D = q^{-1}$.} 
\ben
&& \frac{(1- x_i^{-1} q^{ C_i - \delta_i }) \cdots (1 - x_i^{-1} q^{C_i-1} ) }{(1-q^{-1} \hbar x_i q^{- C_i + \delta_i }) \cdots (1 - q^{-1} \hbar x_i q^{-C_i +1} ) } \cdot \left( \frac{ x_i^{-\delta_i} q^{\delta_i C_i} q^{- \delta_i (\delta_i +1) / 2} }{ q^{-\delta_i} \hbar^{\delta_i} x_i^{\delta_i} q^{-\delta_i C_i} q^{\delta_i (\delta_i +1) / 2} }  \right)^{-1/2} \\
&=& \frac{(q^{-C_i} q x_i)_{\delta_i} }{(q^{-C_i} \hbar x_i)_{\delta_i} } \cdot (-q^{1/2} \hbar^{-1/2})^{-\delta_i}. 
\een

If $C_i <0 < D_i$, the contribution to $M_\vir$ is $- x_i \cO -  z^{C_i +D_i} x_i \cO + z^{C_i} x_i \cO + z^{\max\{C_i + D_i , 0\}} x_i \cO = - z^{\min\{C_i + D_i, 0\}} x_i \cO + z^{C_i} x_i \cO = x_i ( z^{C_i + \delta_i -1} + \cdots + z^{C_i} )$. 
The contribution to $r_c \cdot r_d$ is 
$$
\frac{(q^{-C_i - \delta_i} q x_i)_{\delta_i} }{( q^{-C_i - \delta_i} \hbar x_i)_{\delta_i} } \cdot (-q^{1/2} \hbar^{-1/2})^{-\delta_i} = \frac{(q^{-C_i} \hbar x_i)_{-\delta_i} }{( q^{-C_i} q x_i)_{-\delta_i} } \cdot (-q^{1/2} \hbar^{-1/2})^{-\delta_i}, 
$$
where the last equality follows from (\ref{id-inv}). 
\end{proof}

Same computations as in Theorem \ref{thm-r_d} would also give an explicit presentation of the right $\cA_\sT(\sK, \sN)$-module $\cM_\sT (\sK, \sN) := K_{\sK_\cO \rtimes \bC^*_q \times \sT \times \bC^*_\hbar} (\cT)_\sN$. 
Denote by $t_d \in \cM_\sT (\sK, \sN)$ the \emph{virtual} structure sheaf of the connected component $\pi^{-1} [z^d] \subset \cT$. 

\begin{Proposition} \label{Prop-right-mod}
The right $\cA_\sT(\sK, \sN)$-module structure of $\cM_\sT (\sK, \sN)$ is given by
$$
t_{c} \cdot r_{d} = \prod_{i:D_i <0} \left[ (-q^{1/2} \hbar^{-1/2})^{D_i} \frac{(q^{- C_i } qx_i)_{- D_i} }{ (q^{- C_i} \hbar x_i )_{- D_i} } \right] \cdot t_{c+d}. 
$$
where $C_i := \langle \chi_i, c\rangle$, $D_i := \langle \chi_i, d \rangle$. 
\end{Proposition}

\begin{proof}
By the same computation as in the theorem, for any $c, d \in \cochar (\sK)$, we have
$$
t_c \cdot r_{d} = \frac{\bigwedge^\bullet (M_\vir^\vee)}{\bigwedge^\bullet (q^{-1} \hbar M^\vir)} \cdot K_\vir^{-1/2} \cdot t_{c+d}, 
$$ 
where $K_\vir := \det (M_\vir^\vee - q^{-1} \hbar M_\vir)$, and 
$$
M_\vir \ := \ - z^c \sN_\cO - z^{c+d} \sN_\cO \cap z^c \sN_\cO + z^c \sN_\cO + z^{c+d} \sN_\cO \  = \ z^{c+d} \sN_\cO - z^{c+d} \sN_\cO \cap z^c \sN_\cO . 
$$
The contribution of the $i$-th summand of $\sN$ to $M_\vir$ is $0$ if $D_i\geq 0$, and $z^{C_i + D_i} x_i \cO - z^{C_i} x_i \cO = x_i z^{C_i} (z^{-1} + \cdots + z^{D_i} )$ if $D_i<0$. 
The proposition follows. 
\end{proof}

\begin{Lemma} \label{shift-r-x}
For any $d\in \cochar(\sK)$, and $1\leq j\leq k$, we have
$$
r_d s_j=  q^{-\langle e_j^*, d\rangle} s_j r_d, \qquad t_d s_j = q^{-\langle e_j^*, d\rangle} s_j t_d, 
$$
and $r_d$, $t_d$ commute with $q$, $\hbar$ and $a_i$'s. 
In particular ,
$$
r_d x_i = q^{-\langle \chi_i, d \rangle} x_i r_d, \qquad t_d x_i = q^{-\langle \chi_i, d \rangle} x_i t_d, 
$$
for any $1\leq i\leq n$. 
\end{Lemma}

\begin{proof}
This is the same as \cite[Lemma 3.20]{BFN}. 
\end{proof}

\begin{Remark}
Recall that here by multiplication of $r_d$ and $s_j$ we mean the virtual convolution product $\widehat{*}_\vir$ in the ring $\cA_\sT(\sK, \sN)$. 
The lemma implies that the virtual convolution product is \emph{not} $K_\sK(\pt)$-linear in the second variable, 
and the induced right action of $K_\sK (\pt)$ on $\cM_\sT (\sK, \sN)$ is \emph{not} $K_\sK(\pt)$-linear. 
\end{Remark}

\subsection{Abelian point and quantum Hamiltonian reduction} 

There is a special example of \emph{abelian point}. 

\begin{Example} \label{ab-pt}
Let $\tilde\sK = (\bC^*)^n$ be the standard torus acting on $\sN = \bC^n$, and $\tilde\chi_i$ be the standard basis vectors of $\chr(\tilde\sK)$. 
The relation (\ref{x-s-rel}) is $\tilde x_i = a_i\tilde s_i$, $1\leq i\leq n$. 
Let $b_1, \cdots, b_n$ be the standard basis for $\cochar (\tilde\sK) = \bZ^n$.
The virtual Coulomb branch $\cA_\sT (\tilde\sK, \sN)$ is generated by $\tilde r_{\pm b_i}$, $1\leq i\leq n$, satisfying the relations
\begin{equation} \label{ab-pt-2}
\tilde r_{b_i} \cdot \tilde r_{\pm b_j} = \tilde r_{\pm b_j}  \cdot \tilde r_{b_i}, \qquad i\neq j; 
\end{equation}
\begin{equation} \label{ab-pt-3}
\tilde r_{-b_i} \cdot \tilde r_{b_i} = (-q^{1/2} \hbar^{-1/2})^{-1} \frac{1 - q \tilde x_i}{1-\hbar \tilde x_i}, 
\qquad 
\tilde r_{b_i} \cdot \tilde r_{-b_i} = (-q^{1/2} \hbar^{-1/2})^{-1} \frac{1 - \tilde x_i}{1 - q^{-1} \hbar \tilde x_i}. 
\end{equation}
For any $d = \sum_{i=1}^n D_i b_i$, one has $\tilde r_d = \prod_{i: D_i>0} r_{b_i}^{D_i} \prod_{i: D_i<0} r_{-b_i}^{-D_i}$. 
\end{Example}

A general abelian Coulomb branch $\cA_\sT (\sK, \sN)$ can be constructed from $\cA_\sT (\tilde\sK, \sN)$ by the \emph{quantum Hamiltonian reduction} process (see \cite[Section 3(vii)(d)]{BFN}), which we now describe.

Let $\sT_F:= \tilde \sK / \sK$ be the ``actual" flavor torus. 
The $\sT_F^\vee$-invariant part of $\cA_\sT (\tilde\sK, \sN)$ is generated by $\{\tilde r_d \mid d\in \cochar (\sK) \}$. 
The quantum comoment map is the composition $K_{\sT_F \times \sT \times \bC^*_\hbar \times \bC^*_q} (\pt)_\sN \to K_{\tilde\sK \times \sT \times \bC^*_\hbar \times \bC^*_q} (\pt)_\sN \to \cA_\sT (\tilde\sK, \sN)$. 
The virtual Coulomb branch $\cA_\sT (\sK, \sN)$ is then the quotient of $\cA_\sT (\tilde\sK, \sN)^{\sT_F}$ by the left ideal generated by the image of the quantum comoment map (whose intersection with $\cA_\sT (\tilde\sK, \sN)^{\sT_F}$ is a two-sided ideal).

In other words, there is an isomorphism
\begin{equation} \label{Ham-red}
  K_{\sK \times \sT \times \bC^*_\hbar \times \bC^*_q} (\pt)_\sN \otimes_{K_{\tilde\sK \times \sT \times \bC^*_\hbar \times \bC^*_q} (\pt)_\sN } \cA_\sT (\tilde\sK, \sN)^{\sT_F}  
 \ \cong \ 
 \cA_\sT (\sK, \sN) , 
\end{equation}
where $\tilde x_i \mapsto x_i$, or equivalently $\tilde s_i \mapsto s^{\chi_i} = \prod_{j=1}^k s_j^{\langle\chi_i, e_j \rangle}$. 

\begin{Remark}
The above construction realizes the virtual Coulomb branch $\cA_\sT (\sK, \sN)$ as an $K$-theoretic virtual analogue of the hypertoric enveloping algebra, introduced in \cite{BLPW-h-cat-O}. 
\end{Remark}

Let $\iota: \sK \to \tilde\sK$ be the map induced by the $\sK$-action on $\sN$. 
Then $\chi_i = \iota^* \tilde\chi_i$. 
For any $d\in \cochar(\sK)$, we can write the image $\iota (d) = \sum_{i=1}^n D_i d_i$, with $D_i = \langle \chi_i, d_i \rangle$. 
The isomorphism (\ref{Ham-red}) then identifies
$$
\tilde r_{\iota(d)} = \sum_{i: D_i>0} \tilde r_{b_i}^{D_i} \prod_{i: D_i < 0} \tilde r_{-b_i}^{-D_i} \quad \mapsto \quad r_d. 
$$
This construction can be used to prove some useful properties of the virtual Coulomb branch. 

\begin{Lemma} \label{anti-isom}

There is an anti-automorphism $\tau$ of $\cA_\sT(\sK, \sN)$ as a $K_{\sK \times \sT \times \bC^*_\hbar \times \bC^*_q} (\pt)_\sN$-algebra, sending $\tau (r_d) = r_{-d}$, $d\in \cochar(\sK)$. 

\end{Lemma}

\begin{proof}
By the construction in Example \ref{ab-pt}, it suffices to define $\tau$ for the case of ``abelian point", where it is easy to explicitly check all relations.
\end{proof}

\subsection{Polarization} \label{sec-pol}

By a \emph{polarization}, we mean a Lagrangian subspace $\sN' \subset T^* \sN$; in other words, there is a splitting $T^* \sN \cong \sN' \oplus \hbar^{-1} (\sN')^*$. 

A typical polarization is $\sN$ itself, which we call the \emph{canonical polarization}. 
For the canonical polarization, the torus $\bC^*_\hbar$ acts trivially on $\sN$, and nontrivially on the cotangent fiber $\sN^*$.
However, for a non-canonical choice of polarization, both $\sN'$ and $\hbar^{-1} (\sN')^*$ could be nontrivial under the action of $\bC^*_\hbar$. 

More precisely, suppose that $\sN':= \sN(\Pol)$ is defined by a subset $\Pol \subset \{1, \cdots, n\}$, such that $\sN \cap \sN(\Pol) = \bigoplus_{i\in \Pol} \bC_{\chi_i}$; in other words,
$$
\sN(\Pol) = \bigoplus_{i\in \Pol} \bC_{\chi_i} \oplus \bigoplus_{i\not\in \Pol} \hbar^{-1} \bC_{-\chi_i}. 
$$ 
We would like to see how the virtual Coulomb branch changes if we choose $\sN(\Pol)$ as the polarization to start with. 

More precisely, the virtual Coulomb branch $\cA (\sG, \sN(\Pol))$ is defined in terms of the localized ($\sG_\cO \rtimes \bC^*_q$)-equivariant $K$-theory of the spaces
$$
\cR (\sG, \sN(\Pol) ) := \{ (P, \varphi, s) \mid (P, \varphi) \in \Gr_\sG, \ s \in \sN (\Pol)_\cO, \ \varphi (s) \text{ extends over } 0 \in D \}, 
$$
$$
\cT (\sG, \sN(\Pol)) := \{ (P, \varphi, s) \mid (P, \varphi) \in \Gr_\sG, \ s \in \sN(\Pol)_\cO \}, 
$$
and the associated virtual convolution product. 
Note that we need $\sN\cap \sN(\Pol)$ to be a $\sG$-subrepresentation of $\sN$, which is always true in the abelian case.
The $\bC^*_\hbar$-action is still defined to scale the cotangent direction of $T^* \sN$, rather than that of $T^* \sN(\Pol)$. 

The localized coefficient ring has to be adjusted according to the polarization:
$$
K_{\sK \times \sT \times \bC^*_\hbar \times \bC^*_q} (\pt)_{\sN(\Pol)} := \bC[a^{\pm 1}, s^{\pm 1}, q^{\pm 1}, \hbar^{\pm 1}] \Big[ \frac{1}{1 - q^\bZ \hbar x_i}, i\in \Pol ; \frac{1}{1 - q^\bZ x_i}, i \not\in \Pol  \Big], 
$$
Same calculations as in Theorem \ref{thm-r_d} and Proposition \ref{Prop-right-mod} show the following analogous presentations. 
Denote the virtual structrue sheaves of $\pi^{-1} [z^d] \subset \cR (\sK, \sN(\Pol))$ (resp. $\pi^{-1} [z^d] \subset \cT (\sK, \sN(\Pol))$) by $r_d (\Pol)$ (resp. $t_d (\Pol)$). 

\begin{Lemma}
For any $c, d\in \cochar(\sK)$, we have 
\ben
r_{c} (\Pol) \cdot r_{d} (\Pol) &=& \prod_{i\in \Pol} \left[ (-q^{1/2} \hbar^{-1/2})^{\epsilon (C_i) \delta_i} \cdot  \frac{(q^{-C_i} \hbar x_i)_{\epsilon (C_i) \delta_i} }{(q^{-C_i} q x_i)_{\epsilon (C_i) \delta_i} }   \right]^{-\epsilon (C_i)}  \\
&& \cdot \prod_{i\not\in \Pol} \left[ (-q^{1/2} \hbar^{-1/2})^{\epsilon (C_i) \delta_i} \cdot  \frac{(q^{-C_i} \hbar x_i)_{\epsilon (C_i) \delta_i} }{(q^{-C_i} q x_i)_{\epsilon (C_i) \delta_i} }   \right]^{\epsilon (C_i)} \cdot r_{c+d} (\Pol) ,           
\nonumber  
\een
where $C_i := \langle \chi_i, c \rangle$, $D_i := \langle \chi_i, d\rangle$, and $\delta_i := \delta(C_i, D_i)$.
In particulcar, we have
\ben
 r_{-d} (\Pol) \cdot r_{d} (\Pol) &=&  \prod_{\substack{i: D_i >0, i\in \Pol \\ \text{or } D_i <0, i\not\in \Pol}} \left[ (- q^{1/2} \hbar^{-1/2} )^{D_i} \frac{(\hbar x_i)_{ D_i } }{(q x_i)_{D_i } }   \right]^{-1}   \cdot \prod_{\substack{ i: D_i <0, i\in \Pol \\ \text{or } D_i >0, i\not\in \Pol }} \left[ (- q^{1/2} \hbar^{-1/2} )^{D_i} \frac{(\hbar x_i)_{ D_i } }{(q x_i)_{D_i } } \right] .
\een
\end{Lemma}

\begin{Lemma}
The right $\cA_\sT(\sK, \sN (\Pol))$-module structure of $\cM_\sT (\sK, \sN (\Pol))$ is given by
$$
t_{c} (\Pol) \cdot r_{d} (\Pol) = \prod_{\substack{i:D_i <0, i\in \Pol \\ or \, D_i >0, i\not\in \Pol} } \left[ (-q^{1/2} \hbar^{-1/2})^{D_i} \frac{(q^{- C_i } qx_i)_{- D_i} }{ (q^{- C_i} \hbar x_i )_{- D_i} } \right] \cdot t_{c+d} (\Pol). 
$$
where $C_i := \langle \chi_i, c\rangle$, $D_i := \langle \chi_i, d \rangle$. 
\end{Lemma}

\subsection{Varying polarizations} \label{sec-vary-pol}

Let $\ff(-)$ denote the \emph{fraction field}. 
Consider
$$
K_{\sK \times \sT \times \bC^*_\hbar \times \bC^*_q} (\pt)_\loc := \ff (\bC[a^{\pm 1}, s^{\pm 1}, \hbar^{\pm 1}, q^{\pm 1}]) ,
$$
in which all $K_{\sK \times \sT \times \bC^*_q\times \bC^*_\hbar} (\pt)_{\sN(\Pol)}$ are viewed as subrings.

We would like to relate virtual Coulomb branches defined for different choices of polarizations. 
A rough expectaction is that virtual Coulomb branches for all choices of polarizations are isomorphic after appropriately enlarging the coefficient rings.

One way to realize it is via localization. 
However, since the spaces involved are infinite-dimensional, we need to pass to certain completed coefficient rings.
Consider
$$
K_{\sK \times \sT \times \bC^*_q\times \bC^*_\hbar} (\pt)^\wedge_{\sN(\Pol)} := \bC[a^{\pm 1}, s^{\pm 1}, \hbar^{\pm 1}]((q)) \Big[ \frac{1}{1 - q^\bZ \hbar x_i}, i\in \Pol;  \frac{1}{1 - q^\bZ x_i}, i \not\in \Pol \Big].
$$
In particular, \emph{Laurent series} in $q$, and hence functions such as $(x_i)_\infty$ and $\vartheta(x_i)$ are included in this ring. 
These are regarded as subrings in $K_{\sK \times \sT \times \bC^*_q\times \bC^*_\hbar} (\pt)^\wedge_\loc := \ff (\bC[a^{\pm 1}, s^{\pm 1}, \hbar^{\pm 1}] ((q)))$.

Denote by 
$$
\cA_\sT(\sK, \sN(\Pol))^\wedge := K_{\sK \times \sT \times \bC^*_\hbar \times \bC^*_q} (\pt)_{\sN(\Pol)}^\wedge \otimes \cA_\sT (\sK, \sN(\Pol) ) , 
$$ 
and similar for the right module $\cM_\sT (\sK, \sN(\Pol))^\wedge$. 

Let $l\gg 0$ be a sufficiently large integer. 
Recall that $\cR^l(\sK, \sN(\Pol))$ is defined with $\sN(\Pol)_\cO$ replaced with the cutoff $\sN_\cO / z^l \sN_\cO$. 
Denote the resulting virtual Coulomb branch by $\cA_\sT (\sK, \sN(\Pol) )^{\leq l}$ (resp. the right module $\cM_\sT (\sK, \sN(\Pol) )^{\leq l}$).
By localization, we see that 
$$
\cA_\sT (\sK, \sN(\Pol) )^{\leq l}
\cong
 \bigoplus_{d\in \cochar(\sK)} K_{\sK \times \sT \times \bC^*_\hbar \times \bC^*_q} (\pt)_{\sN(\Pol)} \cdot [z^d], 
$$
for any given polarization $\Pol$. 
Unfortunately, this localized description is not well-defined \footnote{More precisely, the twist by $K_\vir^{1/2}$ does not have a well-defined limit as $l\to\infty$.} as $l\to \infty$, even if we pass to the completed coefficient ring. 

However, the following proposition shows that the transition isomorphism 
\begin{equation} \label{id-mod-l}
\Phi_l : \cA_\sT (\sK, \sN(\Pol) )_\loc^{\leq l} \cong \cA_\sT (\sK, \sN )_\loc^{\leq l} , 
\qquad
\text{(resp. } \cM_\sT (\sK, \sN(\Pol) )_\loc^{\leq l} \cong \cM_\sT (\sK, \sN )_\loc^{\leq l} \text{ )}
\end{equation}
is well-defined as $l\to \infty$, where $\cA_\sT (\sK, \sN(\Pol) )_\loc^{\leq l}$ means the base change over $K_{\sK \times \sT \times \bC^*_\hbar \times \bC^*_q} (\pt)_\loc$. 

\begin{Proposition} 

\begin{enumerate}[1)]
\setlength{\parskip}{1ex}

\item The limit of (\ref{id-mod-l}) as $l\to \infty$ defines an isomorphism of left $K_{\sK \times \sT \times \bC^*_\hbar \times \bC^*_q}(\pt)^\wedge_\loc$-modules
\begin{equation} \label{Phi}
\Phi : \cA_\sT (\sK, \sN(\Pol) )^\wedge_\loc \xrightarrow{\sim} \cA_\sT (\sK, \sN )^\wedge_\loc,
\end{equation}
such that (where $D_i := \langle \chi_i, d\rangle$)
\ben
\Phi(r_d (\Pol) ) &=& \prod_{i: D_i > 0, i\not\in \Pol} \left[ -\hbar^{-1/2} x_i^{-1}  \frac{\vartheta(x_i) }{\vartheta (\hbar x_i) } \cdot \frac{(q \hbar^{-1} x_i^{-1} )_{ D_i } }{(x_i^{-1})_{ D_i }} \right] \\
&& \cdot \prod_{i: D_i < 0, i\not\in \Pol} \left[ -\hbar^{-1/2} x_i^{-1}  \frac{\vartheta(x_i) }{\vartheta (\hbar x_i) } \cdot \frac{(\hbar x_i )_{- D_i } }{(q x_i)_{- D_i }} \right]  \cdot  r_d.
\een

\item The limit of (\ref{id-mod-l}) as $l\to \infty$ defines an isomorphism of left $K_{\sK \times \sT \times \bC^*_\hbar \times \bC^*_q}(\pt)^\wedge_\loc$-modules
\begin{equation} \label{Phi-M}
\Phi : \cM_\sT (\sK, \sN(\Pol) )^\wedge_\loc \xrightarrow{\sim} \cM_\sT (\sK, \sN )^\wedge_\loc ,
\end{equation}
such that
$$
\Phi ( t_d (\Pol)) = \prod_{i\not\in \Pol} \left[ -\hbar^{-1/2} x_i^{-1}  \frac{\vartheta(x_i) }{\vartheta (\hbar x_i) } \cdot (q \hbar^{-1} )^{D_i} \right] \cdot  t_d. 
$$

\end{enumerate}

\end{Proposition}

\begin{proof}
Let us compute the image $\Phi (r_d (\Pol))$, before the renormalization by theta functions. 
In order to compute the difference of virtual classes, look at the difference of the two Lagrangians corresponding to the polarizations
$$
\delta \sN := \sN(\Pol)-\sN = \sum_{i\not\in \Pol} (\hbar^{-1} \bC_{-\chi_i} -\bC_{\chi_i} ). 
$$
Recall that $\pi^{-1}[z^d] = z^d \sN_\cO \cap \sN_\cO$. 
We have 
$$
r_d (\Pol) =  \frac{\bigwedge^\bullet (q^{-1} \hbar \cdot \delta M^\vir)}{\bigwedge^\bullet (\delta M_\vir^\vee)} \cdot \delta K_\vir^{1/2} \cdot r_d, 
$$
where $\delta M^\vir = z^d \sN(\Pol)_\cO \cap \sN(\Pol)_\cO - z^d \sN_\cO \cap \sN_\cO$, and $\delta K = \det (M_\vir^\vee - q^{-1} \hbar M_\vir)$. 
Then
$$
\delta M^\vir = \sum_{i: D_i >0, i\not\in \Pol} (\hbar^{-1} x_i^{-1} \cO - x_i z^{D_i} \cO) + \sum_{i: D_i <0, i\not\in \Pol} (\hbar^{-1} x_i^{-1} z^{-D_i} \cO - x_i \cO). 
$$ 
Let $l \gg 0$ be a sufficiently large integer. 
Consider $\delta M^\vir$ mod $z^l$. 
The contribution to $\Phi (r_d (\Pol))$ from a term with $D_i>0$, $i\not\in \Pol$ turns out to be
\ben
&& \frac{(q^{-1} x_i^{-1}; q^{-1})_l }{(\hbar x_i)_l}  \left( \frac{q^{-l} x_i^{-l} q^{-l (l-1)/2} }{\hbar^l x_i^l q^{l(l-1)/2} } \right)^{-1/2}  \frac{(q^{D_i} x_i^{-1})_l}{(q^{-1} \hbar q^{-D_i} x_i; q^{-1} )_l}  \left( \frac{q^{l D_i } x_i^{-l} q^{l (l-1)/2} }{q^{-l} \hbar^l q^{-l D_i} x_i^l q^{- l (l-1)/2} } \right)^{-1/2} \\
&=& \frac{(qx_i )_l}{(\hbar x_i )_l} \cdot (-q^{1/2} \hbar^{-1/2})^{-l} \cdot \frac{(q^{D_i} x^{-1})_l }{(q \hbar^{-1} q^{D_i} x_i^{-1} )_l } \cdot (-q^{1/2} \hbar^{-1/2} )^l , 
\een
whose limit as $l\to \infty$ is
$$
\frac{(qx_i )_\infty}{(\hbar x_i)_\infty} \cdot \frac{(x_i^{-1})_\infty }{(q \hbar^{-1} x_i^{-1} )_\infty } \cdot \frac{(q \hbar^{-1} x_i^{-1} )_{D_i} }{(x^{-1})_{D_i} } = -\hbar^{-1/2}x_i^{-1} \frac{\vartheta (x_i )}{\vartheta (\hbar x_i)} \cdot \frac{(q \hbar^{-1} x_i^{-1} )_{D_i} }{(x_i^{-1})_{D_i} } . 
$$
The cases $D_i<0$ and for $t_d$'s are similar. 
\end{proof}

In the rest of the paper, we will use isomorphisms $\Phi$ to identify virtual Coulomb branches with different polarizations. 
It is natural to ask whether they also preserve the convolution product. 
The answer is simply no, as the identity $r_0 (\Pol)$ is not preserved.
However, they do preserve the convolution product if we remove the factors involving theta functions.

Consider the following automorphism $\Xi$ of $\cA_\sT (\sK, \sN)^\wedge_\loc$ and $\cM_\sT (\sK, \sN)^\wedge_\loc$ as left $K_{\sK \times \sT \times \bC^*_\hbar \times \bC^*_q}(\pt)^\wedge_\loc$-modules, defined by rescaling 
\begin{equation} \label{Xi}
\Xi (r_d) := \prod_{i\not\in \Pol} \left[ -\hbar^{1/2} x_i  \frac{\vartheta(\hbar x_i) }{\vartheta (x_i) } \cdot (-q^{1/2} \hbar^{-1/2} )^{-D_i} \right] \cdot r_d, 
\end{equation}
\begin{equation} \label{Xi-M}
\Xi (t_d) := \prod_{i\not\in \Pol} \left[ -\hbar^{1/2} x_i  \frac{\vartheta(\hbar x_i) }{\vartheta (x_i) } \cdot (q \hbar^{-1} )^{-D_i}  \right] \cdot t_d.
\end{equation}
Note that both $\Phi$ and $\Xi$ depend on the polarization $\Pol$, which we omit to avoid cumbersome notations.

\begin{Proposition}
The composite isomorphisms of left $K_{\sK \times \sT \times \bC^*_\hbar \times \bC^*_q}(\pt)^\wedge_\loc$-modules 
$$
\Xi \circ \Phi : \ \cA_\sT (\sK, \sN(\Pol) )^\wedge_\loc \xrightarrow{\sim} \cA_\sT (\sK, \sN )^\wedge_\loc , \quad \text{(resp. } \cM_\sT (\sK, \sN(\Pol) )^\wedge_\loc \xrightarrow{\sim} \cM_\sT (\sK, \sN )^\wedge_\loc \text{ )}
$$
preserves the convolution product (resp. the right $\cA_\sT (\sK, \sN )^\wedge_\loc$-module structure), such that 
\begin{equation} \label{Xi-Phi}
\Xi \Phi (r_d (\Pol)) = \prod_{i\not\in \Pol} \left[ (-q^{1/2} \hbar^{-1/2} )^{-D_i} \frac{(\hbar x_i )_{- D_i } }{(q x_i)_{- D_i } } \right]^{ - \epsilon (D_i)}  \cdot  r_d,  
\end{equation}
(resp. $\Xi \Phi (t_d (\Pol)) = t_d$). 
In particular, $\Xi \circ \Phi$ is well-defined as an isomorphism of algebras
$$
\Xi \circ \Phi : \ \cA_\sT (\sK, \sN(\Pol) )_\loc \xrightarrow{\sim} \cA_\sT (\sK, \sN )_\loc, \quad \text{(resp. } \cM_\sT (\sK, \sN(\Pol) )_\loc \xrightarrow{\sim} \cM_\sT (\sK, \sN )_\loc \text{ )}
$$
over $K_{\sT \times \sK \times \bC^*_\hbar \times \bC^*_q}(\pt)_\loc$, without the completion. 
\end{Proposition}

\begin{proof}
The formula for the image $\Xi \Phi (r_d (\Pol))$ follows from (\ref{id-inv}). 
To check that $\Xi \circ \Phi$ preserves the convolution product, it suffices to prove it for the special abelian point case $\cA_\sT (\tilde\sK, \sN)_\loc $ in Example \ref{ab-pt}. 

Let us check the defining relations for $\Xi \Phi (\tilde r_{\pm b_i} (\Pol))$'s. 
Relation (\ref{ab-pt-2}) follows from the commutativity of $\tilde r_{\pm b_j}$ with $\tilde x_i$ for $i\neq j$. 
To check (\ref{ab-pt-3}), for $i\not\in \Pol$, we see that $\Xi \Phi (\tilde r_{-b_i}(\Pol) ) \cdot \Xi \Phi (\tilde r_{b_i}(\Pol) )$ is
\ben
 && (-q^{1/2} \hbar^{-1/2} ) \frac{1 - \hbar \tilde x_i}{1 - q \tilde x_i } \cdot \tilde r_{-b_i} 
 \cdot \left[ (-q^{1/2} \hbar^{-1/2} )^{-1} \frac{1 - \tilde x_i}{1 - q^{-1} \hbar \tilde x_i} \right]^{-1} \cdot \tilde r_{b_i} \\
&=& \left[ (-q^{1/2} \hbar^{-1/2} ) \frac{1 - \hbar \tilde x_i}{1 - q \tilde x_i } \right]^2 \cdot \tilde r_{-b_i} \cdot \tilde r_{b_i} \\
&=&  (-q^{1/2} \hbar^{-1/2} ) \frac{1 - \hbar \tilde x_i}{1 - q \tilde x_i }, 
\een
which coincides with $\Xi \Phi (\tilde r_{-b_i}(\Pol) \cdot \tilde r_{b_i}(\Pol))$. 
The computation for $\Xi \Phi (\tilde r_{b_i}(\Pol) ) \cdot \Xi \Phi (\tilde r_{-b_i}(\Pol) )$ is similar.   
\end{proof}

\vspace{2em}

\section{Higgs branch and quasimaps}

\subsection{Symplectic quotients} \label{sec-symp-quot}

We return to the general case. 
Recall that $\sG$ is a complex connected reductive group, and $\sN$ is a representation of $\sG$. 
Let $T^*\sN = \sN \oplus \sN^*$ be the cotangent bundle, where the natural action of $\sG$ preserves the canonical symplectic structure. 
There is a complex moment map $\mu: T^* \sN \to \fg^*$. 

In order to define quasimaps, we need to consider holomorphic symplectic quotients, or Higgs branches of $(\sG, \sN)$. 
For that purpose, a \emph{stability condition} $\theta \in \chr (\sG)$ has to be chosen. 
The holomorphic symplectic quotient is the GIT quotient $X_\theta := \mu^{-1} (0) /\!/_\theta \sG$, which is an open subscheme of the quotient stack $\fX: = [\mu^{-1}(0) / \sG]$. 

When $\theta$ is chosen generically, $X_\theta$ is a smooth holomorphic symplectic variety.
Unless otherwise specified, we will omit the subscript and denote it simply by $X$.
Typical examples of this type include hypertoric varieties (where $\sG$ is abelian) and Nakajima quiver varieties (where $\sG$ is a product of $GL(n)$). 
When $\theta = 0$, $X_0$ is an affine variety with symplectic singularities. 
The canonical projection $X \to X_0$ is a symplectic resolution.

The stacky quotient $\fX = [\mu^{-1} (0) / \sG]$ is the zero locus of the moment map, which fits into the diagram
$$
\xymatrix{
& & [ T^* \sN \times \fg^* / \sG] \ar[d] \\
[\sN / \sG] \ar@{^{(}->}[r] & \fX \ar@{^{(}->}[r]^-{i_\mu} & [T^* \sN / \sG] \ar@/_15pt/[u]_\mu . 
}
$$
Here $\cV_{\fg^*} := [ T^* \sN \times \fg^* / \sG]$ is the vector bundle over $[T^* \sN / \sG]$ associated with $\fg^*$, and $\mu$ is the section defined by the moment map. 
By standard constructions, the intrinsic normal cone $\cC_\fX$ embeds into the pullback vector bundle stack $\iota^* \cV_{\fg^*}$, which defines a perfect obstruction theory $E_\fX^\bullet$ on $\fX$, in the sense of \cite{Poma}.
Alternatively, $E_\fX^\bullet$ fits into the distinguished triangle 
\begin{equation} \label{E-fX}
\xymatrix{
\cV_{\fg^*}^\vee \ar[r]^-{(d\mu)^\vee} & L_{[T^* \sN / \sG]} \ar[r] & E_\fX^\bullet \ar[r]^{[1]} & 
}
\end{equation}
The cotangent bundle $\Omega_X$ of the stable locus $X\subset \fX$ is the restriction of $E_\fX^\bullet$ to $X$. 

On the other hand, if we choose $T^* \sN$ as the smooth covering of $[T^* \sN / \sG]$, then $L_{[T^*\sN / \sG]} = [\Omega_{T^* \sN} \to \fg^* \otimes \cO ]$ lies in degree $[0,1]$, and the map is exactly $d\mu$. 
Therefore, $E_\fX^\bullet$ (viewed over the smooth covering $\mu^{-1}(0)$) is the pullback of the complex $[\fg \otimes \cO \xrightarrow{(d\mu)^\vee} \Omega_{T^* \sN} \xrightarrow{d\mu} \fg^* \otimes \cO]$, which lies in degree $[-1,1]$. 
It is clear that $E_\fX^\bullet$ is symplectic.

\subsection{Quasimaps}

\begin{Definition} [\cite{CF-K-M, pcmilec}]
A \emph{quasimap} from $\bP^1$ to $X$ is a morphism $f: \bP^1 \to \fX = [\mu^{-1}(0) / \sG]$, which maps into the stable locus $X \subset \fX$ away from a $0$-dimensional subscheme of $\bP^1$. 
\end{Definition}

Alternatively, a quasimap consists of the data $(P, s)$, where $P$ is a principal $\sG$-bundle over $\bP^1$, and $s$ is a section of the associated bundle $P \times_\sG T^* \sN$, which is required to be generically stable. 

The image of $f$ under $\fX \to B \sG$ defines a class $\deg f \in \pi_0 (B \sG) \cong \pi_1 (\sG)$, which we call the \emph{degree} of $f$. 
Let $\Eff(X)$ be \emph{effective cone} of $X$, i.e. the semigroup in $\pi_1(\sG)$ which can be realized by quasimaps. 

Let $\QM (X)$ be the moduli space of quasimaps from $\bP^1$ to $X$, and $\QM(X)_d$ be its connected component of quasimaps of degree $d \in \Eff (X)$. 
As we do not allow the domain $\bP^1$ to vary, the universal curve over the moduli space is a trivial family.
In fact, $\QM(X)$ is an open substack of the Hom stack $\Hom (\bP^1, \fX)$, which admits the following universal diagram
$$
\xymatrix{
\QM (X) \times \bP^1 \ar@{^{(}->}[r] \ar[d]_-\pi & \Hom(\bP^1, \fX) \times \bP^1 \ar[r]^-f \ar[d]_-\pi & \fX \\
\QM (X) \ar@{^{(}->}[r] & \Hom (\bP^1, \fX). 
}
$$

Let $E_\fX^\bullet$ be the perfect obstruction theory on $\fX$ given as in the previous section.
There is a perfect obstruction theory on the Hom stack $\Hom (\bP^1, \fX)$:
$$
E^\bullet := R\pi_* (f^* E_\fX^\bullet \otimes \omega_{\bP^1}) [1],
$$
which restricts to a perfect obstruction theory \cite{Kim-holo-symp, pcmilec} on the open substack $\QM(X)$. 
The perfect-ness is guaranteed by the generic stability condition on the definition of quasimaps. 

We give a heuristic discussion on how the obstruction theory $E^\bullet$ can be interpreted as coming from a ``Lagrangian subspace" $\Hom (\bP^1, [\sN/\sG])$. 
Consider the embedding $i_\sN : [\sN / \sG] \hookrightarrow \fX$. 
The pullback of the complex $E_\fX^\bullet$ to $[\sN/\sG]$ splits as a direct sum $[\fg \otimes \cO \xrightarrow{(d\mu)^\vee} \Omega_{\sN}] \oplus [\Omega_{\sN^*} \xrightarrow{d\mu} \fg^* \otimes \cO]$, where the two summands live in degrees $[-1,0]$ and $[0,1]$ respectively. 
Serre duality on $\bP^1$ implies that the obstruction theory $E^\bullet$, when pulled back to $\Hom (\bP^1, [\sN/\sG])$, is a ``$(-1)$-shifted double" of the natural obstruction theory on $\Hom (\bP^1, [\sN/\sG])$.

In other words, restricted to the ``Langrangian subspace" $\Hom(\bP^1, [\sN/\sG])$, $E^\bullet$ can be realized as follows: first construct the natural obstruction theory on $\Hom(\bP^1, [\sN/\sG])$ itself, and then consider its associated symmetric obstruction theory as in (\ref{E-double}).

\subsection{Vertex functions}

Let $\QM^\circ (X) \subset \QM(X)$ be the open substack, consisting of quasimaps $f$ such that $f(\infty) \in X$.  
Then an evaluation map $\ev_\infty: \QM^\circ (X) \to X$ is well-defined, sending $f$ to $f(\infty)$. 
Since it is not proper, one has to work over certain localized coefficient ring to define its pushforward. 

Let $\bC^*_q$ be the 1-dimensional torus scaling $\bP^1$. 
We assume that the $\bC^*_q$-fixed loci of $\QM^\circ (X)$ are proper, which is true for all hypertoric varieties and Nakajima quiver varieties.
Then the pushforward $(\ev_{\infty})_* : K_0^{\bC^*_q} (\QM^\circ (X))_\loc \to K_{\bC^*_q} (X)_\loc$ is well-defined by $\bC^*_q$-localization. 
Here the localized ring means to $\otimes \bC(q)$. 

In practice, $X$ often admits an action by a torus $\sA$, which descends from $\sN$. 
One can then consider the full equivariant theory over $K_{\sA \times \bC^*_\hbar \times \bC^*_q} (\pt)_\loc$, where the localized ring is in the sense of Definition \ref{loc}.

The perfect obstruction theory constructed in the previous section induces a virtual structure sheaf $\cO_\vir$. 
Twisted by the square root $K_\vir^{1/2}$ of the virtual canonical bundle, we get the modified virtual structure sheaf $\widehat\cO_\vir = \cO_\vir \otimes K_\vir^{1/2}$. 

Let $\tau \in K_{\sA \times \bC^*_\hbar} (\fX):= K_{\sG \times \sA \times \bC^*_\hbar} (\pt)$, which defines a tautological virtual vector bundle $\tau |_X$ in $K_{\sA \times \bC^*_\hbar} (X)$ via the Kirwan map. 
For $s^\chi \in \chr(\sG)$, we denote by $S^\chi := s^\chi |_X$ its image under the Kirwan map. 
Let $\ev_0: \QM(X) \to \fX$ be the evaluation map at $0\in \bP^1$. 

Denote by $Q$ the collection of K\"ahler parameters and $a$ the collection of $\sA$-equivariant parameters. 
In later sections, for simplicity we will omit some of $q$, $\hbar$ and $a$ in the notation and write it as $V^{(\tau)} (Q)$. 

\begin{Definition}[\cite{pcmilec}] \label{Defn-vertex}
The \emph{vertex function} for $X$ with descendent $\tau$ is defined as
$$
V^{(\tau)} (q, \hbar, Q, a) := \sum_{d \in \Eff(X)} Q^d (\ev_\infty)_* ( \ev_0^* \tau \cdot  \widehat\cO_{\vir, d} ) \ \in \ K_{\sA \times \bC^*_\hbar \times \bC^*_q} (X)_\loc [[Q^{\Eff(X)}]]. 
$$
Moreover, the renormalized vertex function is defined as
$$
\widetilde V^{(\tau)} (q, \hbar, Q, a) := e^{\frac{\la \ln S, \ln Q_\sharp \ra}{\ln q}} \cdot \frac{(q T_X^{1/2})_\infty}{(\hbar T_X^{1/2})_\infty} \cdot V^{(\tau)} (q, \hbar, Q, a),
$$
where $T_X^{1/2}$ is the vector bundle associated with $\sN$, and the shifted K\"ahler parameter $Q_\sharp := Q \cdot (-\hbar ^{1/2})^{-\det T_X^{1/2}}$.
This is a multi-valued meromorphic function in $Q$.
\end{Definition}

Vertex functions as above are introduced by A. Okounkov \cite{pcmilec, AOelliptic}, originally called the \emph{bare vertex}, and the boundary condition $f(\infty)\in X$ is described as ``nonsingular".
The name ``vertex" is motivated from its relationship with the topolocial vertex in relative Donaldson--Thomas theory. 
There is another type of vertex functions $\widehat V^{(\tau)} (q, \hbar, Q, a)$, called the \emph{capped vertex}, defined by relative quasimaps, where the domain $\bP^1$ are allowed to bubble into a chain of $\bP^1$'s, and the $\infty$ of the last $\bP^1$ is required to be nonsingular. 

Being able to add the two types of insertions -- nonsingular and relative -- to quasimaps, one can define various invariants and operators on the $K$-theory ring of $X$.
Important examples include the \emph{gluing operator}, which appears in a degeneration formula for the capped vertex, and the \emph{capping operator} $\Psi (q, Q)$, which relates the bare and capped vertices: 
\begin{equation} \label{Psi-V}
\widehat V^{(\tau)} (q, Q) = \Psi (q, Q) \cdot V^{(\tau)} (q, Q). 
\end{equation}

A key feature one can obtain from the capping operator $\Psi (q, Q)$ is the $q$-difference equation. 
In \cite{pcmilec}, it is proved that $\Psi (q, Q)$ satisfies two systems of $q$-difference equations, with respect to the K\"ahler and equivariant paramters respectively.
It follows that (the renormalized) bare vertex functions also satisfy two corresponding systems of $q$-difference equations, conjectured to be related to each other by 3d mirror symmetry. 
The different asymptotic behavior of the vertex function with respect to the two sets of $q$-difference equations also leads to the seminal introduction of elliptic stable envelopes \cite{AOelliptic}. 

Implicitly, $q$-difference equations can also be described as the $q$-difference modules generated by their solutions. 
We will call the one generated by renormalized vertex functions the \emph{quantum $q$-difference modules}.  

Taking $q\to 1$, one can obtain from the quantum $q$-difference module a commutative algebra, called the \emph{Bethe algebra}\cite{AOBethe}, also referred to as the PSZ quantum $K$-theory ring \cite{PSZ, SZ}, or Wilson loop algebra \cite{JMNT}.
It is related to, but in general not the same as the usual quantum $K$-theory ring in the sense of Givental--Lee \cite{Giv-WDVV, Lee, GL}.
The idea dates back to the gauge/Bethe correspondence of Nekrasov--Shatashvili \cite{NS1, NS2}, which relates enumerative geometry to integrable systems such as quantum spin chains (see also \cite{PSZ, UY}). 

Geometrically, the Bethe algebra may also be defined (as in \cite{PSZ}) by inserting $3$ relative conditions. 
The evaluation map from the moduli of relative quasimaps to $X$ is proper, which allows one to take the non-equivariant limit $q\mapsto 1$. 
The class obtained this way is called the \emph{quantum tautological class} $\widehat \tau (Q) := \widehat V^{(\tau)} (1, \hbar, Q, a) \in K_{\sT \times \bC^*_\hbar} (X)_\loc [[Q^{\Eff(X)}]]$. 
They are of the form $\widehat \tau (Q) = \tau |_X + O(Q)$. 

In later sections, we will explore the quantum $q$-difference module and Bethe algebra from the aspect of virtual Coulomb branch. 
The following result will be useful. 

\begin{Lemma}[{\cite[Lemma~4.16]{SZ}}] \label{lemma-Bethe-hat}
Given $\tau, \eta\in K_{\sT \times \bC^*_\hbar} (\fX)$, one has $\widehat{\tau \eta} (Q) = \widehat\tau (Q) * \widehat\eta (Q)$, where $*$ is the product in the Bethe algebra. 
\end{Lemma}

\vspace{2em}

\section{Verma modules and vertex functions: abelian case} \label{sec-Verma-vertex}

In this section, we study the interplay between the virtual Coulomb branch and quasimaps in the abelian case. 
We adopt the notations from Section \ref{sec-ab}. 
Recall that $\sG = \sK$ is a $k$-torus, $\sN = \bigoplus_{i=1}^n \bC_{\chi_i}$, and $\sT = (\bC^*)^n$ is the standard $n$-torus. 
For simplicity, we assume that $\chi_i \neq 0$, for any $i$. 

\subsection{Hypertoric varieties} \label{sec-hypertoric}

Holomorphic symplectic quotients by tori are hypertoric varieties, introduced in \cite{BD, Got}. 
Similarly to their toric analogue, the geometry of hypertoric varieties are well-studied, everything of which can be expressed explicitly in terms of combinatorial data. 
For example, they can be described by the data of hyperplane arrangements. 
However, in this paper we will not adopt that language. 

Let $\theta \in \chr (\sK)_\bR$ be a stability condition. 
In the following, we will assume that $\theta$ is chosen generically, and hence the hypertoric variety $X = \mu^{-1} (0) /\!/_\theta \sK$ is \emph{smooth} \footnote{This means that the action of $\sK$ is \emph{unimodular}.}.
We would like to characterize all possible generic choices of $\theta$. 

Regard the characters $\chi_i $, $1\leq i\leq n$ as vectors in $\chr(\sK)_\bR$.  
A codimension-$1$ hyperplane in $\chr(\sK)_\bR \cong \bR^k$ is called a \emph{wall} if it is spanned by a subset of $\{\chi_i \mid 1\leq i\leq n\}$. 
The connected components of the complement of walls are referred to as \emph{chambers}. 

These walls and chambers are exactly those appearing in the variation of GIT. 
The stability condition $\theta$ is \emph{generic} if and only if it does not lie on any walls.
We call the chamber $\cK$ which contains $\theta$ the \emph{K\"ahler cone} of $X$, and denote it by $\cK(X)$. 

\begin{Remark}
The combinatorial desription of the wall-and-chamber structure follows essentially from the Hilbert--Mumford criterion. 
A more general treatment is given in \cite{HL-S}, where the walls are codimension-1 faces of the \emph{zonotope}. 
In our case, the zonotope is the convex hull in $\chr(\sK)_\bR$ generated by vectors of the form $\sum_{i\in \Pol} \chi_i - \sum_{i\not\in \Pol} \chi_i$, for all possible choices of polarizations. 
It then implies that the edges of the zonotope are parallel to $\pm \chi_i$'s and our description follows. 
\end{Remark}

Next we look at the dual picture. 

\begin{Definition}
A vector $\rho \in \cochar(\sK)_\bZ$ is called a \emph{circuit}, if it is the normal vector of a wall in $\chr(\sK)_\bR$ satisfying $\langle \theta, \rho \rangle \geq 0$, and primitive, in the sense that $\rho / m\in \cochar(\sK)_\bZ$ for $m\in \bZ$ implies that $m =1$. 
\end{Definition}

The collection of circuits also determines a wall-and-chamber structure in the dual space $\cochar(\sK)_\bR$. 
To avoid confusion with the previous terminology, we call them as \emph{cowalls} and \emph{cochambers}. 
The dual $\cK(X)^\vee$ of the K\"ahler cone is the \emph{effective cone} $\Eff(X)_\bR$ of $X$.
It is the cone generated by all circuits, which is the closure of a union of several cochambers. 

$\sT$-fixed points of $X$ are characterized as follows.
Let $\bp \subset \{1, \cdots, n\}$ be a subset of size $k$, such that the characters $\{\chi_i \mid i\in \bp\}$ are linearly independent. 
There is a unique way of writing the stability condition $\theta$ as an $\bR$-linear combination $\theta = \sum_{j \in \bp} c_j \chi_j$, where all $c_j \neq 0$. 
Define \footnote{These are denoted by $\cA_\bp^\pm$ in \cite{Shenf, SZ}. } $\bp^\pm \subset \bp$ to be the subset such that $c_j >0$ (resp. $c_j<0$) if and only if $j\in \bp^+$ (resp. $j\in \bp^-$). 
Then there is a fixed point, still denoted by $\bp$, whose representative $(\vec z, \vec w) \in T^* \bC^n$ can be taken simply as
$$
z_j =1, \ w_j = 0, \text{ if } j \in \bp^+; \qquad z_j = 0, \ w_j =1, \text{ if } j \in \bp^-; \qquad z_i = w_i = 0, \text{ if } i \not\in \bp. 
$$

In other words, $\sT$-fixed points of $X$ are given by simplicial cones in $\chr(\sK)_\bR$ cut out by walls and containing $\theta$. 
More precisely, the cone corresponding to $\bp$ is generated by $\{\chi_j \mid j\in \bp^+ \} \cup \{ - \chi_j \mid j\in \bp^-\}$, which we denote by $\cK (\bp)$ and call the \emph{K\"ahler cone} of $\bp$. 
In particular, it contains the K\"ahler cone $\cK (X)$ of $X$. 

Dually, the effective cone $\Eff(\bp) \subset \cochar (\sK)_\bR$, consisting of effective curve classes mapping to the fixed point $\bp$ , is the dual of $\cK(\bp)$. 
We have $\Eff(\bp) \subset \Eff(X)$, and it is a union of closures of cochambers. 
A curve class $d \in \cochar (\sK)$ lies in $\Eff(\bp)$ if and only if $\langle \chi_j , d \rangle \geq 0$ for $j\in \bp^+$ and $\langle \chi_j, d \rangle \leq 0$ for $j\in \bp^-$.

\begin{Example} \label{A_2}
Figure \ref{Fig} shows the example $\cA_2$, the minimal resolution of $\bC^2 / \bZ_2$, where $\sK = (\bC^*)^2$, $\sN = \bC^3$, and $\chi_1 = (1,0)$, $\chi_2 = (0,1)$, $\chi_3 = (-1,-1)$. 
The stability condition is chosen such that $\theta_1 > \theta_2 >0$. 
There are three fixed points: 
\begin{itemize}
\setlength{\parskip}{1ex}
\item $\bp_{12} = \{1,2\}$, with $\cK (\bp_{12})$ generated by $\chi_1, \chi_2$, and $\Eff(\bp_{12})$ generated by $\rho_1$, $\rho_2$; 
\item $\bp_{13} = \{1,3\}$, with $\cK(\bp_{13})$ generated by $\chi_1, -\chi_3$, and $\Eff(\bp_{13})$ generated by $\rho_1$, $\rho_3$; 
\item $\bp_{23} = \{2,3\}$, with $\cK(\bp_{23})$ generated by $-\chi_2, -\chi_3$, and $\Eff(\bp_{23})$ generated by $\rho_2$, $\rho_3$. 
\end{itemize}
The K\"ahler and effective cone are shown as shaded regions. 

\begin{figure}
\centering
\begin{tikzpicture}[scale=0.7]

\fill[gray!30] (0,0) -- (3,0) -- (3,3) -- cycle ;

\draw (-3,0) -- (3,0);
\draw (0,-3) -- (0,3);
\draw (-3,-3) -- (3,3); 

\draw[->] (0,0) -- (2,0) node[anchor=north west] {$\chi_1$}; 
\draw[->] (0,0) -- (0,2) node[anchor=south east] {$\chi_2$}; 
\draw[->] (0,0) -- (-2,-2) node[anchor=north west] {$\chi_3$}; 

\filldraw (2,1) circle (1pt) node[anchor = west] {$\theta$};

\fill[gray!30] (9,0) -- (12,-3) -- (12,3) -- (9,3) -- cycle ;

\draw (6,0) -- (12,0);
\draw (9,-3) -- (9,3);
\draw (6,3) -- (12,-3); 

\draw[->] (9,0) -- (11,0) node[anchor=north west] {$\rho_2$}; 
\draw[->] (9,0) -- (9,2) node[anchor=south east] {$\rho_1$}; 
\draw[->] (9,0) -- (11,-2) node[anchor=north east] {$\rho_3$}; 

\node at (10.5, 1.5) {$\Eff(X)$};

\end{tikzpicture}
\caption{Wall-and-chamber structures in $\chr(\sK)_\bR$ and $\cochar(\sK)_\bR$ for $\cA_2$}  \label{Fig}
\end{figure}

\end{Example}

\subsection{Vertex functions}

Let $L_i \in K_{\sT \times \bC^*_\hbar} (X)$ be the $K$-theory class of the tautological line bundle associated with the character $\chi_i$; let $S_j$ be the $K$-theory class of the line bundle associated with the standard basis of character $s_j$. 
The relation between them is given by $L_i = a_i \prod_{j=1}^k S_j^{\langle \chi_i, e_j \rangle}$, as already appeared in (\ref{x-s-rel}). The equivariant $K$-theory of $X$ can be expressed as
\begin{eqnarray} \label{K-ring}
K_{\sT \times \bC^*_\hbar} (X) &\cong& \dfrac{ \bC [a_1^{\pm 1}, \cdots, a_n^{\pm 1}, \hbar^{\pm 1}, s_1^{\pm 1}, \cdots, s_k^{\pm 1} ] }{ \langle \prod_{i: \langle \chi_i, \rho \rangle > 0 } (1 - x_i) \prod_{i: \langle \chi_i, \rho \rangle < 0} (1 - \hbar x_i) \mid \rho: \text{circuits} \rangle } \\
&\cong& \dfrac{ \bC [a_1^{\pm 1}, \cdots, a_n^{\pm 1}, \hbar^{\pm 1}, s_1^{\pm 1}, \cdots, s_k^{\pm 1} ] }{ \bigcap_{\bp\in X^\sT} \la 1 - x_j , j\in \bp^+; 1 - \hbar x_j, j \in \bp^- \ra }  \label{K-ring-2} , 
\end{eqnarray}
where $x_i\mapsto L_i$, and $s_j \mapsto S_j$ under the Kirwan map. 

Given a fixed point $\bp \in X^\sT$, the restrictions $L_i |_\bp$ and $S_j |_\bp$ are easy to determine. 
The boundary vectors of the effective cone $\Eff(\bp)$ form a $\bZ$-basis of $\cochar(\sK)$. 
We then have
\begin{equation} \label{K-ring-bp}
L_i |_\bp = 1, \qquad \text{for } i\in \bp^+, \qquad L_i |_\bp = \hbar^{-1} , \qquad \text{for } i\in \bp^-.
\end{equation}
All $S_j |_\bp$, $1\leq j\leq k$ and $L_i |_\bp$ for $i\not\in \bp$ can be solved from the relations in (\ref{K-ring}). 

The following explict formula for the vertex functions of $X$ is computed in \cite{SZ}, by $\sT$-localization and reduction to fixed points. 
The components of the vertex function at the $\sT$-fixed points can be obtained directly, given the following restriction formula for $L_i$'s to the fixed points.

\begin{Proposition}[\cite{SZ}] \label{formula-V}
Let $X$ be a hypertoric variety. 
The vertex function for $X$ is
$$
V^{(\tau (s))} (Q) = \sum_{d\in \Eff(X)} Q^d \cdot \prod_{i=1}^n \left[ (-q^{1/2} \hbar^{-1/2})^{D_i} \frac{(\hbar L_i)_{D_i}}{ (q L_i)_{D_i} }  \right] \cdot \tau (q^{\la e_1^* , d\ra} S_1, \cdots, q^{\la e_k^*, d \ra} S_k ). 
$$
where $D_i := \langle \chi_i, d \rangle$. 
\end{Proposition}

\subsection{Mixed polarization}

Recall that as in Section \ref{sec-pol}, one can choose a different polarization $\sN(\Pol) \subset T^* \sN$ as the starting Lagrangian subspace in the definition of the Coulomb branch. 
Here we would like to furthur generalize this concept, by considering different polarization simultaneously.

\begin{Definition} \label{loc-eff-pol}
Let $\cW$ be a cochamber contained in the effective cone $\Eff(X)$. 
Define the \emph{mixed polarization} with respect to $\cW$ as: 
$$
\Pol [\cW] := \{ i \mid \langle \chi_i, d \rangle \geq 0 , \text{ for all } d\in \cW \}. 
$$
\end{Definition}

It is well-defined by definition of cochambers. 
Therefore, $\sN (\Pol[\cW]) = \bigoplus_{i: \langle \chi_i, \cW \rangle \geq 0} \bC_{\chi_i} \oplus \bigoplus_{i: \langle \chi_i, \cW \rangle \leq  0} \hbar^{-1} \bC_{-\chi_i}$. 
Note that since $\cW$ is a cochamber, $\langle \chi, \cW \rangle = 0$ would imply $\chi=0$, which cannot happen by our nondegeneracy assumption that all $\chi_i\neq 0$. 

Intuitively, we would like to choose a particular ``local" polarization for each degree class $d\in \cochar(\sK)$. 
Precisely, for $d\in \cW \cup (-\cW)$ lying in an effective cochamber or its opposite, we would like to choose the polarization $\Pol[\cW]$. 
This ``local choice", although clear for degrees $d$ lying in cochambers, admits ambiguity for degree classes lying on cowalls, i.e. for such $d$, $\Pol[\cW]$'s coming from the two adjacent cochambers are different. 
However, it turns out that the associated virtual structure sheaves $r_d (\Pol[\cW])$, $t_d (\Pol[\cW])$ are well-defined. 

\begin{Lemma}
Let $\cW$ and $\cW'$ be two cochambers such that $d \in (\overline\cW \cup (-\overline\cW) ) \cap (\overline\cW' \cup (-\overline\cW') )$. 
Let $\Phi$, $\Phi'$, $\Xi$, $\Xi'$ be the isomorphisms defined in (\ref{Phi}) (\ref{Phi-M}) (\ref{Xi}) (\ref{Xi-M}), associated with $\Pol[\cW]$ and $\Pol[\cW']$ respectively. 
Then
$$
\Xi  \Phi (r_d (\Pol[\cW])) = \Xi'  \Phi' (r_d (\Pol[\cW'])), \qquad \Xi  \Phi (t_d (\Pol[\cW])) = \Xi'  \Phi' (t_d (\Pol[\cW'])) . 
$$
\end{Lemma}

\begin{proof}
For  any $i$ such that $i\in \Pol(\cW)$ but $i\not\in \Pol(\cW')$, the inequalities $\la \chi_i , \cW \ra \geq 0$ and $\la \chi_i, \cW' \ra \leq 0$ hold simultaneously. 
Hence if $d\in \overline\cW \cap \overline\cW'$, we must have $\la \chi_i , d \ra = 0$. 
The contribution from the $i$-th summand to $r_d (\Pol[\cW])$ is then trivial, which coincides the contribution from the $i$-th summand to $r_d (\Pol[\cW'])$. 
\end{proof}

\begin{Definition} \label{defn-br_d}

For any $d \in \Eff(X) \cup (-\Eff(X))$, choose an effective cochamber $\cW \subset \Eff(X)$ such that $d\in \overline\cW \cup (- \overline\cW)$. 
Define the elements 
$$
\br_d := \Xi  \Phi (r_d (\Pol[\cW]))  \quad \in \quad \cA_\sT (\sK, \sN)_\loc .
$$
For $d \not\in \Eff(X) \cup (-\Eff(X))$, define $\br_d := r_d$. 
\end{Definition} 


\begin{Proposition} \label{prop-mix}

Denote $D_i := \langle \chi_i, d \rangle$, $C_i := \langle \chi_i, c \rangle$. 

\begin{enumerate}[1)]
\setlength{\parskip}{1ex}
\item For any $d\in \Eff(X)$, we have the explicit presentations
$$
\br_d =  r_d \cdot \prod_{i: D_i < 0} \left[ (- q^{1/2} \hbar^{-1/2} )^{D_i} \frac{(\hbar  x_i)_{D_i}}{(q x_i )_{D_i} } \right]^{-1}  , \quad 
\br_{-d} = \prod_{i: D_i < 0} \left[ (- q^{1/2} \hbar^{-1/2} )^{D_i} \frac{(\hbar x_i)_{D_i}}{(q x_i )_{D_i} } \right]^{-1} \cdot r_{-d} .
$$

\item The anti-automorphism $\tau$ of $\cA_\sT (\sK, \sN)_\loc$ sends $\tau(\br_d) = \br_{-d}$.

\item If $c, d\in \Eff(X)$, then 
\begin{equation} \label{ab-rel-0-eff}
\br_c \cdot \br_d = \br_{c+d} , \qquad \br_{-c} \cdot \br_{-d} = \br_{-c-d}. 
\end{equation}

\item For any $d\in \Eff(X)$, we have 
\begin{equation} \label{ab-rel-eff}
\br_{d} \cdot \br_{-d} = \prod_{i=1}^n \left[ (- q^{1/2} \hbar^{-1/2} )^{-D_i} \frac{(\hbar x_i)_{ -D_i } }{(q x_i)_{-D_i } } \right] , \quad  
\br_{-d} \cdot \br_{d} = \prod_{i=1}^n \left[ (- q^{1/2} \hbar^{-1/2} )^{D_i} \frac{(\hbar x_i)_{ D_i } }{(q x_i)_{D_i} } \right]^{-1} . 
\end{equation}

\item Let $\cC\subset \Eff(X)$ be a $k$-dimensional cone.
The localized virtual Coulomb branch $\cA_\sT (\sK, \sN)_\loc$ is generated by $\{ \br_d \mid d\in \overline\cC \cup (-\overline\cC) \}$ over $K_{\sK \times \sT \times \bC^*_\hbar \times \bC^*_q} (\pt)_\loc$.

\item For $d\in \Eff(X)$, the action of $\br_{\pm d}$ on $t_c$ is
$$
t_c \cdot \br_d = t_{c+d}, \qquad t_c \cdot \br_{-d} = t_{c-d} \cdot \prod_{i=1}^n (-q^{1/2} \hbar^{-1/2} )^{-D_i} \frac{(\hbar x_i)_{-D_i} }{ (qx_i)_{-D_i}}.
$$

\end{enumerate}
\end{Proposition}

\begin{proof}
1), 4) and 6) are straightforward from the definition of $\br_d$, the explicit presentations of $\cA_\sT (\sK, \sN)$ and the map $\Xi \Phi$.
2) follows from Corollary \ref{anti-isom}. 
5) follows from the observation that $\cA_\sT (\sK, \sN)$ is generated by $r_d$'s for $d\in \overline\cW \cup (-\overline\cW)$, where $\cW$ is any cochamber. 

It remains to prove 3), which follows from a tedious but straightforward calculation. 
We will adopt the quantum Hamiltonian reduction approach as described in Example \ref{ab-pt}. 
Let $c, d\in \Eff(X)$. 
Then the product $\br_c \cdot \br_d$ can be computed via their lifts
\ben
&& \tilde r_{\iota (c)} \prod_{i: C_i < 0} (- q^{1/2} \hbar^{-1/2} )^{-C_i} \frac{(q  \tilde x_i)_{C_i}}{(\hbar \tilde x_i )_{C_i} } \cdot \tilde r_{\iota(d)} \cdot \prod_{i: D_i < 0} (- q^{1/2} \hbar^{-1/2} )^{-D_i} \frac{(q  \tilde x_i)_{D_i}}{(\hbar \tilde x_i )_{D_i} }  \\
&=& \prod_{i: C_i>0} \tilde r_{b_i}^{C_i} \prod_{i: C_i<0} \tilde r_{-b_i}^{-C_i} \prod_{i: C_i < 0} (- q^{1/2} \hbar^{-1/2} )^{-C_i} \frac{(q  \tilde x_i)_{C_i}}{(\hbar \tilde x_i )_{C_i} } \\
&& 
\cdot \prod_{i: D_i>0} \tilde r_{b_i}^{D_i} \prod_{i: D_i<0} \tilde r_{-b_i}^{-D_i} \prod_{i: D_i < 0} (- q^{1/2} \hbar^{-1/2} )^{-D_i} \frac{(q  \tilde x_i)_{D_i}}{(\hbar \tilde x_i )_{D_i} } . 
\een
Let us look at the contribution from the $i$-th factor, which falls into six cases (the case that either $C_i =0$ or $D_i=0$ is trivial).
\begin{enumerate}[(i)]
\setlength{\parskip}{1ex}

\item $C_i, D_i >0$. 
Then $\tilde r_{b_i}^{C_i} \tilde r_{b_i}^{D_i} = \tilde r_{b_i}^{C_i +D_i}$. 

\item $C_i, D_i <0$. 
Then $\tilde r_{-b_i}^{-C_i} \cdot (- q^{1/2} \hbar^{-1/2} )^{-C_i} \dfrac{(q  \tilde x_i)_{C_i}}{(\hbar \tilde x_i )_{C_i} } \cdot \tilde r_{-b_i}^{-D_i} \cdot (- q^{1/2} \hbar^{-1/2} )^{-D_i} \dfrac{(q  \tilde x_i)_{D_i}}{(\hbar \tilde x_i )_{D_i} } = \tilde r_{-b_i}^{-C_i-D_i} \cdot (- q^{1/2} \hbar^{-1/2} )^{-C_i-D_i} \dfrac{(q  \tilde x_i)_{C_i+D_i}}{(\hbar \tilde x_i )_{C_i + D_i} } $. 

\item $C_i >0>D_i$, $C_i + D_i >0$. 
Then $\tilde r_{b_i}^{C_i} \tilde r_{-b_i}^{-D_i} (- q^{1/2} \hbar^{-1/2} )^{-D_i} \dfrac{(q  \tilde x_i)_{D_i}}{(\hbar \tilde x_i )_{D_i} } = \tilde r_{b_i}^{C_i +D_i} \cdot \tilde r_{b_i}^{-D_i} \tilde r_{-b_i}^{-D_i} \cdot (- q^{1/2} \hbar^{-1/2} )^{-D_i} \dfrac{(q  \tilde x_i)_{D_i}}{(\hbar \tilde x_i )_{D_i} }$. 
By (\ref{ab-pt-3}), $\tilde r_{b_i}^{-D_i} \tilde r_{-b_i}^{-D_i} = (-q^{1/2} \hbar^{-1/2})^{D_i} \dfrac{(\hbar \tilde x_i)_{D_i}}{(q \tilde x_i)_{D_i} }$. 
Therefore we have $\tilde r_{b_i}^{C_i +D_i}$. 

\item $C_i >0> D_i$, $C_i + D_i <0$. 
Similar to (iii), $\tilde r_{b_i}^{C_i} \tilde r_{-b_i}^{C_i} \cdot \tilde r_{-b_i}^{-C_i -D_i} \cdot (- q^{1/2} \hbar^{-1/2} )^{-D_i} \dfrac{(q  \tilde x_i)_{D_i}}{(\hbar \tilde x_i )_{D_i} }$.
Again $\tilde r_{b_i}^{C_i} \tilde r_{-b_i}^{C_i} = (-q^{1/2} \hbar^{-1/2})^{-C_i} \dfrac{(\hbar \tilde x_i)_{-C_i}}{(q \tilde x_i)_{-C_i} }$. 
We have $\tilde r_{-b_i}^{-C_i -D_i} \cdot (- q^{1/2} \hbar^{-1/2} )^{-C_i - D_i} \dfrac{(q  \tilde x_i)_{C_i+ D_i}}{(\hbar \tilde x_i )_{C_i + D_i} }$. 

\item $C_i <0<D_i$, $C_i + D_i >0$. 
Then $\tilde r_{-b_i}^{-C_i} \cdot  (- q^{1/2} \hbar^{-1/2} )^{-C_i} \dfrac{(q  \tilde x_i)_{C_i}}{(\hbar \tilde x_i )_{C_i} } \cdot \tilde r_{b_i}^{D_i} = \tilde r_{-b_i}^{-C_i} \tilde r_{b_i}^{D_i} \cdot  (- q^{1/2} \hbar^{-1/2} )^{-C_i} \dfrac{(q^{D_i} q  \tilde x_i)_{C_i}}{(q^{D_i} \hbar \tilde x_i )_{C_i} } = \tilde r_{-b_i}^{-C_i} \tilde r_{b_i}^{-C_i} \cdot \tilde r_{b_i}^{C_i + D_i} \cdot  (- q^{1/2} \hbar^{-1/2} )^{-C_i} \dfrac{(q^{D_i} q  \tilde x_i)_{C_i}}{(q^{D_i} \hbar \tilde x_i )_{C_i} }$. 
By (\ref{ab-pt-3}), $\tilde r_{-b_i}^{-C_i} \tilde r_{b_i}^{-C_i} =  (- q^{1/2} \hbar^{-1/2} )^{C_i} \dfrac{(q  \tilde x_i)_{-C_i}}{(\hbar \tilde x_i )_{-C_i} }$. 
We have $\tilde r_{b_i}^{C_i +D_i}$. 

\item $C_i <0<D_i$, $C_i + D_i <0$. 
Simillar to (v), $\tilde r_{-b_i}^{-C_i - D_i} \cdot \tilde r_{-b_i}^{D_i} \tilde r_{b_i}^{D_i} \cdot  (- q^{1/2} \hbar^{-1/2} )^{-C_i} \dfrac{(q^{D_i} q  \tilde x_i)_{C_i}}{(q^{D_i} \hbar \tilde x_i )_{C_i} }$. 
Again $\tilde r_{-b_i}^{D_i} \tilde r_{b_i}^{D_i} =  (- q^{1/2} \hbar^{-1/2} )^{-D_i} \dfrac{(q  \tilde x_i)_{D_i}}{(\hbar \tilde x_i )_{D_i} }$. 
We have $\tilde r_{-b_i}^{-C_i -D_i} \cdot (- q^{1/2} \hbar^{-1/2} )^{-C_i - D_i} \dfrac{(q  \tilde x_i)_{C_i+ D_i}}{(\hbar \tilde x_i )_{C_i + D_i} }$. 
\end{enumerate}
Collecting the results (i)-(v), we get the lift of $\br_{c+d}$. 
The case for $\br_{-c} \cdot \br_{-d}$ follows from 2). 
\end{proof}

\begin{Remark}
3) and 4) are not satisfied by the generators $r_d$ before taking the mixed polarization. 
They only commute if $c,d$ lie in the same \emph{cochamber}.
\end{Remark}

\begin{Remark}
Although the definition of $\br_{\pm d}$'s depends on the particular Higgs branch $X$, the product $\br_d \cdot \br_{-d}$, $\br_{-d} \cdot \br_d$ and $t_c \cdot \br_{\pm d}$ only depend on the $\sK$-representation $\sN$, or on the virtual Coulomb branch $\cA_\sT(\sK, \sN)_\loc $ itself. 
\end{Remark}

\subsection{Valuation-like ring} \label{Sec-val}

Up to now, in order to mix the polarizations, we used the isomorphisms introduced in Section \ref{sec-vary-pol}, which requires passing to the localized ring. 
The resulting Cartan subalgebra $K_{\sK \times \sT \times \bC^*_\hbar \times \bC^*_q} (\pt)_\loc$ is indeed a field, which is not enough for considering the Verma modules. 
We need to restrict to certain subalgebras for that purpose. 

Recall the presentation $K_{\sT \times\bC^*_\hbar} (X) \cong \bC[a^{\pm 1}, \hbar^{\pm 1}, s^{\pm 1}] / \bigcap_{\bp\in X^\sT} I(\bp)$ in (\ref{K-ring-2}), where the ideal
$$
I (\bp) : = \la 1 - x_j , j\in \bp^+; 1 - \hbar x_j, j \in \bp^- \ra .
$$
Geometrically, $\Spec K_{\sT \times\bC^*_\hbar} (X) = \bigcup_{\bp \in X^\sT} \Spec K_{\sT \times\bC^*_\hbar} (\pt)$ embeds into the torus $\Spec K_{\sK \times \sT \times \bC^*_\hbar \times \bC^*_q} (\pt) \cong (\bC^*)^{k+n+2}$, where each $\Spec K_{\sT \times\bC^*_\hbar} (\pt) \cong (\bC^*)^{n+1}$ is an irreducible component.

\begin{Definition} \label{defn-val-ring}
Consider the following valuation-like rings. \footnote{The definition looks more natural if we view $q\in \bC$ as a constant. Then $K_{\sK \times \sT \times \bC^*_\hbar} (\pt)_{\bp}$ is the local ring of $\Spec K_{\sK \times \sT \times \bC^*_\hbar} (\pt)$ at the irreducible component $\Spec K_{\sT \times\bC^*_\hbar} (\pt)$, and $K_{\sT \times\bC^*_\hbar} (X)$ is the \emph{semilocal} ring at the subvariety $\Spec K_{\sT \times\bC^*_\hbar} (X)$.}
\begin{enumerate}[1)]
\setlength{\parskip}{1ex}

\item Define $K_{\sK \times \sT \times \bC^*_\hbar \times \bC^*_q} (\pt)_{\bp}$ the localization of $K_{\sK \times \sT \times \bC^*_\hbar \times \bC^*_q} (\pt)$ at the multiplicative subset 
$$
K_{\sK \times \sT \times \bC^*_\hbar \times \bC^*_q} (\pt) - \bigcup_{N_j , N'_j \in \bZ} \la 1 -q^{N_j} x_j , j\in \bp^+; 1 - q^{N'_j} \hbar x_j, j \in \bp^- \ra.
$$
There is an evaluation map $(-)|_\bp : K_{\sK \times \sT \times \bC^*_\hbar \times \bC^*_q} (\pt)_{\bp} \to K_{\sT \times \bC^*_\hbar \times \bC^*_q} (\bp)_\loc$, defined by the quotient at the maximal ideal $\la 1 - x_j , j\in \bp^+; 1 - \hbar x_j, j \in \bp^- \ra$. 

\item Define $K_{\sK \times \sT \times \bC^*_\hbar \times \bC^*_q} (\pt)_X$ as the localization of $K_{\sK \times \sT \times \bC^*_\hbar \times \bC^*_q} (\pt)$ at the multiplicative subset 
$$
K_{\sK \times \sT \times \bC^*_\hbar \times \bC^*_q} (\pt) - \bigcup_{\bp\in X^\sT} \bigcup_{N_j , N'_j \in \bZ} \la 1 -q^{N_j} x_j , j\in \bp^+; 1 - q^{N'_j} \hbar x_j, j \in \bp^- \ra.
$$
There is an evaluation map $(-)|_X : K_{\sK \times \sT \times \bC^*_\hbar \times \bC^*_q} (\pt)_X \to K_{\sT \times \bC^*_\hbar \times \bC^*_q} (X)_\loc$, defined by the quotient at the intersection of ideals $\bigcap_{\bp \in X^\sT} \la 1 - x_j , j\in \bp^+; 1 - \hbar x_j, j \in \bp^- \ra$. 

\end{enumerate}

\end{Definition}
Equivalently, $K_{\sK \times \sT \times \bC^*_\hbar \times \bC^*_q} (\pt)_X = \bigcap_{\bp \in X^\sT} K_{\sK \times \sT \times \bC^*_\hbar \times \bC^*_q} (\pt)_{\bp}$ is the intersection of the subrings in the fraction field $K_{\sK \times \sT \times \bC^*_\hbar \times \bC^*_q} (\pt)_\loc$.
The restriction of the maps $(-)|_\bp$, $(-)|_X$ to $K_{\sK \times \sT \times \bC^*_\hbar \times \bC^*_q} (\pt)$ is the usual Kirwan surjection.
Note that it cannot be defined for the field $K_{\sK \times \sT \times \bC^*_\hbar \times \bC^*_q} (\pt)_\loc$. 

\begin{Definition} \label{Defn-cA-X}
\begin{enumerate}[1)]
\setlength{\parskip}{1ex}

\item Define $\cA_{\sT} (\sK, \sN)_\bp$ as the subalgebra of $\cA_\sT (\sK, \sN)_\loc$ generated by $K_{\sK \times \sT \times \bC^*_\hbar \times \bC^*_q} (\pt)_{\bp}$ and $\br_d$ for all $d \in \Eff(\bp) \cup (-\Eff(\bp))$. 

\item Define $\cA_\sT (\sK, \sN)_X$ as the subalgebra of $\cA_\sT (\sK, \sN)_\loc$ generated by $K_{\sK \times \sT \times \bC^*_\hbar \times \bC^*_q} (\pt)_X$ and $\br_d$ for all $d \in \Eff(X) \cup (-\Eff(X))$.
\end{enumerate}

\end{Definition} 

The original virtual Coulomb branch $\cA_\sT (\sK, \sN)$ is graded by the lattice $\cochar (\sK)$:
$$
\cA_\sT (\sK, \sN) = \bigoplus_{d\in \cochar (\sK)} \cA_\sT^{d} (\sK, \sN), 
$$ 
where $\cA^{d}_\sT (\sK, \sN)$ is generated by the single element $r_d$. 
$\cA^{0}_\sT (\sK, \sN) = K_{\sK \times \sT \times \bC^*_\hbar \times \bC^*_q} (\pt)_\sN$ is the Cartan subalgebra. 
Same holds for the localized version $\cA_\sT^d (\sK, \sN)_\loc$, which is also generated by $\br_d$.
The subalgebras $\cA_{\sT} (\sK, \sN)_\bp$, $\cA_\sT (\sK, \sN)_X$ inherit the grading. 

\begin{Lemma} \label{br_d-generator}
If $d\in \Eff(X) \cup (-\Eff(X))$, then $\cA^d_{\sT} (\sK, \sN)_\bp \subset K_{\sK \times \sT \times \bC^*_\hbar \times \bC^*_q} (\pt)_\bp \cdot \br_d$, and $\cA^d_{\sT} (\sK, \sN)_X \cong K_{\sK \times \sT \times \bC^*_\hbar \times \bC^*_q} (\pt)_X \cdot \br_d$. 
\end{Lemma}

\begin{proof}
For any $c, d\in \Eff(\bp)$, since $\br_c \br_d = \br_{c+d}$ and $\br_{-c} \br_{-d} = \br_{-c-d}$, it suffices to prove that $\br_{-c} \br_d , \br_d \br_{-c} \in K_{\sK \times \sT \times \bC^*_\hbar \times \bC^*_q} (\pt)_\bp \cdot \br_d$, under the assumption that $-c+d \in \Eff(X)$ (the other case $-c+d \in -\Eff(X)$ follows by using the anti-involution $\tau$). 
But we have $\br_{-c} \br_d = \br_{-c} \br_c \br_{-c+d}$ and $\br_d \br_{-c} = \br_{-c+d} \br_c \br_{-c}$. 
The result then follows from the fact that $\br_{-c} \br_c, \br_c \br_{-c} \in K_{\sK \times \sT \times \bC^*_\hbar \times \bC^*_q} (\pt)_\bp$ for $c\in \Eff(\bp)$. 
\end{proof}

\subsection{Verma modules associated with fixed points}

Let $\bp\in X^\sT$ be a fixed point and $\Eff(\bp)$ be its effective cone.
Recall that $S_j$, $1\leq j\leq k$, are tautological line bundles on $X$ associated with $s_j$'s. 

\begin{Definition} \label{Verma}
A \emph{highest weight vector} $v$ of the virtual Coulomb branch $\cA_\sT (\sK, \sN)_\bp$ with respect to $\bp$, is defined to be such that 
$$
s_j \cdot v = S_j|_\bp \cdot   v, \qquad \cA_\sT^{-d} (\sK, \sN)_\bp \cdot v = 0, \qquad \text{for all } 1\leq j\leq k, \ d \in \Eff(\bp) \backslash \{0\}. 
$$
The \emph{Verma module} $M(\bp)$ is defined as the left module
$$
M(\bp) := \cA_\sT (\sK, \sN)_\bp \otimes_{\bigoplus_{d\in \Eff(\bp)} \cA_\sT^{-d} (\sK, \sN)_\bp } \bC  v. 
$$
\end{Definition}

In particular, any element in the Cartan subalgebra $\cA^0_\sT (\sK, \sN)_\bp$ acts on $v$ via the evaluation map $(-)|_\bp$. 
Moreover, for any $d \in \Eff(\bp) \backslash \{0\}$, $\br_{d} \br_{-d} v = 0$. 
These two results are compatible since we always have $(\br_d \br_{-d}) |_\bp = 0$. 

\begin{Lemma} \label{vanishing}
\begin{enumerate}[1)]
\setlength{\parskip}{1ex}

\item $\cA^d_\sT (\sK, \sN)_\bp \cdot v = 0$ for any $d\not\in \Eff(\bp)$. 
As a result, $M(\bp)$ is a vector space over $K_{\sT \times \bC^*_\hbar \times \bC^*_q} (\bp)_\loc$, with a basis formed by $\br_d \cdot v$, $d\in \Eff(\bp)$. 

\item There exists a unique bilinear form $\langle \ ,\  \rangle$ on $M(\bp)$, valued in the residue field $K_{\sT \times \bC^*_q \times \bC^*_\hbar} (\bp)_\loc$, such that $\langle v, v \rangle = 1$, and is contravariant, in the sense that $\br_d$ is adjoint to $\br_{-d}$ for any $d\in \Eff(\bp)$. 

\end{enumerate}

\end{Lemma}

\begin{proof}
By definition, for $d\not\in \Eff(\bp)$, to prove $\cA^d_\sT(\sK, N)_\bp \cdot v = 0$, it suffices to prove that for any decomposition $d = \sum_{\alpha=1}^l d^{(\alpha)}$, with $d^{(\alpha)} \in \pm \Eff(\bp)$, $\br_{d^{(l)}} \cdots \br_{d^{(1)}} v = 0$. 
By Lemma \ref{br_d-generator}, taking $-c$ to be the first $d^{(\alpha)}$ lying in $-\Eff(\bp)$ and $d = d^{(1)} + \cdots + d^{(\alpha-1)}$, it then suffices to prove that for $c, d \in \Eff(\bp)$, $-c+d\not\in \Eff(\bp)$, $\br_{-c} \br_d v = 0$. 

Let $\{\rho_j \mid j\in \bp\}$ be the $\bZ$-basis of $\Eff(\bp)$, i.e. satisfying $\la \chi_j ,\rho_i \ra = \pm \delta_{ij}$, for any $j\in \bp^\pm$. 
Then $c$ can be decomposed into a sum of $\rho_j$'s, and it suffices to prove the case where $c = \rho_{j_0}$ for some $j_0 \in \bp$. 
In that case, write $\br_d = \prod_j \br_{\rho_j}^{|D_j|}$. 
We must have $D_{j_0} = 0$ since $-c+d \not\in \Eff(\bp)$. 

On the other hand, a direct computation shows that $\br_{-\rho_i} \br_{\rho_j} = u \cdot \br_{\rho_j} \br_{-\rho_i}$ whenever $i\neq j$, where $u$ is a rational function in $1 - q^m x_i$, for $i\not\in \bp$, $m\in \bZ$. 
In particular, $u$ is invertible in $K_{\sK \times \sT \times \bC^*_\hbar \times \bC^*_q} (\pt)_\bp$. 
Therefore, up to a unit, we can move $\br_{-\rho_{j_0}}$ to the rightmost and it annihilates $v$. 
\end{proof}

\subsection{Verma modules and quasimaps}

We now describe how the Verma module of the virtual Coulomb branch can be realized from quasimaps. 
 
Recall that $\QM(X)^\circ_d$ is the moduli space of quasimaps from $\bP^1$ to $X$ of degree $d$, where $\infty$ maps to the stable locus; we denote by $\QM(X; \bp)_d^\circ \subset \QM(X)_d^\sT$ its subspace of quasimaps into $\bp$, and $K_0^{\sT \times \bC_\hbar^* \times \bC^*_q} (\QM (X; \bp)^\circ_d )_\loc := K_0^{\sT \times \bC_\hbar^* \times \bC^*_q} (\QM (X; \bp)^\circ_d ) \otimes \bC(q)$.  

\begin{Theorem} \label{thm-verma}
Let $X$ be a hypertoric variety, $\bp \in X^\sT$ be a fixed point, and $\Eff(\bp)$ be its effective cone. 
There is a natural action of $\cA_\sT (\sK, \sN)_\bp$ on $
\bigoplus_{d\in \Eff(\bp)} K_0^{\sT \times \bC_\hbar^* \times \bC^*_q} (\QM (X; \bp)^\circ_d )_\loc$, 
realizing it as the Verma module $M(\bp)$. 

\end{Theorem}

\begin{proof}
First, let us specify the fixed loci of the moduli of quasimaps.
A $(\sT \times \bC^*_\hbar \times \bC^*_q)$-fixed point of $\QM(X)^\circ_d$ consists of the data $(L, s)$, where $L$ is a line bundle of degree $d$ on $\bP^1$, whose restriction to $\bP^1 - \{0\}$ is trivial, and $s|_{\bP^1 - \{0\}}$ represents the coordinates of a fixed point $\bp\in X^\sT$. 
Let $z_1, \cdots, z_n, w_1, \cdots, w_n$ denote the components of $s$ corresponding to coordinates in $T^*\sN$. 
Then by the stability condition, the section $s|_{\bP^1 - \{0\}}$ has nonzero components $z_j \neq 0$, $j\in \bp^+$ and $w_j \neq 0$, $j\in \bp^-$, and all other components vanish. 
Given $d\in \Eff(\bp)$, choose a cochamber $\cW$ such that $d\in \overline\cW$. 
By Definition \ref{loc-eff-pol}, the indices $\bp^+ \subset \Pol$ and $\bp^- \cap \Pol = \emptyset$; in other words, for the splitting $T^* \sN = \sN (\Pol) \oplus \hbar^{-1} \sN (\Pol)$, all components of $s|_{\bP^1 - \{0\}}$ corresonding to $ \hbar^{-1} \sN (\Pol)$ vanish.

Therefore, the restriction of a quasimap $f$ to the formal neighborhood at $0$ defines a map from $D = \Spec \bC[[z]]$ to $\fX$, i.e. a point in $\cT(\Pol[\cW])$. 
Let $\cT(\Pol[\cW]; \bp)^\circ \subset \cT(\Pol[\cW])$ be the open subspace where the section $s |_{D^*}$ has the same nonzero components as $s |_{\bP^1 - \{0\}}$.  
We then have 
$$
( \QM(X; \bp)^\circ_d)^{\sT \times \bC^*_\hbar \times \bC^*_q} \cong \bp \cong (\cT_{-d} (\Pol[\cW]; \bp)^\circ )^{\sK \times \sT \times \bC^*_\hbar \rtimes \bC^*_q} , 
$$
where the subscript in $\cT_{-d}$ means the component over $[z^{-d}]$. 
By localization,
\begin{eqnarray} \label{QM-cT}
&& \bigoplus_{d\in \Eff(X)} K_0^{\sT \times \bC_\hbar^* \times \bC^*_q} (\QM (X; \bp)^\circ_d )_\loc \\ 
&\cong& \bigoplus_{d\in \Eff(X)} K_{\sT \times \bC_\hbar^* \times \bC^*_q} (\bp)_\loc \cdot [z^{-d}] 
\nonumber \\
&\cong& \bigoplus_{d\in \Eff(X)} K_0^{\sK \times \sT \times \bC_\hbar^* \rtimes \bC^*_q} (\cT_{-d} (\Pol[\cW]; \bp)^\circ ) 
\nonumber \\
&\cong& \bigoplus_{d\in \Eff(X)} K_0^{\sK \times \sT \times \bC_\hbar^* \rtimes \bC^*_q} (\cT_{-d} (\Pol[\cW]; \bp) ) \otimes_{K_{\sK \times \sT \times \bC^*_q\times \bC^*_\hbar} (\pt)} K_{\sT \times \bC_\hbar^* \times \bC^*_q} (\bp) 
\nonumber .
\end{eqnarray}
The last identity is due to the fact that the $K$-theory restriction to the open subspace $(\cT_{-d} (\Pol[\cW]; \bp)^\circ$ coincides with the Kirwan map $\kappa: K_{\sK \times \sT \times \bC^*_q\times \bC^*_\hbar} (\pt) \twoheadrightarrow K_{\sT \times \bC_\hbar^* \times \bC^*_q} (\bp)$, which furthermore after localization, coincides with the evaluation map $(-)|_\bp$ in Definition \ref{defn-val-ring}.  
So it suffices to consider the $K$-theory of $\cT_{-d} (\Pol[\cW]; \bp)$, and then apply the evaluation map. 

Recall that there is a natural right module $\cM_\sT (\sK, \sN)_\bp$ over $\cA_\sT (\sK, \sN )_\bp$, built up exactly from the space $\cT (\Pol[\cW]; \bp)$.  
Consider the nondegenerate $\bC$-bilinear pairing $\langle \ ,  \ \ra$ on $\cM_\sT (\sK, \sN )_\bp$, defined as
$$
\la t_c, t_d \ra = \delta_{c, -d}, \qquad \text{for any } c, d\in \cochar (\sK). 
$$
We can then identify $\cM_\sT (\sK, \sN )_\bp$ with its restricted dual, which is naturally a left $\cA_\sT (\sK, \sN )_\bp$-module. 
By 6) of Proposition \ref{prop-mix}, this left action is
$$
\br_d \cdot t_c = t_{d+c}, \qquad \br_{-d} \cdot t_c = \prod_{i=1}^n (-q^{1/2} \hbar^{-1/2})^{-D_i} \frac{(\hbar x_i)_{-D_i} }{(q x_i)_{-D_i} } \cdot t_{-d+c},   
$$
for any $d\in \Eff(X)$, $c\in \cochar (\sK)$. 

Under the evaluation map $(-)|_\bp$, we see exactly that $\br_{-d} \cdot t_c = 0$, and $t_0$ is a highest weight vector in the sense of Definition \ref{Verma}. 
It generates all $t_d$ for $d\in \Eff(\bp)$, under the action of $\br_d$'s. 
The theorem then follows from the identification (\ref{QM-cT}). 
\end{proof}

\subsection{Vertex functions as Whittaker functions} \label{sec-q-diff-ab}

In this subsection, we define the Whittaker vector for the virtual Coulomb branch and express the vertex function as a Whittaker function. 
The similar idea already appears in \cite{Bra-ins-1}.

Let $X$ be a hypertoric variety, and $\bp\in X^\sT$ be a fixed point. 
Let $M(\bp)$ be the associated Verma module, $v$ be its highest weight vector, and $\la \ , \ \ra$ be the contravariant bilinear form. 

Recall that $Q = (Q_1, \cdots, Q_k)$ is the collection of K\"ahler parameters. 
Given $d \in \cochar(\sK)$, $Q^d = \prod_{j=1}^k Q_j^{\la e_j^*, d\ra}$. 
For any $\chi \in \chr(\sK)$, denote by $q^{\chi Q\partial_Q}$ the $q$-difference operator such that $Q^d \mapsto q^{\la \chi, d\ra} Q^d$. 
In particular, $q^{Q_j \partial_{Q_j}} = q^{e_j^* Q \partial_Q}$ is the $q$-difference operator $Q_j \mapsto q Q_j$.  

\begin{Definition}
A \emph{Whittaker vector} of the virtual Coulomb branch $\cA_\sT (\sK, \sN)_\bp$ is an element $W_\bp (Q) \in M(\bp) [[Q^{\frac{1}{2} \Eff(\bp)}]]$, such that $\br_{-c} W_\bp (Q) = Q^{c/2} W_\bp (Q)$, for any $c\in \Eff(\bp)$. 
\end{Definition}

For any $d$, the generator $\br_{-d}$ admits an inverse 
$$
\br_{-d}^{-1} = \br_d (\br_{-d} \br_d)^{-1} = (\br_d \br_{-d})^{-1} \br_d 
$$
in the localized virtual Coulomb branch $\cA_\sT (\sK, \sN)_\loc$.
In general, they are not in the subalgebra $\cA_\sT (\sK, \sN)_\bp$, by the explicit presentation of $\br_{-d} \br_d$. 

\begin{Lemma}
	
\begin{enumerate}[1)]
\setlength{\parskip}{1ex}

\item For $d\in \Eff(\bp)$, the action of $\br_{-d}^{-1}$ on the Verma module $M(\bp)$ is well-defined.
In particular, we have $\br_{-d}^{-1} v = \br_d (\br_{-d} \br_d)^{-1} v = (\br_{-d} \br_d) |_\bp^{-1} \cdot \br_d v$. 

\item The Whittaker vector $W_\bp(Q)$ exists uniquely. 
Explicitly, it is
$$
W_\bp (Q) = \sum_{d \in \Eff(\bp)} Q^{d/2} \cdot \br_{-d}^{-1} v = \sum_{d \in \Eff(X)} Q^{d/2} \cdot \br_{-d}^{-1} v. 
$$
The defining equality $\br_{-c} W_\bp (Q) = Q^{c/2} W_\bp (Q)$ holds for any $c\in \Eff(X)$. 
\end{enumerate}

\end{Lemma}

\begin{proof}
For 1), it suffices to investigate the action
$$
\br_{-d}^{-1} \br_{d'} v = \br_d (\br_{-d} \br_d)^{-1} \br_{d'} v =  (\br_{-d} \br_d)^{-1} |_{s^\chi \mapsto q^{\la \chi, d' \ra} s^\chi , \bp} \cdot \br_{d+ d'} v.
$$
One can see explicitly from (\ref{ab-rel-eff}) that $(\br_{-d} \br_d) |_{s^\chi \mapsto q^{\la \chi, d' \ra} s^\chi , \bp} \neq 0$.
	
For 2), the second identify is true since $\br_d v = 0$ unless $d \in \Eff(\bp)$.  
The uniqueness follows from the definition $\br_{-c} W_\bp (Q) = Q^{c/2} W_\bp (Q)$, which implies a recursive relation between the $Q$-coefficients of $W_\bp (Q)$. 
We then check the explicit formula satiesfies the definition of Whittaker vector. 
\ben
Q^{c/2} \sum_{d\in \Eff(\bp)} Q^{d/2} \cdot \br_{-d}^{-1} v = \sum_{d\in \Eff(\bp)} Q^{(c+d)/2} \cdot \br_{-d}^{-1} v = \br_{-c} \sum_{d\in \Eff(\bp)} Q^{(c+d)/2} \cdot \br_{-c-d}^{-1} v.
\een
The RHS equals $\br_{-c} W_\bp (Q)$, because for any $d' \in \Eff(\bp)$, $\br_{-c} \br_{-d'}^{-1} v = \br_{-c} \frac{\br_{d'} v}{(\br_{-d'} \br_{d'} )|_\bp} \in \cA^{-c + d'}_\sT (\sK, \sN)_\bp \cdot v = 0$ unless $d' \in c+ \Eff(\bp)$, by Lemma \ref{vanishing}. 

For the last statement, given $c\in \Eff(X)$, write $c = d_1 - d_2$, where $d_1, d_2 \in \Eff(\bp)$. 
It follows from 3) of Proposition \ref{prop-mix} that $\br_{\pm d_1} = \br_{\pm c} \br_{\pm d_2}$. 
Then $Q^{d_1 / 2}  W_\bp (Q) = \br_{-d_1} W_\bp (Q) = \br_{-c} \br_{-d_2}  W_\bp (Q) = \br_{-c} Q^{d_2 / 2}  W_\bp (Q)$, which implies the equality for $c$. 
\end{proof}

Choose a descendent insertion $\tau (s) = \tau (s_1, \cdots, s_k)$, which we allow to lie in $K_{\sK \times \sT \times \bC^*_\hbar \times \bC^*_q} (\pt)_X = \bigcap_{\bp \in X^\sT}  K_{\sK \times \sT \times \bC^*_\hbar \times \bC^*_q} (\pt)_\bp$.

\begin{Proposition} \label{thm-V-Verma}
The descendent vertex function restricted to $\bp$ is the descendent Whittaker function associated with $W_\bp (Q)$:
$$
V^{(\tau (s) )} (Q) \Big|_\bp = \la W_\bp (Q), \tau (s) W_\bp (Q) \ra. 
$$
\end{Proposition}

\begin{proof}
The $Q^d$-coefficient of $\la W_\bp (Q), \tau (s) W_\bp (Q) \ra$ is
$$ 
\frac{\la v, \br_{-d} \tau (s) \br_d v \ra}{(\br_{-d} \br_d) |_\bp^2} = 
\frac{\la v, \br_{-d} \br_d \tau (q^{\la e_1^*, d\ra} s_1, \cdots, q^{\la e_k^*, d \ra} s_k) v \ra}{(\br_{-d} \br_d) |_\bp^2} = 
\frac{\tau (q^{\la e_1^*, d\ra} s_1, \cdots, q^{\la e_k^*, d \ra} s_k) |_\bp 
}{
(\br_{-d} \br_d) |_\bp
} , 
$$
which is equal to 
$$
\tau (q^{\la e_1^*, d\ra} S_1 |_\bp, \cdots, q^{\la e_k^*, d \ra} S_k |_\bp)
\cdot 
\prod_{i=1}^n (- q^{1/2} \hbar^{-1/2} )^{D_i} \frac{(\hbar L_i |_\bp )_{ D_i } }{(q L_i |_\bp)_{D_i} }  . 
$$
The statement then follows from Proposition \ref{formula-V}.
\end{proof}

\subsection{Quantum $q$-difference equations and Bethe algebra}

We can now obtain $q$-difference equations satisfied by the vertex function. 
Note that by definition, $q^{\chi_i Q \partial_Q}$ shifts $x_i$ by $q$. 
Recall the renormalized vertex function introduced in Definition \ref{Defn-vertex}
$$
\widetilde V^{(\tau (s) )} (Q) := e^{\frac{\la \ln S, \ln Q \cdot (-\hbar^{1/2})^{-\det T_X^{1/2} } \ra}{\ln q}} \cdot \frac{(q T_X^{1/2})_\infty}{(\hbar T_X^{1/2})_\infty} \cdot V^{(\tau (s) )} (Q).
$$
Here we have $\la \ln S, \ln Q\ra = \sum_{j=1}^k \ln S_j \ln Q_j$. 
Under the action of $q^{\chi_i Q \partial_Q}$, the renormalized version is further shifted by $S^{\chi_i}$. 

\begin{Corollary} \label{cor-q-diff}

\begin{enumerate}[1)]
\setlength{\parskip}{1ex}

\item Let $1\leq j\leq k$, and $c \in \Eff(X)$. 
Then
$$
q^{Q_j \partial_{Q_j}} \widetilde V^{(\tau (s))} (Q) =  \widetilde V^{(s_j \tau (s))} (Q) , \qquad Q^{c} \widetilde V^{(\tau (s))} (Q) = \widetilde V^{ (\br_c \tau (s) \br_{-c} ) } (Q). 
$$

\item For any $c \in \Eff(X)$, the vertex function $\widetilde V (Q)$ is annihilated by the $q$-difference operators
\ben
&& \prod_{i: C_i >0} (a_i  q^{\chi_i Q \partial_Q}; q^{-1})_{C_i} \prod_{i: C_i<0} (\hbar a_i q^{\chi_i Q \partial_Q} )_{-C_i} \\
&& - Q^c \cdot (-q^{1/2} \hbar^{-1/2})^{\sum_{i=1}^n C_i}  \prod_{i: C_i >0} (\hbar a_i q^{\chi_i Q \partial_Q} )_{C_i} \prod_{i: C_i<0} (a_i q^{\chi_i Q \partial_Q}; q^{-1})_{-C_i} , 
\een
where $C_i = \langle \chi_i, c\ra$. 
\end{enumerate}

\end{Corollary}

\begin{proof}
For 1), the first identity follows from Lemma \ref{shift-r-x}, and the second from Proposition \ref{thm-V-Verma} and the identity
\ben
\la Q^{c/2} W_\bp (Q), \tau (s) Q^{c/2} W_\bp (Q) \ra &=& \la \br_{-c} W_\bp (Q), \tau (s) \br_{-c} W_\bp (Q) \ra \\
&=& \la W_\bp (Q), \br_c \tau (s) \br_{-c} W_\bp (Q) \ra . 
\een
The result holds for restrction to any $\bp$, and hence also holds for the global vertex function. 
2) follows from 1) by taking $\tau (s) = 1$. 
\end{proof}

\begin{Remark}
When $c\in \Eff(X)$ is a circuit, and hence $C_i \in \{0, \pm 1\}$, we recover the $q$-difference equations 1) of \cite[Theorem 5.14]{SZ}, up to some shift of K\"ahler parameters and a prefactor in the definition of vertex functions, which is responsible for the absorption of the equivariant parameters $a_i$ into the $q$-difference operators. 
\end{Remark}

We now give a description of the $q$-difference module generated by the vertex function in terms of the virtual Coulomb branch, as well as the Bethe algebra.
 
Consider the ring of $q$-difference operators:
$$
\cD_q  := \bC [[Q^{\Eff(X)}]] [ q^{\pm Q_j \partial_{Q_j}} , 1\leq j\leq k ] ,
$$
acting naturally on the space $K_{\sT \times \bC^*_\hbar \times \bC^*_q} (X)_\loc [[Q^{\Eff(X)}]]$, where all descendent vertex functions take values.

On the other hand, there is a $K_{\sT \times \bC^*_\hbar \times \bC^*_q} (X)_\loc$-linear $\cD_q$-action on the Cartan subalgebra $\cA^0_\sT (\sK, \sN)_X$, defined as
$$
q^{Q_j \partial_{Q_j}} \bullet \tau(s) := s_j \tau (s), \qquad Q^c \bullet \tau(s) := \br_c \tau(s) \br_{-c}, \qquad c\in \Eff(X). 
$$

Moreover, if we apply the $q\mapsto 1$ limit to $\cA_\sT (\sK, \sN)_X$, all products among the generators $\br_d$ and $s_j$'s become commutative. 
We denote the commutative algebra by $\cA_\sT (\sK, \sN)_{X, q=1}$. 

\begin{Theorem} \label{Thm-ab-D-mod}

\begin{enumerate}[1)]
\setlength{\parskip}{1ex}

\item The $\cD_q$-module generated by the vertex function $\widetilde V(Q)$ over $K_{\sT \times \bC^*_\hbar} (\pt)_\loc$ is isomorphic to \footnote{Here the tensor product $\otimes_\bC$ is as $\bC[ q^{\pm Q_j \partial_{Q_j}} , 1\leq j\leq k ]$-modules, not as algebras.}
$$
\bC [[Q^{\Eff(X)}]] \otimes_\bC \cA^0_\sT(\sK, \sN)_X \, / \, \la 1 \otimes \br_c \tau(s) \br_{-c}  - Q^c \otimes \tau(s) , \ c\in \Eff(X) \ra , 
$$
where $\widetilde V^{(\tau (s))} (Q)$ is sent to $1 \otimes \tau(s)$. 

\item The Bethe algebra of $X$ over $K_{\sT \times \bC^*_\hbar} (\pt)_\loc$ is isomorphic to
$$
\left. \cA^0_\sT (\sK, \sN)_{X, q=1} [[Q^{\Eff(X)}]] \,  \middle/ \, \la \br_c \br_{-c} |_{q=1} - Q^c, \ c\in \Eff(X) \ra \right. , 
$$
where $\tau(s) |_{q=1} \in \cA^0_\sT (\sK, \sN)_{X, q=1}$ is sent to the quantum tautological class $\widehat\tau (Q)$. 

\end{enumerate}

\end{Theorem}

\begin{proof}
1) follows from Corollary \ref{cor-q-diff}. 
For 2), first we see that there is a surjection from $\cA^0_\sT (\sK, \sN)_{X, q=1} [[Q^{\Eff(X)}]]$ to the Bethe algebra, sending $\tau(s)|_{q=1}$ to $\widehat\tau(Q)$. 
The map is well-defined since different $\tau$'s with the same $q=1$ limit yield the same quantum tautological class. 
It preserves the algebra structure by Lemma \ref{lemma-Bethe-hat}. 
It is also clear that any $\br_c \br_{-c} |_{q=1} - Q^c$ lies in the kernel. 
Finally, when setting $Q\to 0$, and if $c\in \Eff(X)$ is a circuit, the ideal generated by $\br_c \br_{-c}$ is exactly $\langle \prod_{i: \langle \chi_i c \rangle >0} (1-x_i) \prod_{i: \langle \chi_i, c \rangle <0} (1-\hbar x_i) \ra$, since other factors in  $\br_c \br_{-c}$ are invertible.
We obtain an isomorphism to $K_{\sT \times \bC^*_\hbar} (X)_\loc$ by the presentation (\ref{K-ring}). 
The theorem then follows from Nakayama's lemma. 
\end{proof}

\vspace{2em}

\section{Nonabelian case} \label{sec-nonab}

In this section, we return to the general case, where $\sG$ is a connected reductive group (with $\pi_1 (\sG)$ torsion-free), and $\sN$ is a finite-dimensional representation. 

\subsection{Higgs branch abelianization}

In this subsection we review the abelianization of the Higgs branch $X$. 
The approach was applied in \cite{PSZ, KPSZ} to write down the explicit formulas for vertex functions\footnote{For an application of abelianization to 3d $\cN = 2$ theories, see \cite{Webb, Wen, GY}.}.

Recall that we have a complex moment map $\mu_\sG: T^* \sN \to \fg^*$. 
Choose a stability condition $\theta \in \chr (\sK)^\sW$, and let $X = \mu_\sG^{-1}(0) /\!/_\theta \sG = \mu_\sG^{-1} (0)^{\sG-s} / \sG$ be the holomorphic symplectic quotient, where $\mu_\sG^{-1} (0)^{\sG-s}$ denotes the $\sG$-stable locus. 

Let $\sK \subset \sG$ be a fixed maximal torus, and view $\sN$ as a $\sK$-representation. 
Let $\pr: \fg^* \to \fk^*$ be the projection to the dual Lie algebra of $\sK$. 
The moment map for $\sK$ is then the composition $\mu_\sK = \pr \circ \mu_\sG: T^* \sN \to \fk^*$. 
Let $X^\ab := \mu_\sK^{-1} (0) /\!/_\theta \sK = \mu_\sK^{-1} (0)^{\sK-s} / \sK$ be the hypertoric variety.

\begin{Assumption} \label{Ass}

We make the following assumptions.

\begin{enumerate}[(i)]
\setlength{\parskip}{1ex}

\item $\theta$ is chosen, such that $\mu^{-1}(0)^{\sG-ss} = \mu^{-1}(0)^{\sG-s}$, and hence $X$ is smooth.
Same for $X^\ab$.

\item There is a torus $\sA$ acting on $\sN$, which commutes with $\sG$, such that $X$ is $(\sA \times \bC^*_\hbar)$-equivariantly formal. 

\item The $(\sA \times \bC^*_\hbar)$-equivariant Kirwan map $\kappa: K_{\sG \times \sA\times \bC^*_\hbar} (\pt)  \to K_{\sA \times \bC^*_\hbar} (X)$ is surjective.

\end{enumerate}
\end{Assumption}

Assumption (i)-(iii) hold for typical examples we are interested in, such as Nakajima quiver varieties, where the torus $\sA$ is the flavor torus acting on the framing spaces. 
See \cite{MN} for (iii) and \cite{Gon} for (ii). 

Note that the stability condition $\theta$ now lies in $\chr(\sK)_\bR^\sW$.
For the wall-and-chamber structure, let us consider the following finer structure in $\chr(\sK)_\bR$: 
besides the characters $\chi_i$'s that appear in $\sN$, all roots $\alpha \in \Phi$ of $\sG$ also define the walls. 
In other words, walls are given by hyperplanes spanned by \emph{generalized roots} (see \cite{BFN}). 
As in \cite[Theorem 5.1]{HL-S}, under the Assumption (i) above, the walls of $X$ are the intersection of such walls with $\chr(\sK)_\bR^\sW$. 

The map $\pi_1 (\sK) \to \pi_1(\sG)$ can be viewed as the dual of the embedding $\chr(\sK)^\sW \to \chr(\sK)$. 
The cowall-and-cochamber structure of $X$ is again the dual of the wall-and-chamber structure, lying in $\pi_1 (\sG)$. 
In particular, the effective cone $\Eff(X)$ is again the dual of the K\"ahler cone.
Since $\theta$, and the set of characters $\{\chi_i \}$ is $\sW$-invariant, the effective cone $\Eff(X^\ab)$ is also $\sW$-invariant.
It is then divided by root hyperplanes into subcones, where its intersection with the dominant Weyl cochamber $\Eff(X^\ab) \cap \cochar(\sG)_{\bR, +}$ is a $\sW$-fundamental domain.

The abelian and nonabelian symplectic quotients fit into the following diagram
$$
\xymatrix{
Y= \mu_\sG^{-1} (0)^{\sG-s} / \sK \ar@{^{(}->}[r]^-j \ar[d]_p & \mu_\sG^{-1} (0)^{\sK-s} / \sK \ar@{^{(}->}[r]^-{i_0} &  \mu_\sK^{-1} (0)^{\sK-s} / \sK =X^\ab \\
X= \mu_\sG^{-1} (0)^{\sG-s} / \sG,
}
$$
where $j$ is an open embedding, $p$ is a $\sG / \sK$-fibration, and $i_0$ is a closed embedding, given by the zero section of the bundle associated with the roots of $\fg$.
We denote $Y:= \mu_\sG^{-1} (0)^{\sG-s} / \sK$, $i:= i_0 \circ j$, and denote by $i^!:= j^* \circ i_0^!$ the l.c.i. pullback in $K$-theory.

The relationship between the cohomologies of $X$ and $X^\ab$ has been studied in \cite{HP, Shenf}. 
The following is a $K$-theoretic analogue.
Given a choice of Weyl chamber, we denote 
$$
e_+ := \prod_{\alpha\in \Phi_+} (s^{\alpha/2}-s^{-\alpha/2})  = s^\rho \prod_{\alpha \in \Phi_+} (1-s^{-\alpha}) ,
 \qquad
e:= e_+ \prod_{\alpha \in \Phi} (1- \hbar s^{\alpha}),
$$
which are $\sW$-anti-invariant. 
Let $\ann(e)$ be the annihilator of its images in the $K$-theory of $X^\ab$ under the Kirwan map.

\begin{Lemma} \label{Lemma-ab-nonab-K}
\begin{enumerate}[1)]
\setlength{\parskip}{1ex}

\item The map $p_* \circ i^!$ gives an isomorphism
\begin{equation} \label{ab-K}
K_{\sA \times \bC^*_\hbar} (X) \cong ( K_{\sA \times \bC^*_\hbar} (X^\ab) / \ann (e) )^\sW  . 
\end{equation}

\item Given a presentation $K_{\sA \times \bC^*_\hbar} (X^\ab) \cong K_{\sK \times \sA \times \bC^*_\hbar} (\pt) / I$, we have
$$
K_{\sA \times \bC^*_\hbar} (X)_\loc \cong K_{\sK \times \sA \times \bC^*_\hbar} (\pt)^\sW \otimes \ff (K_{\sA \times \bC^*_\hbar} (\pt )) \Big/ \Big\la e_+^{-1} \sum_{w\in \sW} \sgn(w) w f  , \ f \in I \Big\ra . 
$$

\end{enumerate}

\end{Lemma}

\begin{proof}
1) follows from \cite[Theorem 2.3]{HP} and the equivariant Riemann--Roch theorem \cite{EG}. 
Note that by the equivariant formality, $K_{\sA \times \bC^*_\hbar} (X)$ is flat over $K_{\sA \times \bC^*_\hbar} (\pt)$, and hence injects into its completion at the augmentation ideal of $K_{\sA \times \bC^*_\hbar} (\pt)$.
For 2), we observe from 1) the invertibility of $\prod_{\alpha \in \Phi} (1- \hbar s^{\alpha})$ in the localized $K$-theory. 
As an analogue of \cite[Theorem 4.4]{ES}, 2) follows from the fact that the map $f \mapsto e_+^{-1} \sum_{w\in \sW} \sgn(w) w f$ is an exact morphism of $K_{\sK} (\pt)^\sW$-modules, which projects $K_{\sK} (\pt)$ to $K_{\sK} (\pt)^\sW$. 
\end{proof}

\subsection{Abelianization of vertex function} \label{sec-nonab-vertex}

One obtains the following induced diagram
$$
\xymatrix{
\QM (\bP^1, Y)  \ar[d]_p   \ar@{^{(}->}[r]^-j \ar[d]_p & \QM(\bP^1, \mu_\sG^{-1} (0)^{\sK-s} / \sK ) \ar@{^{(}->}[r]^-{i_0} &  \QM(\bP^1, X^\ab) \\
\QM (\bP^1, X),
}
$$
where by abuse of notations, the maps are still denoted by $p$, $i_0$ and $j$, where $i_0$ is a closed embedding and $j$ is an open embedding. 
We still denote the composition by $i = i_0 \circ j$.

Moreover, there is a similar diagram for the stacky quotients $\fX := [ \mu_\sG^{-1} (0) / \sG]$, $\fY := [\mu_\sG^{-1} (0) / \sK]$, $\fX^\ab := [\mu_\sK^{-1} (0) / \sK]$, as well as the Hom-stacks $\Hom(\bP^1, \fX)$. 
Denote the universal family and evaluation map by $\pi: \bP^1 \times \Hom(\bP^1, \fX) \to \Hom(\bP^1, \fX)$ and $f: \bP^1 \times \Hom(\bP^1, \fX) \to \fX$. 

Given a root $\alpha$ of $\fg$, let $L_\alpha$ be the line bundle on $Y$ associated with $\alpha$. 
Let $\cU$ be the object in the derived category of $\QM(\bP^1, Y)$, which lies in degree $[0,1]$:
$$
\cU := R\pi_* f^* \Big( \bigoplus_{\alpha \in \Phi} L_\alpha \Big). 
$$ 
Note that since $Y$ is an abelian quotient, the connected components of $\QM(\bP^1, Y)$ are indexed by $\cochar(\sK)$. 
So on each connected component, the fiber dimensions of cohomology groups $h^i (\cU^\vee)$ for $i = -1, 0$ stay constant. 
Therefore $h^i(\cU^\vee)$ are vector bundles. 

\begin{Lemma} \label{abelianization-obs}
Let $E^\bullet_{\QM(\bP^1, X)}$, $E^\bullet_{\QM(\bP^1, Y)}$, $E^\bullet_{\QM (\bP^1, X^\ab)}$ be perfect obstruction theories constructed in Section \ref{sec-symp-quot}. 
\begin{enumerate}[1)]
\setlength{\parskip}{1ex}

\item We have a distinguished triangle
$$
\xymatrix{
p^* E^\bullet_{\QM(\bP^1, X)} \ar[r] & E^\bullet_{\QM(\bP^1, Y)} \ar[r] & \cU^\vee \ar[r] & p^* E^\bullet_{\QM(\bP^1, X)} [1].
}
$$
As a result, $\cU^\vee$ is a relative perfect obstruction theory for the map $p: \QM(\bP^1, Y) \to \QM (\bP^1, X)$, such that
$$
\widehat\cO_{\QM(\bP^1, Y), \vir} = \widehat p^!_\vir \widehat\cO_{\QM (\bP^1, X), \vir}.
$$

\item We have a distinguished triangle \footnote{By abuse of notation, we also denote by $\cU$ the complex defined similarly on $\QM(\bP^1, \mu_\sG^{-1} (0)^{\sK-s} / \sK)$.}
$$
\xymatrix{
\hbar \otimes \cU^\vee \ar[r] & i_0^* E^\bullet_{\QM(\bP^1, X^\ab)} \ar[r] & E^\bullet_{\QM(\bP^1, \mu_\sG^{-1} (0)^{\sK-s} / \sK)} \ar[r] & \hbar \otimes \cU^\vee [1],
}
$$
which restricts via $j$ to a similar triangle on $\QM(\bP^1, Y)$.
As a result, $i_0$ is a regular embedding, and $\hbar \otimes \cU^\vee [1]$ is a relative obstruction theory for $i_0$, of amplitude $[-2, -1]$, and satisfying Assumption \ref{Ass-invert}.
We have
$$
\widehat\cO_{\QM(\bP^1, Y), \vir} 
= j^* \widehat i^!_{0,\vir} \widehat\cO_{\QM(\bP^1, X^\ab), \vir} ,
$$
where $\widehat i^!_{0, \vir}$ is the (modified) virtual l.c.i. pullback defined in Definition \ref{pullback}.
\end{enumerate}

\end{Lemma}

\begin{proof}
The two distinguished triangles follows from the definition of obstruction theories. 
1) then follows from the functoriality of (usual) virtual pullbacks, while 2) follows from the functoriality for the virtual l.c.i. pullback, i.e. Proposition \ref{functoriality} and (\ref{functoriality-hat}). 
\end{proof}

Let $V(X; Q)$ and $V (X^\ab; \tilde Q)$ be the vertex functions of $X$ and $X^\ab$ respectively, where the K\"ahler parameters $Q^{\bar d}$ for $X$ record the connected components $\bar d \in \pi_1 (\sG)$. 
For a degree class $d\in \cochar (\sK) = \pi_1 (\sK)$ of $Y$ or $X^\ab$, we write $\bar d$ for its image in $\pi_1 (\sG)$. 

In practice, the above result is often written in terms of components at fixed points. 
From now on, we assume that $X^{\sA \times \bC^*_\hbar}$ consists of \emph{isolated points}. 
Let $\bp\in X^{\sA \times \bC^*_\hbar}$ be a fixed point.
The preimage $p^{-1} (\bp) \cong \sG / \sK$ is an affine fibration over $\sG / \sB$, where $\sB\subset \sG$ is a Borel subgroup containing $\sK$.  
If we apply the $\sT$-action on $Y$, the $\sT$-fixed points of $p^{-1} (\bp)$ consists of elements of the form $w \cdot \tilde\bp$, where $\tilde\bp \in p^{-1} (\bp)$ is any point, and $w\in \sW$ is an element of the Weyl group. 
We see that under the isomorphism (\ref{ab-K}), $\cO_{\tilde\bp}$ is mapped to $\cO_\bp$. 

\begin{Proposition} \label{prop-nonab-V}
Assume that $X$ admits isolated $(\sA \times \bC^*_\hbar)$-fixed points. 
Then
\ben
V^{(\tau (s)) } (X; Q) \Big|_\bp &=&  \sum_{d\in \Eff (X^\ab)} Q^{\bar d} \cdot \prod_{\alpha\in \Phi} \left[ (-q^{1/2} \hbar^{-1/2})^{D_\alpha} \frac{(\hbar L_\alpha)_{D_\alpha}}{ (q L_\alpha)_{D_\alpha} }  \right]^{-1} \Big|_{\tilde\bp} \\
&& \cdot \prod_{i=1}^n \left[ (-q^{1/2} \hbar^{-1/2})^{D_i} \frac{(\hbar L_i)_{D_i}}{ (q L_i)_{D_i} }  \right] \Big|_{\tilde\bp} \cdot \tau (q^{\la e_1^* , d\ra} S_1, \cdots, q^{\la e_k^*, d \ra} S_k ) \Big|_{\tilde\bp}. 
\een
where $D_i := \langle \chi_i, d \rangle$, and $D_\alpha := \la \alpha, d\ra$, and the RHS is understood as valued in the $\sA$-equivariant theory with specialization of equivariant parameters $K_\sT(\pt) \to K_\sA(\pt)$. 
\end{Proposition}

\begin{proof}
Since $\tilde\bp$ is a point in the preimage of $\bp$, and in particular, $p_* [\cO_{\tilde\bp}] = [\cO_\bp]$, it suffices to compare the localized virtual structure sheaves of the fixed loci $\QM(\bP^1, \bp) \subset \QM (\bP^1, X)^\sA$ and $\QM(\bP^1, \tilde\bp) \subset \QM(\bP^1, Y)^\sT$. 
It suffices to compare the obstruction theories via Lemma \ref{abelianization-obs}. 
We have 
\ben
&& \widehat\cO_{\QM(\bP^1, \bp)} \cdot \frac{ \bigwedge^\bullet h^{-1} (\cU^\vee) \otimes \det h^{-1} (\cU^\vee)^{-1/2} }{ \bigwedge^\bullet h^{0} (\cU^\vee) \otimes \det h^{0} (\cU^\vee)^{-1/2} } \\
&& = \frac{ \bigwedge^\bullet h^0 (\hbar \otimes \cU^\vee) \otimes \det h^{0} (\hbar \otimes \cU^\vee)^{-1/2} }{\bigwedge^\bullet h^{-1} (\hbar \otimes \cU^\vee) \otimes \det h^{-1} (\hbar \otimes \cU^\vee)^{-1/2} } \cdot \widehat\cO_{\QM(\bP^1, \tilde\bp)} , 
\een
where the $h^0$ contribution comes from the virtual normal bundle, i.e. the difference of deformation theories. 
\end{proof}

\begin{Remark}
From now on we will allow descendent insertion $\tau(s) \in K_{\sK \times \sT \times \bC^*_\hbar \times \bC^*_q} (\pt)_\loc$ (or a completion of it), as long as $\tau (q^{\la e_1^* , d\ra} S_1, \cdots, q^{\la e_k^*, d \ra} S_k ) \big|_{X}$ is well-defined for all $d\in \Eff(X)$.
\end{Remark}


\subsection{Abelianized virtual Coulomb branch}

Let $\iota$ be the closed embedding of the moduli of triples $\cR(\sK, \sN)$ into $\cR(\sG, \sN)$, induced by the embedding $\Gr_\sK \hookrightarrow \Gr_\sG$. 
The same arguments as in \cite[Section 5(i) and Remark 5.23]{BFN} yield an (localized) abelianization isomorphism between the Coulomb branch $\cA_\sA (\sG, \sN)$ and the \emph{abelianized} $\cA_\sA (\sK, \sN)^\sW$.
Since little is known about the nonabelian Coulomb branch $\cA_\sA (\sG, \sN)$, in the following, we will work with the abelianized Coulomb branch $\cA_\sA (\sK, \sN)^\sW$.

As in Section \ref{Sec-val}, let us specify the coefficient ring.
Similarly to the abelian case, we have the decomposition into irreducible components $\Spec K_{\sA \times \bC^*_\hbar} (X) = \bigcup_{\bp\in X^{\sA \times \bC^*_\hbar}} \Spec K_{\sA \times \bC^*_\hbar} (\bp)$, and they embeds into the ambient torus $\Spec K_{\sG \times \sA \times \bC^*_\hbar} (\pt) = \Spec K_{\sK \times \sA \times \bC^*_\hbar} (\pt)^\sW$.
Denote by $I(\bp)$ the ideal of the component $\Spec K_{\sA \times \bC^*_\hbar} (\bp)$, and let $q^\bZ I(\bp)$ denotes the union of all its $q$-shifts.

\begin{Definition}
\begin{enumerate}[1)]
\setlength{\parskip}{1ex}

\item Define $K_{\sG \times \sA \times \bC^*_\hbar \times \bC^*_q} (\pt)_{\bp}$ the localization of $K_{\sG \times \sA \times \bC^*_\hbar \times \bC^*_q} (\pt)$ at the multiplicative subset 
$$
K_{\sG \times \sA \times \bC^*_\hbar \times \bC^*_q} (\pt) - q^\bZ I(\bp).
$$
There is an evaluation map $(-)|_\bp : K_{\sG \times \sA \times \bC^*_\hbar \times \bC^*_q} (\pt)_{\bp} \to K_{\sA \times \bC^*_\hbar \times \bC^*_q} (\bp)_\loc$, defined by the quotient at the maximal ideal $I(\bp)$. 

\item Define $K_{\sG \times \sA \times \bC^*_\hbar \times \bC^*_q} (\pt)_X$ as the localization of $K_{\sG \times \sA \times \bC^*_\hbar \times \bC^*_q} (\pt)$ at the multiplicative subset 
$$
K_{\sG \times \sA \times \bC^*_\hbar \times \bC^*_q} (\pt) - \bigcup_{\bp\in X^{\sA \times \bC^*_\hbar}} q^\bZ I(\bp).
$$
There is an evaluation map $(-)|_X : K_{\sG \times \sA \times \bC^*_\hbar \times \bC^*_q} (\pt)_X \to K_{\sA \times \bC^*_\hbar \times \bC^*_q} (X)_\loc$, defined by the quotient at the intersection of ideals $\bigcap_{\bp \in X^{\sA \times \bC^*}} I(\bp)$. 
\end{enumerate}
\end{Definition}
Note that the coefficient ring $K_{\sG \times \sA \times \bC^*_\hbar \times \bC^*_q} (\pt)_X$ is the base change of the abelianized coefficient ring $K_{\sK \times \sT \times \bC^*_\hbar \times \bC^*_q} (\pt)_{X^\ab}^\sW$ via the composition
$$
K_{\sK \times \sT \times \bC^*_\hbar \times \bC^*_q} (\pt)_{X^\ab}^\sW \to K_{\sK \times \sA \times \bC^*_\hbar \times \bC^*_q} (\pt)_{X^\ab}^\sW \to K_{\sG \times \sA \times \bC^*_\hbar \times \bC^*_q} (\pt)_X.
$$
In particular, elements of the form $1- q^\bZ s^{\alpha}$ and $1 - q^\bZ \hbar s^\alpha$ are invertible. 

Let $\cA_\sA (\sK, \sN)_{X^{\ab}}$ be the virtual Coulomb branch with \emph{mixed polarization} with respect to $X$ defined in Definition \ref{Defn-cA-X}. 
We define
$$
\cA_\sA (\sK, \sN)^\sW_X := K_{\sG \times \sA \times \bC^*_\hbar \times \bC^*_q} (\pt)_X \otimes_{K_{\sK \times \sT \times \bC^*_\hbar \times \bC^*_q} (\pt)_{X^\ab}^\sW} \cA_\sA (\sK, \sN)_{X^{\ab}}^\sW .
$$

\subsection{Vertex function, quantum $q$-difference module and Bethe algebra} \label{sec-nonab-D-mod}

Consider the element
$$
\Delta (s) := \prod_{\alpha\in \Phi} \frac{(\hbar s^\alpha)_\infty}{(q s^\alpha)_\infty} \quad \in \quad  K_{\sG \times \sA \times \bC^*_\hbar \times \bC^*_q} (\pt)_X,
$$
where we consider the completion over $\bC[[q]]$.

Let $X$ be as in Section \ref{sec-nonab-vertex}, and $V^{(\tau(s))} (X; Q)$ (resp. $\widetilde V^{(\tau(s))} (X; Q)$) be its vertex function (resp. renormalized vertex function). 
Let $\bp\in X^{\sA\times \bC^*_\hbar}$ be a fixed point, and $\tilde \bp \in Y^{\sT} \subset (X^\ab)^{\sT}$ be a fixed point in $Y$, which projects to $\bp$. 
Let $M(\tilde\bp)$ be the Verma module of the abelian virtual Coulomb branch $\cA_\sT (\sK, \sN)_{\tilde\bp}$ associated with $\tilde\bp$, with highest weight vector $v$. 
We can now express the vertex function of the nonabelian Higgs branch $X$ in terms of the Verma module of the abelianized virtual Coulomb branch. 

\begin{Proposition} \label{Prop-V-Verma-nonab}
Assume that $X$ admits isolated $(\sA\times \bC^*_\hbar)$-fixed points. 
The descendent vertex function of $X$ restricted to $\bp$ is
\ben
V^{(\tau(s))} (X; Q) \Big|_\bp &=& \frac{1}{\Delta(s) |_\bp} \cdot V^{(\tau (s) \Delta (s))} (X^\ab; \tilde Q) \Big|_{\tilde Q^d \mapsto Q^{\bar d}}, 
\een
where $\bar d$ is the image of $d$ in $\pi_1 (\sG)$.
\end{Proposition}

\begin{proof}
This follows from Proposition \ref{thm-V-Verma} and \ref{prop-nonab-V}. 
\end{proof}

For any $\chi\in \chr(\sK)$, let $q^{\chi \tilde Q\partial_{\tilde Q} }$ the $q$-difference operator defined in Section \ref{sec-q-diff-ab}.  
If $\chi\in \chr(\sK)^\sW$, let $q^{\chi Q\partial_Q}$ be its restriction to functions in $Q$.

\begin{Corollary} \label{cor-q-diff-nonab}
Assume that $X$ admits isolated $(\sA\times \bC^*_\hbar)$-fixed points. 
Let $\chi\in \chr(\sK)^\sW$, and $c \in \Eff(X^\ab) \cap \cochar(\sG)_+$. 
Then 
$$
q^{\chi Q \partial_Q } \widetilde V^{(\tau (s))} (X; Q) =  \widetilde V^{(s^\chi \tau (s))} (X; Q) ,
$$
$$ 
Q^{\bar c} \widetilde V^{(\tau (s))} (X; Q) = \widetilde V^{ \left( |\sW|^{-1} \sum_{w\in \sW} \br_{wc} \tau (s) \br_{-wc} \cdot \prod_{\alpha\in \Phi} \frac{(q s^\alpha)_{-\la \alpha, w c \ra}}{(\hbar s^\alpha)_{-\la \alpha, wc\ra }} \right) } (X; Q),
$$
$$
0 = \widetilde V^{ \left( e_+^{-1} \sum_{w\in \sW} \sgn(w) \br_{wc} \tau (s) \br_{-wc} \cdot \prod_{\alpha\in \Phi} \frac{(q s^\alpha)_{-\la \alpha, w c \ra}}{(\hbar s^\alpha)_{-\la \alpha, wc\ra }} \right) } (X; Q) . 
$$
\end{Corollary}

\begin{proof}
This follows from Proposition \ref{Prop-V-Verma-nonab} and Corollary \ref{cor-q-diff}.
\end{proof}

\begin{Remark}
These identities give a description to the quantum $q$-difference module generated by the descendent vertex functions. 
Note that the product $\br_{wc} \tau (s) \br_{-wc} = \br_{wc} \br_{-wc} \tau (s)|_{s^\chi \mapsto q^{\langle \chi, wc \ra}}$ can be computed explicitly using (\ref{ab-rel-eff}) and is a well-defined descendent insertion.
\end{Remark}

Consider the ring of $q$-difference operators
$$
\cD_q :=  \bC [[Q^{\Eff(X)}]] [ q^{\chi Q \partial_{Q}} , \chi\in \chr(\sK)^\sW ] ,
$$
which naturally acts on the space $K_{\sA \times \bC^*_\hbar \times \bC^*_q} (X)_\loc [[Q^{\Eff(X)}]]$. 

On the other hand, consider the Cartan subalgebra $\cA_\sA^0 (\sK, \sN)^\sW_X$ of the abelianized virtual Coulomb branch. 
It admits a $\cD_q$-action, given by
$$
q^{\chi Q \partial_{Q}} \bullet \tau(s) := s^\chi \tau (s), 
\qquad 
Q^{\bar c} \bullet \tau(s) := \frac{1}{ |\sW| } \sum_{w\in \sW} \br_{wc} \tau (s) \br_{-wc} \cdot \prod_{\alpha\in \Phi} \frac{(q s^\alpha)_{-\la \alpha, w c \ra}}{(\hbar s^\alpha)_{-\la \alpha, wc\ra }}, 
$$
where $\chi\in \chr(\sK)^\sW$, $c\in \Eff(X^\ab) \cap \cochar(\sG)_+$. 
Recall that $e_+ = s^\rho \prod_{\alpha\in \Phi_+} (1-s^{-\alpha})$. 

\begin{Theorem} \label{Thm-D-mod-nonab}

Assume that $X$ admits isolated $(\sA \times \bC^*_\hbar)$-fixed points. 
\begin{enumerate}[1)]
\setlength{\parskip}{1ex}

\item The $\cD_q$-module generated by all descendent vertex functions $\widetilde V^{(\tau(s))} (X ;Q)$ over $K_{\sA \times \bC^*_\hbar} (\pt)_\loc$ is isomorphic to \footnote{The tensor product $\otimes_\bC$ is as $\bC [ q^{\chi Q \partial_{Q}} , \chi\in \chr(\sK)^\sW ]$-modules, not as algebras.}
$$
\dfrac{
\bC [[Q^{\Eff(X)}]] \otimes_\bC \cA_\sA^0 (\sK, \sN)^\sW_X
}{
\left\la 1\otimes \br_{wc} \tau (s) \br_{-wc} \cdot  \prod_{\alpha\in \Phi} \frac{(q s^\alpha)_{-\la \alpha, w c \ra}}{(\hbar s^\alpha)_{-\la \alpha, wc\ra }} - Q^{\bar c} \otimes \tau(s)   \right\ra \cap \left( \bC [[Q^{\Eff(X)}]] \otimes_\bC \cA_\sA^0 (\sK, \sN)^\sW_X \right)
}, 
$$
where $c$ ranges over $\Eff(X^\ab) \cap \cochar(\sG)_+$ and $w$ ranges over $\sW$. 
$\widetilde V^{(\tau (s))} (X; Q)$ is sent to $1 \otimes \tau(s)$. 

\item The Bethe algebra of $X$ over $K_{\sA \times \bC^*_\hbar} (\pt)_\loc$ is isomorphic to
$$
\dfrac{ 
\cA_\sA^0 (\sK, \sN)^\sW_{X, q=1} [[Q^{\Eff(X)}]] }{ \left\la  \br_{wc} \br_{-wc} \big|_{q=1} \cdot \prod_{\alpha\in \Phi} \frac{(1-s^\alpha)^{-\la \alpha, w c \ra}
}{
(1- \hbar s^\alpha)^{-\la \alpha, wc\ra }}  - Q^{\bar c} \right\ra \cap \cA_\sA^0 (\sK, \sN)^\sW_{X, q=1} [[Q^{\Eff(X)}]] 
} , 
$$
where $c$ ranges over $\Eff(X^\ab) \cap \cochar(\sG)_+$ and $w$ ranges over $\sW$. $\tau(s) |_{q=1} \in \cA^0_\sA (\sK, \sN)_{X, q=1}$ is sent to the quantum tautological class $\widehat\tau (Q)$. 

\end{enumerate}

\end{Theorem}

\begin{proof}
The proof is the same as in Theorem \ref{Thm-ab-D-mod}. 
The only nontrivial part to check is that the limit of the Bethe algebra as $Q\to 0$ coincides with the ordinary $K$-theory $K_{\sA \times \bC^*_\hbar} (X)_\loc$. 
Recall from Lemma \ref{Lemma-ab-nonab-K} that the relations of $K_{\sA \times \bC^*_\hbar} (X)_\loc$ are given by $e_+^{-1} \sum_{w\in \sW} \sgn(w) w f$, where $f$ come from the relations of the abelian $K$-theory, i.e. the ideal given in (\ref{K-ring}). 
These relations are recovered by the products $e_+^{-1} \sum_{w\in \sW} \sgn(w) \br_{wc} \br_{-wc} |_{q=1}$. 
The $Q\to 0$ limits of the Bethe relations then coincide with the ones in the ordinary $K$-theory, since $e_+$ is invertible in the localized ring. 
\end{proof}

Note that the $q$-difference operators act as $q^{\chi Q \partial_{Q}} \bullet \tau(s) = s^\chi \tau (s)$. 
The $s^\chi$'s then can be multiplied from the left by the rule
$$
(1 \otimes s^\chi ) ( Q^{\bar c} \otimes \tau (s)) = q^{\chi Q \partial_Q} ( Q^{\bar c} \otimes \tau (s))  = Q^{\bar c} \otimes q^{\la \chi, c\ra} s^\chi \tau (s). 
$$
We then have the following description of the $q$-difference module over the non-localized coefficient ring. 
\begin{Corollary} \label{Cor-refinement}
Given $w\in \sW$ and $c\in \Eff(X^\ab) \cap \cochar(\sG)_+$, let $f_{wc}(s), g_{wc}(s) \in \bC[s^{\pm 1}, a^{\pm 1}, \hbar^{\pm 1}, q^{\pm 1}]$ be such that $\frac{f_{wc} (s)}{g_{wc} (s)} = \br_{wc} \tau (s) \br_{-wc} \cdot  \prod_{\alpha\in \Phi} \frac{(q s^\alpha)_{-\la \alpha, w c \ra}}{(\hbar s^\alpha)_{-\la \alpha, wc\ra }}$.
Then the $\cD_q$-module generated by all descendent vertex functions $\widetilde V^{(\tau(s))} (X ;Q)$ over $K_{\sA\times \bC^*_\hbar} (\pt)$ is
$$
\dfrac{
\bC [[Q^{\Eff(X)}]] \otimes_\bC \bC[s^{\pm 1}, a^{\pm 1}, \hbar^{\pm 1}, q^{\pm 1}]^\sW
}{
\left\la 1\otimes f_{wc} (s) - Q^{\bar c} \otimes g_{wc}(q^{\la \chi, w c \ra} s^\chi) \tau(s)   \right\ra \cap \left( \bC [[Q^{\Eff(X)}]] \otimes_\bC \bC[s^{\pm 1}, a^{\pm 1}, \hbar^{\pm 1}, q^{\pm 1}]^\sW \right)
}.
$$
\end{Corollary}

\begin{Remark}
In principle, just as in 2) of Corollary \ref{cor-q-diff}, one may try to compute $f_{wc} (s)$ and $g_{wc} (s)$ and obtain explicit scalar $q$-difference equations.
However, as in most of related prior works \cite{JMNT}, $q$-difference equations of the abelianized theory are easy to obtain, but those in terms of K\"ahler parameters $Q$ for the nonabelian theory are difficult to write down explicitly. 
\end{Remark}

\begin{Remark}
This result can be interpreted as a $K$-theoretic generalizaiton of the quantum Hikita conjecture \cite{KMP}. 
See Appendix \ref{sec-q-Hikita} for a discussion.
In particular, as $Q\to 0$, up to some invertible factors, the relations in the Bethe algebra are given by $\br_{wc} \br_{-wc}$, where $wc$ lives in the positive part of the Coulomb branch. 
Hence, the $Q\to 0$ limit of the Bethe algebra, i.e. the ordinary equivariant $K$-theory is expressed as the $B$-algebra in the sense of \cite{BLPW-sym-dual-2}.  
Our result then identifies the Bethe algebra, i.e. the quantum $K$-theory ring, with the $K$-theoretic virtual Coulomb branch, i.e. a deformation of the coordinate ring of the $K$-theoretic mirror. 
\end{Remark}

\begin{Example}
Let us consider a concrete example and see how one can obtain the Bethe algebra from the above results. 
In the case $X = T^*\Gr (k,n)$, with $k\leq n$, we have $\sG = GL(k)$, $\sN = \Hom(\bC^n, \bC^k)$, $\sK = (\bC^*)^k$. 
The abelianization is $X^\ab = (T^* \bP^{n-1})^k$. 
In this example the mixed polarization reduces to the canonical polarization $\Pol = \sN$. 

If we choose the cocharacter $c = (1, 0, \cdots, 0)$, and take $w\in S_k$ such that $wc = (0, \cdots, 1, \cdots, 0)$, where the $1$ lies in the $j$-th position, then we have
$$
\br_{wc} \br_{-wc} = \prod_{i=1}^n (-q^{1/2} \hbar^{-1/2})^{-1} \frac{(\hbar a_i^{-1} s_j)_{-1}}{(q a_i^{-1} s_j)_{-1}} = (-q^{1/2} \hbar^{-1/2})^{-n} \prod_{i=1}^n \frac{1 - a_i^{-1} s_j}{1 - q^{-1} \hbar a_i^{-1} s_j}, 
$$
$$
\prod_{\alpha\in \Phi} \frac{(q s^\alpha)_{-\la \alpha, w c \ra}}{(\hbar s^\alpha)_{-\la \alpha, wc\ra }} = \prod_{i \neq j} \frac{(q s_i / s_j)_{-1} }{(\hbar s_i / s_j)_{-1} } \frac{(q s_j / s_i)_1}{(\hbar s_j / s_i)_1} = \prod_{i\neq j} \frac{1 - q^{-1} \hbar s_i / s_j}{1 - s_i / s_j} \frac{1 - q s_j / s_i}{1 - \hbar s_j / s_i}. 
$$
The Bethe equation given in Theorem \ref{Thm-D-mod-nonab} is 
$$
(-\hbar^{1/2})^n \prod_{i=1}^n \frac{1 - \hbar a_i^{-1} s_j}{1 - a_i^{-1} s_j} \prod_{i\neq j} \frac{1 - \hbar s_i / s_j}{\hbar - s_i/s_j} = Q, \qquad \text{for any } 1\leq j\leq k, 
$$
which recovers the Bethe equations in \cite{PSZ, UY}. 
\end{Example}

\vspace{2em}

\section{Application: wall-crossing} \label{sec-wall-cx}

As an application of the virtual Coulomb branch, let us consider the wall-crossing phenomenon arising from variation of the stability condition $\theta$. 

\subsection{Wall-crossing for abelian virtual Coulomb branch}

Suppose we are in the case of Section \ref{sec-hypertoric}, where $\theta \in \chr (\sK)_\bR$ is a stability condition. 
Let $\cK (X)$ be the K\"ahler cone of $X$, defined as the chamber containing $\theta$. 
Let $\cK' \subset \chr(\sK)_\bR$ be another chamber, and let $\theta' \in \cK'$ be another stability condition. 
Let $X' := \mu^{-1} (0) /\!/_{\theta'} \sK$ be the hypertoric variety with stability condition $\theta'$. 

Recall that the effective cone $\Eff(X)$ is generated by all \emph{circuits}, i.e. primitive vectors in $\cochar(\sK)$ normal to the walls in $\chr(\sK)_\bR$, whose directions are determined by $\theta$. 
We call a wall $H$ \emph{separating}, if $\theta$ and $\theta'$ lie on the different sides of $H$. 
Correspondingly, a circuit $\rho$ for $X$ perpendicular to a separating wall $H$ is no longer a circuit for $X'$; but $-\rho$ is the circuit for $X'$ associated to $H$. 
We call such $\rho$ a \emph{reversing} circuit\footnote{Geometrically, a reversing circuit gives a contracted curve under the birational projection $X_\theta \to X_{\theta_0}$, where $\theta_0$ lies on the corresponding separating wall. See \cite{Kon}.}. 

\begin{Example}
Let us show how the wall-and-chamber structure changes for $X=\cA_2$ in Example \ref{A_2}. 
Let $\theta'$ be the stability condition such that $\theta'_2 > \theta'_1 >0$. 
The K\"ahler cone and effective cone for $X'$ are shown in Figure \ref{Fig-2}. 
We see that the hyperplane spanned by $\chi_3$ is a separating wall, and $\rho_3$ is a reversing circuit. 
\begin{figure}
\centering
\begin{tikzpicture}[scale=0.7]

\fill[gray!30] (0,0) -- (3,3) -- (0,3) -- cycle ;

\draw (-3,0) -- (3,0);
\draw (0,-3) -- (0,3);
\draw (-3,-3) -- (3,3); 

\draw[->] (0,0) -- (2,0) node[anchor=north west] {$\chi_1$}; 
\draw[->] (0,0) -- (0,2) node[anchor=south east] {$\chi_2$}; 
\draw[->] (0,0) -- (-2,-2) node[anchor=north west] {$\chi_3$}; 

\filldraw (1,2) circle (1pt) node[anchor = west] {$\theta'$};

\fill[gray!30] (9,0) -- (12,0) -- (12,3) -- (6,3) -- cycle ;

\draw (6,0) -- (12,0);
\draw (9,-3) -- (9,3);
\draw (6,3) -- (12,-3); 

\draw[->] (9,0) -- (11,0) node[anchor=north west] {$\rho_2$}; 
\draw[->] (9,0) -- (9,2) node[anchor=south east] {$\rho_1$}; 
\draw[->] (9,0) -- (11,-2) node[anchor=north east] {$\rho_3$}; 

\node at (10.5, 1.5) {$\Eff(X')$};

\end{tikzpicture}
\caption{Wall-and-chamber structures in $\chr(\sK)_\bR$ and $\cochar(\sK)_\bR$ for $X'$}  \label{Fig-2}
\end{figure}
\end{Example}

Recall that $\cA_\sT (\sK, \sN)_X$, the virtual Coulomb branch with mixed polarization, is defined as the subalgebra in $\cA_\sT (\sK, \sN)_\loc$, generated by $\br_{\pm d}$ for $d\in \Eff(X)$, where $\br_d = r_d (\Pol)$ are generators whose polarizations are determined by the cochambers. 
While the localized virtual Coulomb branch $\cA_\sT (\sK, \sN)_\loc$ is independent of $X$, the subalgebra $\cA_\sT (\sK, \sN)_X$ depends on $X$. 

Denote by $\br'_d$ the generators for $\cA_\sT (\sK, \sN)_{X'}$. 

\begin{Proposition} \label{wall-cx-br}
If $\rho \in \Eff(X)$ is a reversing circuit for $X$ and $X'$, then in $\cA_\sT (\sK, \sN)_\loc$,
$$
\br'_{\pm \rho} = \br_{\mp \rho}^{-1}. 
$$
Otherwise, if $\rho \in \Eff(X)$ is a non-reversing circuit, then $\br'_{\pm \rho} = \br_{\pm \rho}$. 
\end{Proposition}

\begin{proof}
Apply the explicit presentation in Proposition \ref{prop-mix}, and note that $-\rho\in \Eff(X')$. 
$\br_{-\rho} \cdot \br'_\rho$ is then
\ben
 && \prod_{i: D_i < 0} \left[ (- q^{1/2} \hbar^{-1/2} )^{D_i} \frac{(\hbar x_i)_{D_i}}{(q x_i )_{D_i} } \right]^{-1} \cdot r_{-d}  \cdot \prod_{i: -D_i < 0} \left[ (- q^{1/2} \hbar^{-1/2} )^{-D_i} \frac{(\hbar x_i)_{-D_i}}{(q x_i )_{-D_i} } \right]^{-1} \cdot r_{d} \\
&=& \prod_{i: D_i < 0} \left[ (- q^{1/2} \hbar^{-1/2} )^{D_i} \frac{(\hbar x_i)_{D_i}}{(q x_i )_{D_i} } \right]^{-1} 
\prod_{i: D_i > 0} \left[ (- q^{1/2} \hbar^{-1/2} )^{-D_i} \frac{(\hbar q^{D_i} x_i)_{-D_i}}{(q q^{D_i} x_i )_{-D_i} } \right]^{-1} r_{-d} r_d. 
\een
where $D_i:= \la \chi_i, \rho \ra$. 
One can then see the RHS is actually 1, by (\ref{qd-shift-x}) and (\ref{ab-rel-special}). 
The case for $\br'_{-\rho}$ is similar. 
\end{proof}

\subsection{Wall-crossing for quantum $q$-difference modules}

Now suppose that we are in the case of Section \ref{sec-nonab-D-mod}. 
Let $\theta$ and $\theta'$ be stability conditions chosen in $\chr(\sK)_\bR^\sW$, and $X$, $X'$ be the corresponding nonabelian holomorphic symplectic quotients.  
We assume that $X$ and $X'$ satisfy Assumption \ref{Ass}, and admits isolated $(\sA \times \bC^*_\hbar)$-fixed points. 

By Theorem \ref{Thm-D-mod-nonab}, one can see that the relations in the presentation of the quantum $q$-difference modules (and also the Bethe algebra) only involves polynomials in $Q$-variables. 
More precisely, let $\cD\cM^\pol (X)$ denote
$$
\dfrac{
\bC [Q^{\Eff(X)}] \otimes_\bC \bC[s^{\pm 1}, a^{\pm 1}, \hbar^{\pm 1}, q^{\pm 1}]^\sW
}{
\left\la 1\otimes \br_{wc} \tau (s) \br_{-wc} \cdot  \prod_{\alpha\in \Phi} \frac{(q s^\alpha)_{-\la \alpha, w c \ra}}{(\hbar s^\alpha)_{-\la \alpha, wc\ra }} - Q^{\bar c} \otimes \tau(s)   \right\ra \cap \left( \bC [Q^{\Eff(X)}] \otimes_\bC \bC[s^{\pm 1}, a^{\pm 1}, \hbar^{\pm 1}, q^{\pm 1}]^\sW \right)
},
$$
which is a module over the subring of $q$-difference operators 
$$
\cD_q^{\mathrm{pol}} (X) :=  \bC [Q^{\Eff(X)}] [ q^{\chi Q \partial_{Q}} , \chi\in \chr(\sK)^\sW ] .
$$
Then the quantum $q$-difference module in Corollary \ref{Cor-refinement} is
$$
\bC [[Q^{\Eff(X)}]] \otimes_{\bC [Q^{\Eff(X)}]} \cD\cM^\pol (X). 
$$
The same construction works for $X'$, and also for the Bethe algebras. 

Consider the ring of $q$-difference operators
$$
\cD_q^\rat :=  \bC [Q^{\pm \Eff (X)}] [ q^{\chi Q \partial_{Q}} , \chi\in \chr(\sK)^\sW ] ,
$$
which contains both $\cD_q^{\mathrm{pol}} (X)$ and $\cD_q^{\mathrm{pol}} (X')$. 
To compare the two $\cD_q$-modules, we need to consider the common ring of $q$-difference operators $\cD_q^\rat$.

\begin{Theorem} \label{Thm-wall-cx}
Based changed over $\cD_q^\rat$, the quantum $q$-difference modules $\cD_q^{\mathrm{pol}} (X)$ and $\cD_q^{\mathrm{pol}} (X')$ generated by descendent vertex functions $\widetilde V^{(\tau(s))} (X; Q)$ and $\widetilde V^{(\tau(s))} (X'; Q)$ for $X$ and $X'$ over $K_{\sA\times \bC^*_\hbar} (\pt)$ are isomorphic. 
In particular, their $q$-difference equations are the same.
The same holds for the Bethe algebras. 
\end{Theorem}

\begin{proof}
First one observes that the kernel of the $q$-difference module has the same intersection with $\bC[s^{\pm 1}, a^{\pm 1}, \hbar^{\pm 1}, q^{\pm 1}]^\sW$ if we take all $\tau(s) \in \cA_\sA^0 (\sK, \sN)_\loc$, not necessarily the $\sW$-invariant polynomials.
Moreover, the relations are generated by those where $c$ are circuits for $X^\ab$. 
Therefore, it suffices to show that for those $c$ which are \emph{reversing} circuits for $X^\ab$ and $(X')^\ab$, we obtain the same relations\footnote{Note that since $\theta$, $\theta'$ are $\sW$-invariant, $c$ is reversing if and only if $wc$ is, for all $w\in \sW$.}.

To see that, we let $\tau'(s) := \br_{wc} \tau (s) \br_{-wc} \cdot  \prod_{\alpha\in \Phi} \frac{(q s^\alpha)_{-\la \alpha, w c \ra}}{(\hbar s^\alpha)_{-\la \alpha, wc \ra }}$, invert the above relation, and apply Proposition \ref{wall-cx-br}:
\ben
1 \otimes \tau' (s) &=& Q^{ \bar c} \otimes \br_{wc}^{-1} \cdot \tau' (s) \prod_{\alpha\in \Phi} \frac{(\hbar s^\alpha)_{-\la \alpha, w c \ra}}{(q s^\alpha)_{-\la \alpha, wc \ra }} \cdot \br_{-wc}^{-1} \\
&=& Q^{\bar c} \otimes \br'_{-wc} \cdot \tau' (s) \prod_{\alpha\in \Phi} \frac{(\hbar s^\alpha)_{-\la \alpha, w c \ra}}{(q s^\alpha)_{-\la \alpha, wc \ra }} \cdot \br'_{wc} \\
&=& Q^{\bar c} \otimes \br'_{-wc} \tau' (s)  \br'_{wc} \cdot \prod_{\alpha\in \Phi} \frac{(\hbar q^{\la \alpha, wc \ra }s^\alpha)_{-\la \alpha, w c \ra}}{(q q^{\la \alpha, wc \ra } s^\alpha)_{-\la \alpha, wc \ra }}. 
\een
Further multiplied by $Q^{- \bar c} \otimes 1$, this is the same as the relation for $-\bar c \in \Eff(X')$ by (\ref{qd-shift-x}). 
\end{proof}

\begin{Remark}
This proves the \cite[Conjecture 1]{Din-cap}, which states that the $q$-difference equations satisfied by vertex functions stay unchanged under the variation of GIT. 
The same statement for quantum cohomology holds by \cite{MO}.
\end{Remark}

\appendix

\vspace{1ex}

\section{Pochhammer symbols} \label{P-symbol}

The infinite Pochhammer symbol is defined as $(x; q)_\infty := \prod_{m=0} (1-q^m x)$. 
For any $d\in \bZ$, the finite Pochhammer symbol is defined as
$$
(x; q)_d := \frac{(x; q)_\infty}{(q^d x; q)_\infty} = \left\{ \begin{aligned}
& (1-x) \cdots (1-q^{d-1} x), && \qquad d>0 \\
& 1 , && \qquad d = 0\\
& \frac{1}{(1-q^{-1} x) \cdots (1-q^d x) }, && \qquad d<0 . 
\end{aligned} \right.  
$$
We will denote it simply by $(x)_d$ if the unit of progression is $q$. 

$(x)_\infty$ is only well-defined this way as a holomorphic function for $x\in \bC$, if we take $q$ to be a complex number with $|q|<1$. 
However, $(x)_d$ is actually a rational function in $x$ and $q$, and it is straightforward to check the useful identities
\begin{equation}
(x; q^{-1})_d = (-1)^d x^{d} q^{-d(d-1)/2} \cdot (x^{-1})_d, 
\end{equation}
\begin{equation} \label{qd-shift-x}
(q^{-d} x)_d = (q^{-1} x; q^{-1})_d = \frac{1}{(x)_{-d}}, 
\end{equation}
\begin{equation} \label{id-inv}
(-q^{1/2} \hbar^{-1/2} )^{d} \cdot \frac{(\hbar x )_d }{(q x)_d } = (-q^{1/2} \hbar^{-1/2} )^{-d} \cdot \frac{(x^{-1} )_{-d} }{(q \hbar^{-1} x^{-1} )_{-d} }  , 
\end{equation}
for any $d\in \bZ$. 

The Jacobi theta function is defined as
$$
\vartheta (x) \ := \ (x^{1/2} - x^{-1/2} ) \cdot (qx)_\infty \cdot  (qx^{-1})_\infty  \ = \ x^{1/2} \cdot (qx)_\infty \cdot (x^{-1})_\infty , 
$$
which satisfies
$$
\vartheta (x^{-1}) = -\vartheta (x), \qquad \vartheta (qx) = - q^{-1/2} x^{-1} \vartheta (x).  
$$

\vspace{1ex}

\section{Relationship with quantum Hikita conjecture} \label{sec-q-Hikita}

In this appendix, we discuss the relationship between our quantum $q$-difference module and the $D$-module $M$ obtained in \cite{KMP} for the quantum Hikita conjecture. 
We restrict to the abelian case.
Let $\sG = \sK$ and $\sN$ be as in Section \ref{sec-ab}. 

First, let us give a description of the virtual Coulomb branch in terms of BFN's Coulomb branch.
By the computation in Proposition \ref{thm-r_d}, the $K$-theoretic quantized Coulomb branch $\cA_\sT^\BFN (\sK, \sN)$ is generated by $r_d^\BFN$ for $d\in \cochar(\sK)$, with relations
\begin{equation} \label{BFN-ab-rel}
r_c^\BFN \cdot r_d^\BFN = \prod_{C_i > 0 >D_i} (x_i^{-1} q^{C_i-1}; q^{-1} )_{\delta_i} \prod_{C_i < 0 < D_i} (x_i^{-1} q^{C_i} )_{\delta_i} \cdot r_{c+d}^\BFN,
\end{equation}
where $C_i = \la \chi_i, c\ra$, $D_i = \la \chi_i, d\ra$, $\delta_i = \min \{ |C_i|, |D_i| \}$. 

Let $\xi_i = c_1^{\sT \times \bC^*_\hbar} (x_i)$ be the weights of the equivariant parameters $x_i = a_i s^{\chi_i}$. 
The \emph{cohomological limit} from $K$-theory to cohomology is
\begin{equation} \label{coh-lim}
1 - x_i^{-1} q^{-N}  \quad \to \quad \xi_i+ N \eta. 
\end{equation}
Under this limit, the relation (\ref{BFN-ab-rel}) tends to \cite[(4.7)]{BFN} \footnote{We omit an extra $\frac{1}{2} \hbar$ in \cite{BFN}, or absorb it into $\gamma_i$. The $\eta$ here is the $\hbar$ in \cite{BFN}. }, which is also the same as \cite[(10)]{KMP}, known as the hypertoric enveloping algebra of the mirror. 
We denote the cohomological limit of the generators by $r_c^{\BFN, H}$.

Let $X$ be a hypertoric variety as in Section \ref{sec-hypertoric}, and $\cA_\sT (\sK, \sN)_X$ be the virtual Coulomb branch with mixed polarization. 
Recall from Proposition \ref{prop-mix} that for $c\in \Eff(X)$, 
$$
\br_c \cdot \br_{-c} = \prod_{i=1}^n \left[ (- q^{1/2} \hbar^{-1/2} )^{-C_i} \frac{(\hbar x_i)_{ -C_i } }{(q x_i)_{-C_i } } \right] .
$$
Consider the quantum $q$-difference module in Theorem \ref{Thm-ab-D-mod},
$$
\bC [[Q^{\Eff(X)}]] \otimes_\bC \cA^0_\sT(\sK, \sN)_X \, / \, \la 1 \otimes \br_c \tau(s) \br_{-c}  - Q^c \otimes \tau(s) , \ c\in \Eff(X) \ra .
$$
Note that in the abelian case the $q$-difference module is generated by the identity descendent $\tau(s) = 1$. 
We can furthermore consider its refinement (see Corollary (\ref{Cor-refinement}))
\begin{equation} \label{refinement}
\frac{
\bC [[Q^{\Eff(X)}]] \otimes_\bC \bC[s^{\pm 1}, a^{\pm 1}, \hbar^{\pm 1}, q^{\pm 1}] 
}{ 
\la 1 \otimes f_c (s)  - Q^c \otimes g_c(q^{\la \chi, c\ra} s^\chi) , \ c\in \Eff(X) \ra 
}, 
\end{equation}
where $f_c(s), g_c(s)$ are two Laurent polynomials such that $\br_c \br_{-c} = \frac{f_c(s)}{g_c(s)}$. 
They can be computed explicitly, in which way we can obtain the $q$-difference equations in Corollary \ref{cor-q-diff} 2):
$$
f_c (s) = (-q^{1/2}\hbar^{-1/2})^{-\la \det \sN, c\ra} \prod_{C_i>0} (x_i; q^{-1})_{C_i} \prod_{C_i<0} (\hbar x_i)_{-C_i}, 
$$ 
$$
g_c(s) = \prod_{C_i>0} (q^{-1} \hbar x_i; q^{-1})_{C_i} \prod_{C_i<0} (q x_i)_{-C_i}
$$
Here we have used (\ref{qd-shift-x})

We will consider the \emph{mirror Calabi--Yau limit}\footnote{In enumerative 3d mirror symmetry \cite{SZ}, the $\hbar$-equivariant parameter for the mirror is $\hbar^! = q \hbar^{-1}$. So the limit $\hbar = q$ corresponds to the CY limit $\hbar^! = 1$ for the mirror.} $\hbar = q$. 
We denote the limit $f_c(s) |_{\hbar = q}$ by $\bar f_c(s)$, and denote its cohomological limit by $\bar f_c^H (s)$.

\begin{Theorem}
Under the mirror CY specialization $\hbar = q$ and the cohomological limit (\ref{coh-lim}), the quantum $q$-difference module (\ref{refinement}) specializes to the quantum $D$-module $M$ in \cite[Section 6.4]{KMP} for the theory mirror to $(\sK, \sN)$.
\end{Theorem}

\begin{proof}
The effective cone $\Eff(X)$ in our case corresponds to the cone $\mathbb{N} \Sigma_+$ in \cite{KMP}.
Note that we set $\hbar = q$, $\bar f_c (s) = \bar g_c(s)$.
If we furthermore apply the cohomological limit (\ref{coh-lim}), then $\bar f_c^H (s) = (-1)^{\la \det\sN, c\ra} r_c^{\BFN, H} r_{-c}^{\BFN, H}$.
The relation in (\ref{refinement}) then becomes
$$
1 \otimes (-1)^{\la \det\sN, c\ra} r_c^{\BFN, H} r_{-c}^{\BFN, H}  - Q^c \otimes (r_c^{\BFN, H} r_{-c}^{\BFN, H} ) |_{\xi_i \mapsto \xi_i + \la \chi_i, c\ra \eta}, 
$$
which coincides with the relation $r(c)$ in \cite[Section 6.4]{KMP}.
The sign $(-1)^{\la \det\sN, c\ra}$ is absorbed into $Q^c$ due to the shift by the theta characteristic \cite{MO}.
\end{proof}

\bibliographystyle{abbrv}
\bibliography{reference}

\vspace{1ex}

\noindent
Zijun Zhou\\
Kavli Institute for the Physics and Mathematics of the Universe (WPI), \\
The University of Tokyo, \\
5-1-5 Kashiwanoha, Kashiwa, Chiba 277-8583, Japan.\\
zijun.zhou@ipmu.jp

\end{document}